\documentclass[12pt,reqno]{amsart}
\usepackage[foot]{amsaddr}
\usepackage[utf8]{inputenc}
\usepackage[T1]{fontenc}
\usepackage{amsfonts}
\usepackage[all]{xy}
\usepackage[title,titletoc]{appendix}
\usepackage{graphicx}
\usepackage{amssymb}
\usepackage{amsthm}
\usepackage{stmaryrd}
\usepackage{mathtools}
\usepackage{comment}
\usepackage{esint}
\usepackage{hyperref}
\usepackage{xcolor}
\hypersetup{
   colorlinks,
   linkcolor={red!80!black},
   citecolor={blue!70!black},
   urlcolor={blue!90!black}
}
\usepackage{microtype}

\usepackage{upgreek}

\newcommand{\fcl}[1]{\llbracket{#1}\rrbracket}

\newcommand{\cB}{\mathcal{B}}
\newcommand{\cC}{\mathcal{C}}

\newcommand{\cH}{\mathcal{H}}
\newcommand{\cL}{\mathcal{L}}
\newcommand{\cN}{\mathcal{N}}
\newcommand{\cP}{\mathcal{P}}
\newcommand{\cQ}{\mathcal{Q}}

\newcommand{\cS}{\mathcal{S}}
\newcommand{\cT}{\mathcal{T}}
\newcommand{\cU}{\mathcal{U}}
\newcommand{\cD}{\mathcal{D}}

\newcommand{\cW}{\mathcal{W}}

\newcommand{\JfH}{J_\mathrm{H}f}
\newcommand{\DfH}{D_\mathrm{H}f}

\newcommand{\dcc}{d_{cc}}
\DeclareMathOperator{\rad}{rad}
\DeclareMathOperator{\height}{height}

\newcommand{\Aff}{\mathsf{Aff}}
\newcommand{\Lin}{\mathsf{Lin}}
\newcommand{\Con}{\mathsf{Con}}
\DeclareMathOperator{\mesh}{mesh}
\DeclareMathOperator{\Op}{Op}
\DeclareMathOperator{\gen}{gen}

\DeclareMathOperator{\Lip}{Lip}

\usepackage{tikz}
\usepackage{fourier}
\usepackage[maxbibnames=99,backend=biber, sorting=nyt, doi=false, url=false]{biblatex}
\usepackage{pgfornament}

\newcommand{\HH}{\mathbb{H}}

\newcommand{\R}{\mathbb{R}}
\newcommand{\Z}{\mathbb{Z}}
\newcommand{\ud}[0]{\,\mathrm{d}}

\newcommand{\zero}{\mathbf{0}}

\newcommand{\from}{\colon}

\newtheorem{thm}{Theorem}[section]

\newtheorem{lemma}[thm]{Lemma}
\newtheorem{prop}[thm]{Proposition}
\newtheorem{cor}[thm]{Corollary}
\newtheorem{defn}[thm]{Definition}
\theoremstyle{remark}

\DeclareMathOperator{\VCone}{VCone}

\DeclareMathOperator{\Kor}{Kor}
\DeclareMathOperator{\diam}{diam}

\DeclareMathOperator{\Wind}{Wind}
\DeclareMathOperator{\id}{id}
\DeclareMathOperator{\supp}{supp}
\DeclareMathOperator{\inter}{inter}
\DeclareMathOperator{\closure}{clos}

\date{\today}
\title{
Area of H\"older curves and coarea formula on the Heisenberg group}
\author{Gioacchino Antonelli}
\address[Gioacchino Antonelli]{Department of Mathematics, University of Notre Dame, Hurley Hall, 255 Hurley, Notre Dame, IN 46556, United States}
\email{gantonel@nd.edu}
\author{Robert Young}
\address[Robert Young]
{New York University, Courant Institute of Mathematical Sciences, 251 Mercer Street, 10012, New York, USA}
\email{ryoung@cims.nyu.edu}

\begin{document}

\begin{abstract}
    We prove the coarea formula for Lipschitz maps from the subriemannian $n$th Heisenberg group $\HH_n$ to $\mathbb R^{2n}$. Our result is new even when $n=1$ and provides the simplest vector-valued instance of the coarea formula in subriemannian geometry. This answers a question left open in the works of Magnani, Kozhevnikov, Magnani--Stepanov--Trevisan, and Julia--Nicolussi Golo--Vittone.

    The main difficulty of the proof is that a fiber of a $C^1_{\mathrm{H}}$ map $f\from \HH_n\to \R^{2n}$ is typically an unrectifiable curve. Its measure depends on the symplectic area of its projection to $\R^{2n}$. A bound on this area would imply the coarea formula, but examples of Kozhevnikov show that this area can be infinite or undefined. 

    To overcome this, we introduce an integral that we use to define both the symplectic area of
$\frac{1}{2}$--H\"older curves in $\R^{2n}$ and the symplectic area of projections of vertical
curves in $\mathbb H_n$. Then, we give a geometric condition for this integral to converge. This yields, in addition, new results on the existence of the signed area of $\tfrac12$--H\"older planar
curves that may be of independent interest. Finally, we use $\beta$--number estimates from the F\"assler--Orponen Dorronsoro Theorem to show that this geometric condition holds for almost every fiber.
\end{abstract}

\maketitle
\tableofcontents

\section{Introduction}

\subsection{The coarea formula}\label{Intro}
The coarea formula is a basic tool in geometric measure theory that relates the total measure of the fibers of a Lipschitz map $f\from \R^m\to\R^n$ to the Jacobian $Jf$ of that map \cite[Section 3.2]{FedererBook}. Many generalizations of the coarea formula to the setting of subriemannian manifolds are known, especially for Lipschitz maps from the $(2n+1)$--dimensional Heisenberg group $\HH_n$ to $\R^k$ with $k\leq n$ \cite[Theorem 1.1]{MagnaniAreaImpliesCoarea}, or for maps $\HH_n\to \R^{2n}$ that satisfy additional regularity properties (for example, for  $C^{1,\alpha}_{\mathrm{H}}$--regular maps in \cite{Kozhevnikov, MagnaniStepanovTrevisan}). The case of Lipschitz maps from $\HH_n$ to $\R^k$ with $k> n$ has remained open for many years.

In this paper, we prove the following coarea formula for Lipschitz maps from $\HH_n$ to $\R^{2n}$. We refer the reader to Section~\ref{sec:prel} for definitions and notation.
Recall that the Hausdorff measure $\mathcal{H}^{2n+2}$ related to a left-invariant homogeneous distance is a Haar measure on $\HH_n$, and $\mathcal{L}^{2n}$ denotes the Lebesgue measure on $\mathbb R^{2n}$. In this paper we normalize $\mathcal{H}^{2n+2}$ so that it coincides with the Lebesgue measure $\mathcal{L}^{2n+1}$ in exponential coordinates on $\HH_n\equiv \mathbb R^{2n+1}$. By \cite{Pansu, MagnaniPhD}, a Lipschitz map $f\from \HH_n\to\mathbb R^{2n}$ is Pansu differentiable almost everywhere, so the Pansu differential $Df_x\from \HH_n\to \R^{2n}$ is defined for $\cH^{2n+2}$--almost every $x\in \HH_n$  (see \eqref{eqn:PDiff}). Let $\DfH_x \from \R^{2n}\to \R^{2n}$ be the restriction of $Df_x$ to the horizontal plane. This is a linear map, and we define $\JfH(x) = \det \DfH_x$. Let $\cH^2$ be the $2$--dimensional Hausdorff measure on $\HH_n$, normalized so that if $I\subset \HH_n$ is a segment of the $z$--axis, then $\cH^2(I)$ is equal to its Euclidean length. 
\begin{thm}\label{thmCoareaLip}
    Let $\HH_n$ be the $n$th Heisenberg group endowed with a left-invariant homogeneous distance. Let $f\from \HH_n\to\mathbb R^{2n}$ be a Lipschitz map. For every measurable set $U\subset\HH_n$,
    \[
    \int_U |\JfH(x)|\ud\mathcal{H}^{2n+2}(x) = \int_{\mathbb R^{2n}}\mathcal{H}^2(f^{-1}(w)\cap U) \ud\mathcal{L}^{2n}(w).
    \]
\end{thm}

This result is new even when $n=1$. Even in this case, Theorem~\ref{thmCoareaLip} was considered a challenging open problem since the earliest development of geometric measure theory on subriemannian Carnot groups; see, for instance, the introduction of \cite{MagnaniCoareaInequality}, \cite[Remark 1.5]{JNGV}, and the introduction of \cite{CorniReverse}. 

The main reason that the coarea formula for Lipschitz maps $\HH_n\to\mathbb R^{2n}$ poses a challenge is the lack of a good rectifiability theory for the fibers of these maps. 
In \cite{FSSCMathAnn, FSSCILip, FSIlip}, Franchi, Serapioni, and Serra Cassano introduced a theory of rectifiability for surfaces in $\HH_n$ and in Carnot groups; see \cite{PhDAntonelli} for a survey on the subject. This theory covers horizontal surfaces (surfaces which are tangent to horizontal subgroups almost everywhere) and $H$--regular surfaces (surfaces whose tangents are transverse to a horizontal subgroup). This theory is foundational to geometric measure theory in Carnot groups (see also \cite{AK00}), and it is used in the proofs of the coarea formula for Lipschitz maps $\HH_n\to \mathbb R^{k}$ with $1\leq k\leq n$ \cite[Theorem 1.1]{MagnaniAreaImpliesCoarea}, and for real-valued Lipschitz maps with an arbitrary Carnot group as domain \cite[Theorem 3.5]{MagnaniMathNachr}. 

The fibers of a Lipschitz map $f\from \HH_n\to \R^{2n}$ are typically \emph{vertical curves} which are not rectifiable in the sense above. A vertical curve satisfies a cone condition similar to the cone condition defining Lipschitz graphs or intrinsic Lipschitz graphs (see Section \ref{sec:verticalcurves}); its only possible tangent subgroup is the center of $\HH_n$, which is neither horizontal nor transverse to a horizontal subgroup. In fact, while rectifiable surfaces have integer Hausdorff dimension, vertical curves can have fractional Hausdorff dimension \cite{AY-Vertical}. Even if $f$ is $C^1_\mathrm{H}$ (i.e., $\DfH_x$ is defined everywhere and depends continuously on $x$), its fibers can have $\mathcal{H}^2$ measure $0$ or $\infty$ \cite{Kozhevnikov}. Consequently, previous results on the coarea formula have often either assumed some additional regularity on $f$ (for instance, see \cite{KarmanovaVodopyanov}, \cite{MagnaniNonHorizontal}, \cite{Kozhevnikov}, or \cite{MagnaniStepanovTrevisan}) or led to a coarea-type inequality rather than an identity: see \cite[Theorem 2.6]{MagnaniCoareaInequality}.

We solve this problem as follows. A vertical curve $\zeta\from [0,1]\to \HH_n$ with $z(\zeta(0)^{-1}\zeta(1)) > 0$ projects to a curve $\gamma$ in $\R^{2n}$. We connect the endpoints of $\gamma$ by a line segment to obtain a closed curve $\hat{\gamma}$. If $\gamma$ is Lipschitz, then $\hat{\gamma}$ corresponds to a $1$--chain $[\hat{\gamma}]$ and there is a $2$--chain $\beta$ such that $\partial \beta = [\hat{\gamma}]$. Let $\omega$ be the standard symplectic form on $\R^{2n}$; then one can show
$$
\cH^2(\zeta) = z\left(\zeta(0)^{-1}\zeta(1)\right) - \int_\beta \omega;
$$
see, e.g., Theorem~\ref{thm:S-cH2}.

In general, however, $\gamma$ may not be Lipschitz. A vertical curve has Hausdorff dimension roughly $2$, so $\gamma$ is often $\frac{1}{2}$--Hölder. We thus define the symplectic area $S(\hat{\gamma})$ of $\hat{\gamma}$ as an integral that agrees with $\int_\beta \omega$ when $\gamma$ is Lipschitz; see Section~\ref{sec:symplectic} for details. Then, by work of Kozhevnikov \cite[Cor.\ 5.4.16]{Kozhevnikov}, if $\zeta$ is a fiber of a surjective $C^1_{\mathrm{H}}$--map for which $S(\hat\gamma)$ exists, $\cH^2(\zeta) = z(\zeta(0)^{-1}\zeta(1)) - S(\hat{\gamma})$: see Theorem~\ref{thm:S-cH2}.

Since $\gamma$ can be $\frac{1}{2}$--Hölder or worse, $S(\hat{\gamma})$ may not exist; this is the case for the examples constructed in \cite{Kozhevnikov} and \cite{AY-Vertical}. We must therefore give a criterion for when $S(\hat{\gamma})$ exists and show that the criterion holds for almost every fiber of $f$. We can divide the proof into five rough parts.

\begin{enumerate}
\item As in the Euclidean case, we can reduce the coarea formula for Lipschitz maps to the coarea formula for $C^1_\mathrm{H}$ maps $f\from \HH_n\to \R^{2n}$ (Section~\ref{sec:ReductionToC1H}). 
\item A coarea inequality due to Magnani \cite{MagnaniCoareaInequality} (Theorem~\ref{thm:CoareaInequality}) implies the coarea formula when $\JfH = 0$. 
We thus consider the case that $x\in \HH_n$ and $\JfH(x) \ne 0$. After composing $f$ with a linear map, we may suppose that $\DfH$ is close to the identity on a small ball $B$ around $x$; then $f^{-1}(w)\cap B$ is a vertical curve (see Section \ref{sec:verticalcurves}) for all $w\in f(B)$.
\item Let $f^{-1}(w)\cap B$ be a vertical curve and let $K\subset f^{-1}(w)\cap B$ be a connected subset with $\diam(K)>0$. We parameterize $K$ by a map $\zeta\from [0,1]\to K$ and use affine approximations of $f$ to define a sequence $g_i$ of piecewise-linear approximations of $\gamma = \pi\circ \zeta$ (the projection of $\zeta$ onto the horizontal plane). We can bound the geometry of the $g_i$'s in terms of the $\beta$--numbers of $f$. (For any ball $B\subset \HH_n$, $\beta_f(B)$ measures how well $f|_B$ can be approximated by an affine function, see Section~\ref{sec:approximations} for details.)
\item Let $\hat g_i$ be the closed curve obtained by connecting the endpoints of $g_i$ by a line segment. We show that the limit $\lim_{i\to \infty} S(\hat g_i)$ is bounded in terms of a sum of squares of $\beta$--numbers of $f$. Furthermore, when this sum is finite, $S(\hat\gamma)$ is well-defined and equal to $\lim_{i\to \infty} S(\hat g_i)$. This is the most technically involved section of the proof, so we state the main results in Section~\ref{sec:symplectic-area} and prove them in Section~\ref{sec:area-formula}.
\item Finally, Fässler and Orponen's version of Dorronsoro's Theorem \cite{FO-Dorronsoro} gives a bound on sums of squares of $\beta$--numbers of $f$. We use this to show that for almost every $w\in \R^{2n}$, if $K\subset f^{-1}(w)\cap B$ is connected, then $\pi(K)$ has well-defined symplectic area. Estimating this symplectic area lets us prove Theorem~\ref{thmCoareaLip}. (Section~\ref{sec:density-proof})

\end{enumerate}

While this strategy proves the coarea formula for Lipschitz maps $\HH_n\to \mathbb R^{2n}$, the case of maps $\HH_n\to \mathbb R^{k}$ with $n\ge 2$ and $n+1\le k < 2n$ remains open. As with the case considered here, the fibers of a map $\HH_n\to \mathbb R^{k}$ might be unrectifiable according to the theory developed starting in \cite{FSSCMathAnn}, so it is possible that techniques in this paper will generalize to such maps.

\subsection{Area formula for \texorpdfstring{$\frac{1}{2}$--Hölder}{1/2--Hölder} curves}\label{sec:intro-area-rn}

The core of this paper is the proof of Proposition~\ref{prop:vertical-curve-area}, which gives a condition for the existence of the symplectic area $S(\gamma)$ of the projection of a vertical curve. The condition in Proposition~\ref{prop:vertical-curve-area} requires some additional notation and definitions, so we will not state it here. Nonetheless, the methods of Proposition~\ref{prop:vertical-curve-area} can be used to prove the following criterion for the existence of the signed area of a $\frac{1}{2}$--Hölder curve in $\R^2$.

For any continuous map $\gamma = (\gamma_1, \gamma_2) \from [a,b]\to \R^2$, we give a geometric meaning to the integral
\begin{equation}\label{eq:def-A-integral}
  A(\gamma) = \frac{1}{2} \int_a^b (\gamma_1 \ud \gamma_2 - \gamma_2 \ud \gamma_1),
\end{equation}
as follows. For any partition $P=\{t_1,\dots, t_k\}$ with $a=t_1 < \dots < t_k = b$, we let
\begin{equation}\label{eqn:signed-area}
  A_P(\gamma) = \frac{1}{2}\sum_{i=1}^{k-1} (\gamma_1(t_i) \gamma_2(t_{i+1}) - \gamma_2(t_i) \gamma_1(t_{i+1})),
\end{equation}
so that $A_P(\gamma)$ is the signed area of the curve connecting $\gamma(t_1), \dots, \gamma(t_k)$ by line segments. Let $\mesh(P) = \max_{i}(t_{i+1}-t_i)$. We say that $A(\gamma)$ exists if there is an $L\in \R$ such that for any $\epsilon>0$, there is a $\delta>0$ such that if $\mesh(P)<\delta$, then $|L-A_P(\gamma)|<\epsilon$. In this case, we write
\begin{equation}\label{eq:def-A}
  A(\gamma) = \lim_{\mesh(P)\to 0} A_P(\gamma) = L,
\end{equation}
and call $A(\gamma)$ the \emph{(signed) area} of $\gamma$. When $\gamma$ is Lipschitz, one can prove that $A(\gamma)$ is the usual signed area of $\gamma$ (see Section~\ref{sec:symplectic}). 
  
  It is natural to approximate $A(\gamma)$ in terms of approximations of $\gamma$. For example, for $i\ge 0$, we can let $\gamma_i$ be the piecewise-linear curve connecting $\gamma(0), \gamma(2^{-i}), \gamma(2\cdot 2^{-i}), \dots, \gamma(1)$. When $\gamma$ is $\alpha$--Hölder, the region between $\gamma_i$ and $\gamma_{i+1}$ consists of $2^i$ triangles with diameter of order $2^{-i\alpha}$, so 
  $$|A(\gamma_{i+1}) - A(\gamma_i)| \lesssim 2^i 2^{-2i\alpha}.$$
  When $\alpha > \frac{1}{2}$, this difference decays exponentially, so $\lim_i A(\gamma_i)$ exists, and in fact one can use Young integration to show that $A(\gamma) = \lim_i A(\gamma_i)$. 
  
  When $\gamma$ is $\frac{1}{2}$--Hölder, however, this argument breaks down; $\lim_i A(\gamma_i)$ may not converge, and even when it converges, $A(\gamma)$ may not exist (see Lemma \ref{lem:gamma-i-no-area}).
  
  Nonetheless, we prove the following theorem, which gives a criterion for the existence of $A(\gamma)$. 

\begin{thm}\label{thm:ExistenceAreaCurves}
Let $\gamma\from [0,1] \to \R^{2}$ be a $\frac{1}{2}$--H\"older curve.
  Let 
  $$
  \gamma_i = \overline{\gamma(0), \gamma(2^{-i}), \gamma(2\cdot 2^{-i}), \dots, \gamma(1)}
  $$
  be the curve connecting $\gamma(0), \gamma(2^{-i}), \dots, \gamma(1)$ by line segments. For all $i\geq 0$ and $0\leq j\leq 2^i-1$, let 
  $$
  \delta_{i,j}(\gamma) := \diam\left\{\gamma\big(j 2^{-i}\big), \gamma\big((2j+1) 2^{-i-1}\big), \gamma\big((j+1) 2^{-i}\big)\right\}.
  $$
  If \begin{equation}\label{eqn:CondSigmaIntro}
  \sigma(\gamma):=\sum_{i=0}^\infty \sum_{j=0}^{2^i-1} \delta_{i,j}(\gamma)^2 < \infty,
  \end{equation}
  then $A(\gamma)$ exists and $A(\gamma) = \lim_i A(\gamma_i)$.
\end{thm}
The $\frac{1}{2}$--Hölder condition in Theorem~\ref{thm:ExistenceAreaCurves} is sharp in the sense that for every $\epsilon>0$ one can construct $(\frac{1}{2}-\epsilon)$--H\"older curves with $\sigma(\gamma)<\infty$ for which $A(\gamma)$ does not exist, see Lemma~\ref{lem:gamma-i-no-area} and Lemma~\ref{lem:sigma-bound}. 

Theorem~\ref{thm:ExistenceAreaCurves} is a special case of Proposition~\ref{prop:rn-curve-area}, which we will prove in Appendix~\ref{app:rn-curve-area}. We will also show that under the assumptions of  Theorem~\ref{thm:ExistenceAreaCurves}, we have that $\mathcal{H}^2(\gamma([0,1]))=0$, and if $\gamma$ is a closed curve, then the winding number $\Wind(\gamma,z)$ of $\gamma$ around $z$ is defined for almost every $z$ and the signed area of $\gamma$ satisfies
$$
A(\gamma) = \int_{\R^2} \Wind(\gamma,z)\ud z,
$$
see Proposition~\ref{prop:winding}. For a discussion of how our result in Theorem~\ref{thm:ExistenceAreaCurves} compares with probabilistic notions of integration, including the It\^o integral for Brownian motion, see Section~\ref{sec:Probability}.

\subsection{Acknowledgments}

G.A. has been partially supported by the NSF DMS Grant No. 2550590. The authors would like to thank D.\ Vittone and F.\ Glaudo for useful discussions around the topic of this paper.

\vspace{0.5cm}

\section{Preliminaries}\label{sec:prel}
\medskip

\subsection{The Heisenberg groups and general notation}\label{sec:H1}

In this paper, we fix $n\geq 1$ a natural number. The $n$th Heisenberg group $\HH_n$ is the unique simply connected Lie group whose Lie algebra is $$\mathfrak{h}_n:=\mathrm{Lie}(\HH_n)=\langle X_1,Y_1,\ldots,X_n,Y_n\rangle \oplus \langle Z\rangle,$$ with $[X_i,Z]=[Y_i,Z]=0$ and $[X_i,Y_i]=Z$ for every $1\leq i\leq n$. We will identify $\HH_n$ with $\mathfrak{h}_n=\mathbb R^{2n+1}$ by means of exponential coordinates. Namely, 
\[
(x_1,y_1,\ldots,x_n,y_n,z) \leftrightarrow \exp(x_1X_1+y_1Y_1+\ldots+x_nX_n+y_nY_n+zZ).
\]

We denote the coordinate functions on $\HH_n$ by $x_i,y_i,z\from \HH_n\to \R$. Let $\pi\from \HH_n \to \R^{2n}$, $\pi(a) = (x_1(a),y_1(a),\ldots,x_n(a),y_n(a))$ and let $$\omega((x_1,y_1,\ldots,x_n,y_n),(\bar x_1,\bar y_1,\ldots,\bar x_n,\bar y_n)) = \sum_{i=1}^n x_i\bar y_i-\bar x_iy_i,$$ be the standard symplectic form on $\R^{2n}$. Then $[u,v] = \omega(\pi(u),\pi(v))Z$.
We can write the multiplication using the Baker--Campbell--Hausdorff formula; for $a,b\in \mathfrak{h}_n$, we have
$$
\exp(a)\exp(b) = \exp\left(a+b+\frac{1}{2}[a,b]\right).
$$
Using the identification of $\HH_n$ with $\mathfrak{h}_n$, we write
\begin{equation}
  \label{eq:group-law}
  a\cdot b = ab=a+b+\frac{1}{2}[a,b] = a + b + \frac{1}{2}\omega(\pi(a),\pi(b)) Z,
\end{equation}
or, in coordinates,
\[
\begin{aligned}
  a b &= (x_1^a,y_1^a,\ldots,x_n^a,y_n^a,z^a)(x_1^b,y_1^b,\ldots,x_n^b,y_n^b,z^b)  \\
  &= (x_1^a+x_1^b,y_1^a+y_1^b,\ldots,x_n^a+x_n^b,y_n^a+y_n^b,z(ab)),
\end{aligned}
\]
where 
\[z(a b):=z_a+z_b+\frac{1}{2}\sum_{i=1}^n x_i^ay_i^b-x_i^by_i^a.\]
(We use $ab$ and $a\cdot b$ interchangeably.)
  For $a\in \HH_n$, we denote the one-parameter subgroup associated to $a$ by $\langle a \rangle$. For $X\in\mathfrak{h}_n$ and $t\in\mathbb R$, we denote $X^t:=\exp(tX)$.

For $q\in \HH_n$, we define the Korányi norm of $q$ by
\begin{equation}\label{eq:def-kor}
  \|q\|_{\Kor} = \sqrt[4]{|\pi(q)|^4 + 16 z(q)^2},
\end{equation}
where $|\cdot|$ denotes the Euclidean norm in $\R^{2n}$.
This gives rise to a left-invariant metric $d_{\Kor}(p,q) = \|p^{-1}q\|_{\Kor}$. Let us denote with $\dcc$ the Carnot--Carath\'eodory metric on $\HH_n$, so that $\dcc(p,q)$ is the minimal length of a horizontal curve from $p$ to $q$.
\begin{lemma}\label{lem:cc-kor}
  For any $p, q \in \HH_n$,
  $$d_{\Kor}(p,q)\le \dcc(p,q) \le 3 d_{\Kor}(p,q).$$
\end{lemma}
\begin{proof}
  It suffices to consider the case $q=\zero$. If $p\in \HH_n$, then there is a horizontal geodesic $\gamma$ from $\zero$ to $p$ with length $\dcc(\zero,p)$. The length of a horizontal curve is the same in the Korányi and Carnot--Carathéodory metrics, so $d_{\Kor}(\zero,p)\le \dcc(\zero,p)$.

  There is a horizontal line of length $|\pi(p)|$ from $\zero$ to $\pi(p)$, and since
  $$\pi(p)^{-1} p = Z^{z(p)} = \left[\pm X_1^{\sqrt{|z(p)|}}, Y_1^{\sqrt{|z(p)|}}\right],$$ we have $\dcc(\pi(p), p)\le 4\sqrt{|z(p)|} \le 2 d_{\Kor}(\zero,p)$. Therefore,
  \[
  \begin{aligned}
  \dcc(\zero,p) &\le \dcc(\zero, \pi(p)) + \dcc(\pi(p), p) \le |\pi(p)| + 4\sqrt{|z(p)|} \\
  &\le 3 d_{\Kor}(\zero,p),
  \end{aligned}
  \]
  as desired.
\end{proof}

Throughout this paper, we will write $d$ for $d_{\Kor}$. 

For $r>0$, let $S_r\from \HH_n\to \HH_n$ be the scaling automorphism:
\[
S_r(x_1,y_1,\ldots,x_n,y_n,z):=(rx_1,ry_1,\ldots,rx_n,ry_n,r^2z).
\]
A continuous distance $D$ on $\HH_n$ is \emph{homogeneous} if $$D(S_r(a),S_r(b))=rD(a,b)$$ for every $a,b\in \HH_n$ and $r>0$. It is a standard fact that $d=d_{\Kor}$ is a left-invariant homogeneous distance and that it is bilipschitz equivalent to any left-invariant homogeneous distance on $\HH_n$. Throughout the paper we work without loss of generality with the Kor\'anyi distance.

The Hausdorff dimension of $\HH_n$ with respect to $d$ is $2n+2$. Let $\mathcal{H}^\ell$ denote the $\ell$--dimensional Hausdorff measure with respect to $d$. We denote by $\mathcal{L}^\ell$ the $\ell$--dimensional Lebesgue measure on $\mathbb R^\ell$. In this paper we normalize the Hausdorff measure $\mathcal{H}^2$ so that 
\begin{equation}\label{eq:h2-normalize}
  \mathcal{H}^2(\{Z^a : a\in [0,1]\})=1
\end{equation}
and normalize $\mathcal{H}^{2n+2}$ to be equal to the Lebesgue measure $\cL^{2n+1}$ on $\HH_n\equiv \mathbb R^{2n+1}$. When we integrate over $\HH_n$ or $\R^{2n}$, it is with respect to the appropriate Lebesgue measure unless otherwise specified.

Let $U$ be a measurable subset of $\HH_n$. We say that $f\from U\to\mathbb R^{2n}$ is {\em Pansu differentiable at $x\in U$} if there exists a homogeneous (commuting with $S_r$ for every $r>0$) homomorphism $Df_x\from\HH_n\to\mathbb R^{2n}$ such that
\begin{equation}\label{eqn:PDiff}
\lim_{U\ni y\to x}\frac{|f(y)-f(x)-Df_x(x^{-1}\cdot y)|}{d(x,y)}=0.
\end{equation}
If $Df_x$ exists, then $Df_x( Z)= Df_x([X_1,Y_1])=[Df_x(X_1),Df_x(Y_1)]=0$. With a slight abuse of notation, we will denote with $Df_x$ also the linear map $ Df_x\circ\iota\from \mathbb R^{2n}\to\mathbb R^{2n}$, where $\iota\from\{z=0\}\hookrightarrow \mathbb \HH_n$. If $L\from\mathbb R^{2n}\to\mathbb R^{2n}$ is linear, we denote with $\|L\|_{\mathrm{Op}}$ the usual operator norm:
\[
\|L\|_{\mathrm{Op}}:=\sup_{v:|v|=1} |L(v)|.
\]

Let $\Omega\subset \HH_n$ be an open set. We say that $f\from\Omega\to\mathbb R^{2n}$ is {\em of class $C^1_{\mathrm H}$ in $\Omega$} if the map $x\mapsto Df_x$ is continuous on $\Omega$. By \cite[Theorem 5.3.7]{Kozhevnikov}, fibers of nonsingular $C^1_{\mathrm H}$ maps are curves.
\begin{thm}[{\cite[Theorem 5.3.7]{Kozhevnikov}}]
  Let $f\from \HH_n \to \R^{2n}$ be a $C^1_{\mathrm{H}}$ map and let $q\in \HH_n$ be a point such that $Df_q$ is nonsingular. There is a neighborhood $W$ of $q$ such that $f^{-1}(f(q))\cap W$ is a simple curve.
\end{thm}
Consequently
\begin{cor}\label{cor:nonsingular}
  Let $U\subset \HH_n$ be an open set and let $f\from \HH_n \to \R^{2n}$ be a $C^1_{\mathrm{H}}$ map such that $Df_q$ is nonsingular for any $q\in U$. Then, for any $w\in f(U)$, $f^{-1}(w)\cap U$ is a $1$--manifold that is closed in the relative topology on $U$.
\end{cor}

Furthermore, $C^1_{\mathrm{H}}$ maps satisfy the following lemma. If $B$ is a ball, we let $\rad(B)$ be its radius and for $\rho>0$, we let $\rho B$ be the ball of radius $\rho \rad(B)$ with the same center.
\begin{lemma}\label{lem:ControlF}
  Let $f\from \HH_n\to \R^{2n}$ be a $C^1_{\mathrm{H}}$ map, let $c>0$, and let $B=B_R(p)\subset \HH_n$ be a ball such that $\|Df_q - \id\|_{\mathrm{Op}}<c$ for all $q\in 5B$. Then for every $x,y\in B$,
  \begin{equation}\label{eq:ControlFPrel}
    |(f(x)-\pi(x))-(f(y)-\pi(y))|\le 3c d(x,y).
  \end{equation}
\end{lemma}
\begin{proof}
  For $x\in \HH$, let
  $$\tilde f(x)=f(x)-\pi(x)-(f(p)-\pi(p)).$$
  Then $\tilde{f}(p) = 0$ and $\|D\tilde{f}_q\|_{\mathrm{Op}}<c$ for all $q\in 5B$. By Lemma~\ref{lem:cc-kor}, if $x,y \in B$, there is a horizontal curve $\gamma$ from $x$ to $y$ with $\ell(\gamma) \le 3d(x,y) \le 6R$. Then $\gamma\subset 5B$, so 
  $$|\tilde{f}(y) - \tilde{f}(x)| = \left|\int_0^1 D\tilde{f}_{\gamma(t)}(\gamma'(t))\ud t\right| \le c \ell(\gamma) \le 3 c d(x,y),$$
  as desired.
\end{proof}

Finally, in this paper we will frequently use the notation
\[
A \lesssim B,
\]
to mean that there is a universal constant $C>0$ such that $A\leq CB$. Similarly for $A\gtrsim B$. When we write $A\approx B$ we mean that $B\lesssim A \lesssim B$. Moreover, when we write $A\lesssim_{\lambda} B$, we mean that the constant $C$ might depend on $\lambda$. Similarly for $\gtrsim_{\lambda}$, and $\approx_\lambda$.

\subsection{General facts on vertical curves}\label{sec:verticalcurves}
 For $\lambda>0$, let
\begin{equation}\label{eqn:Vcone}
  \begin{split}
    \VCone_\lambda & := \left\{q\in \HH_n : |z(q)| \ge \lambda |\pi(q)|^2\right\} \\ 
    & = 
  \bigg\{(x_1,y_1,\ldots,x_n,y_n,z)\in \HH_n : |z| \ge \lambda \sum_{i=1}^n(x_i^2+y_i^2)\bigg\}.
  \end{split}
\end{equation}
This is a homogeneous cone centered on $\langle Z\rangle$; as $\lambda$ increases, the cone converges to $\langle Z\rangle$. A subset $E\subset \HH_n$ is a \emph{$\lambda$--vertical curve} if $E\subset p \VCone_\lambda$ for all $p\in E$. Notice that as $\lambda$ increases, $\VCone_\lambda$ becomes smaller and $\lambda$--vertical curves become closer to vertical lines. We should understand $\lambda$ as a measure of the \emph{verticality} of the curve.

Our interest in vertical curves stems from the following consequence of Corollary~\ref{cor:nonsingular} and  Lemma~\ref{lem:ControlF}. 
\begin{cor}\label{cor:fibers-are-vertical}
  Let $\lambda > 0$. There is a $c>0$ with the following property. 
  Let $B \subset \HH_n$ be an open ball and let $f\from \HH_n \to \R^{2n}$ be a $C^1_{\mathrm{H}}$ map such that $\|Df_q - \id\|_{\mathrm{Op}} < c$ for all $q\in 5B$. Then, for any $w\in f(B)$, $B\cap f^{-1}(w)$ is a $\lambda$--vertical $1$--manifold that is relatively closed in $B$. If $\lambda = 2$, then $c = \frac{1}{20}$ suffices.
\end{cor}
\begin{proof}
  We suppose that $c<\frac{1}{6}$. Then $Df_q$ is nonsingular for $q\in B$, so $B\cap f^{-1}(w)$ is a $1$--manifold by Corollary~\ref{cor:nonsingular}.
  If $p,q\in B$ and $f(p) = f(q)$, then $|\pi(p)-\pi(q)|\le 3c d(p,q)$ by Lemma~\ref{lem:ControlF}. Then
  \begin{multline*}
    d(p,q) = \sqrt[4]{|\pi(p) - \pi(q)|^4 + 16 z(p^{-1}q)^2} \le |\pi(p) - \pi(q)| + 2\sqrt{|z(p^{-1}q)|} \\
    \le \frac{1}{2} d(p,q) + 2\sqrt{|z(p^{-1}q)|}.
  \end{multline*}
  Therefore, 
  $$|z(p^{-1}q)| \ge \frac{1}{16} d(p,q)^2 \ge \frac{1}{144 c^2} |\pi(p^{-1}q)|^2.$$
  If $144c^2\lambda < 1$, this implies $q\in p\VCone_\lambda$, as desired. 
\end{proof}

We define
\[
\VCone_\lambda^+ :=\VCone_\lambda\cap\{z\geq 0\}, \qquad \VCone_\lambda^- :=\VCone_\lambda\cap\{z\leq 0\}.
\]
Define the following relation.
\begin{equation}\label{eq:def-order}
    g \prec h \qquad \text{if $z(g^{-1}h) > 0$}.
\end{equation}

Lemmas~\ref{lem:vertical-properties}--\ref{lem:Biholder} are direct analogues of lemmas that were proved in \cite{AY-Vertical} for the first Heisenberg group. The proofs in $\HH_1$ work for $\HH_n$ as well.

\begin{lemma}[{\cite[Lemma 3.1(3) \& Proposition 3.2]{AY-Vertical}}]\label{lem:vertical-properties}
 Let $\lambda>0$, and let $E\subset \HH_n$ be a $\lambda$--vertical curve. Then:
  \begin{enumerate}
  \item\label{Item3} If $\lambda > \frac{1}{4}$, or if $E$ is connected, then \eqref{eq:def-order} defines a total order on $E$. Moreover, if $E$ is the image of an injective map $\gamma\from I\to\HH_n $, then $\gamma$ is either monotone increasing in the sense that
  \begin{equation}\label{eqn:First}
  \gamma(s) \prec \gamma(t) \quad \forall s<t,
  \end{equation}
  or it is monotone decreasing in the sense that
  \begin{equation}\label{eqn:Second}
  \gamma(s) \succ \gamma(t) \quad \forall s<t.
  \end{equation}
  \item  Let $\lambda>0$ and let $E$ be a connected $\lambda$--vertical curve. Then $E$ is homeomorphic to an interval and thus $E$ is parametrized by an increasing $\lambda$--vertical map $\gamma\from I \to\HH_n$.
\end{enumerate}
\end{lemma}

\begin{lemma}[{\cite[Lemma 3.3]{AY-Vertical}}]

\label{lem:compact-intersections-v2}
  Let $\lambda>0$. There is a $C>0$ such that for all $p,q\in \HH_n$, if $q\in p\VCone_\lambda^+$, then
  $$\diam\left(p\VCone_\lambda^+\cap q \VCone_\lambda^-\right)\le C d(p,q).$$
  
  In particular, if $E$ is a connected $\lambda$--vertical curve, $p,q\in E$ with $p\prec q$, and $I\subset E$ is the subinterval of $E$ from $p$ to $q$, then $\diam I\le C d(p,q)$. 
\end{lemma}

\begin{lemma}[{\cite[Lemma 3.5]{AY-Vertical}}]\label{lem:Biholder}
  There is a $0<\theta<1$ depending on $\lambda$ such that the following holds. 
  Let $E\subset \HH_n$ be a connected $\lambda$--vertical curve.
  Let $v, w\in E$ such that $v\prec w$ and let $0 < \epsilon < 1$. There is an $N\in \mathbb{N}$ depending on $n$, $\lambda$, and $\epsilon$ such that there is a sequence $v=q_0\prec q_1\prec \ldots\prec q_k=w$ with $k\leq N$ that satisfies
  \begin{equation}\label{eqn:biHolder-upper}
    d(q_i,q_{i+1})\leq \epsilon d(v,w)\qquad \text{for $i=0,\dots,k-1$}
  \end{equation}
  and
  \begin{equation}\label{eqn:biHolder-lower}
    d(q_i,q_j)\geq \theta \epsilon d(v,w)\qquad \text{for all $i\ne j$.}
  \end{equation}
\end{lemma}

The following lemmas are consequences of these bounds.
\begin{lemma}\label{lem:vertical-diameter-v2}
  Let $\lambda>0$ and let $\zeta\from I\to \HH_n$ be a vertical curve which is monotone increasing in the sense of Lemma~\ref{lem:vertical-properties}. Then for all $s,t\in I$ with $s<t$,
  $$\diam \zeta([s,t])^2 \approx_\lambda d(\zeta(s),\zeta(t))^2\approx_\lambda z(\zeta(s)^{-1} \zeta(t)).$$
\end{lemma}
\begin{proof}
  If $C$ is as in Lemma~\ref{lem:compact-intersections-v2}, then  
  $$d(\zeta(s),\zeta(t))\le \diam \zeta([s,t]) \le C d(\zeta(s),\zeta(t)).$$
  Let $g=\zeta(s)^{-1}\zeta(t)$ so that $g\in \VCone_\lambda^+$. Then $z(g)\ge 0$ and $|\pi(g)|^2 \le \lambda^{-1}z(g)$. By \eqref{eq:def-kor}, 
  $$2\sqrt{z(g)} \le d(\zeta(s),\zeta(t)) = \sqrt[4]{|\pi(g)|^4 + 16 z(g)^2} \le (16 + \lambda^{-2})^{\frac{1}{4}}\sqrt{z(g)}.$$
  This proves the lemma.
\end{proof}

\begin{lemma}\label{lem:chains}
  For any $c>0$ and $\lambda>0$, there is an $N>0$ with the following property. Let $\zeta\from [a,b]\to \HH_n$ be a $\lambda$--vertical curve, and let $a \le s_1 < \dots < s_k \le b$ be such that $\diam(\zeta([s_{j-1}, s_j])) \ge c \diam(\zeta([a,b]))$ for all $j$. Then $k \le N$. 
  
  (By Lemma~\ref{lem:vertical-diameter-v2}, we can replace $\diam(\zeta([s_{j-1}, s_j]))$ (resp.\ $\diam(\zeta([a,b]))$) by $d(\zeta(s_{j-1}),\zeta(s_j))$ (resp.\ $d(\zeta(a),\zeta(b))$), up to possibly changing $c$.)
\end{lemma}
\begin{proof}
  Without loss of generality, suppose that $\zeta$ is monotone increasing in the sense of Lemma~\ref{lem:vertical-properties}.  
  
  By Lemma~\ref{lem:Biholder} and Lemma~\ref{lem:vertical-properties}, there is a sequence $a = t_0 < t_1< \dots < t_m = b$ such that 
  $$
  \diam \zeta([t_{i-1}, t_{i}]) < c \diam \zeta([a,b]),
  $$
  for $i = 1,\dots, m$, and $m$ is bounded by some $N\in \mathbb{N}$ which depends only on $n$, $\lambda$, and $c$.

  By way of contradiction, suppose that $k > m$. By the pigeonhole principle, there are $i$ and $j$ such that $\{s_{j-1},s_j\} \subset [t_{i-1},t_i]$. Then 
  $$\diam \zeta([s_{j-1},s_j]) \le \diam \zeta([t_{i-1},t_i]) < c \diam \zeta([a,b]),$$
  which is a contradiction. Therefore, $k \le N$, as desired.
\end{proof}

\subsection{Symplectic area}\label{sec:symplectic}

In this section, we define the symplectic area of a curve in $\R^{2n}$ and prove some basic properties. We first recall the symplectic area of a Lipschitz curve.
\begin{defn}\label{def:Symplectic}
  Let $a<b$ and let $\gamma \from [a,b] \to \R^{2n}$ be a Lipschitz curve. Let
  \begin{equation}\label{eq:def-zeta}
    \zeta = \frac{1}{2} \sum_i (x_i \ud y_i - y_i \ud x_i),
  \end{equation}
    
  so that $\ud \zeta = \omega$ is the standard symplectic form. We define 
  \begin{equation}\label{eqn:SymplecticArea0}
  \hat{S}(\gamma) = \int_\gamma \zeta.
  \end{equation}
\end{defn}

We extend $\hat{S}$ to a class of continuous curves as follows.
For points $p_1,\dots, p_k\in \R^{2n}$ let $\overline{p_1,\dots,p_k}$ denote the piecewise-linear curve that connects $p_1,\dots,p_k$ by line segments.

\begin{defn}\label{def:LevySimplecticArea}
  Let $a<b$ and let $\gamma\from [a,b] \to \R^{2n}$ be a continuous map. For any partition $P=\{t_1,\dots,t_k\}$ with $a=t_1<\dots<t_k=b$, we let 
  $$\gamma_P = \overline{\gamma(t_1),\dots, \gamma(t_{k})},$$
  parameterized so that $\gamma_P(t_i) = \gamma(t_i)$ for all $i$. Then
  \begin{equation}\label{eqn:SymplecticArea}
    \hat{S}(\gamma_P) = \sum_{i=1}^{k-1} \sum_{j=1}^n \frac{x_j(\gamma(t_i)) y_j(\gamma(t_{i+1})) - y_j(\gamma(t_i)) x_j(\gamma(t_{i+1}))}{2}.
  \end{equation}
    
  Let $\mesh(P)=\max_i \{t_{i+1}-t_i\}$. We say that \(\lim_{\mesh(P)\to 0} \hat{S}(\gamma_P)=L\) exists if for any $\epsilon>0$, there is a $\delta>0$ such that if $\mesh(P)<\delta$, then $|L - \hat{S}(\gamma_P)|<\epsilon$. In this case, we write
  $$
  S(\gamma) = \lim_{\mesh(P)\to 0} \hat{S}(\gamma_P),
  $$
  and call $S(\gamma)$ the \emph{symplectic area} of $\gamma$.
\end{defn}

Then the following lemma holds.
\begin{lemma}\label{lem:symplectic-consistent}
  Let $\gamma\from I\to\mathbb R^{2n}$ be a Lipschitz curve. Then the symplectic area $S(\gamma)$ is finite and $S(\gamma) = \hat{S}(\gamma)$.
\end{lemma}
\begin{proof}
Let $\gamma_P$ be the piecewise-linear curve introduced in Definition~\ref{def:LevySimplecticArea}. Then
    \[\hat{S}(\gamma_P) =
    \int_{\gamma_P}\zeta = \frac{1}{2}\sum_{i=1}^n\int_I \left((\gamma_P)_{x_i}(s)(\gamma_P)_{y_i}'(s)- (\gamma_P)_{y_i}(s)(\gamma_P)_{x_i}'(s)\right)\ud s.
    \]
    Now notice that as $\mesh(P)\to 0$ we have that $\gamma_P\to \gamma$ uniformly and $\gamma_P'\to \gamma'$ in $L_1$ by the dominated convergence theorem. Thus, as $\mesh(P)\to 0$ we have
    \[
    \int_{\gamma_P}\zeta\to\int_\gamma \zeta = \frac{1}{2}\sum_{i=1}^n\int_I \left(\gamma_{x_i}(s)\gamma_{y_i}'(s)- \gamma_{y_i}(s)\gamma_{x_i}'(s)\right)\ud s<\infty,
    \]
    as desired.
\end{proof}

Since $\hat{S}(\gamma) = S(\gamma)$ for Lipschitz curves, we will write $S(\gamma)$ for the symplectic area of $\gamma$, regardless of whether $\gamma$ is Lipschitz or not. Then $S(\gamma)$ satisfies properties similar to $\hat{S}(\gamma)$.

\begin{lemma}\label{lem:symplectic-area-props-v2}   Let $I\subset \R$ be a closed interval. The following properties hold:
  \begin{enumerate}
  \item Let $I,J\subset \R$ be compact intervals and let $\gamma\from I\to\mathbb R^{2n}$ be a continuous curve. If $S(\gamma)$ exists and if $\alpha\from J\to I$ is a monotone nondecreasing homeomorphism, then $S(\gamma\circ \alpha) = S(\gamma)$.
  \item If $\gamma_1$ and $\gamma_2$ are continuous curves, $S(\gamma_1)$ and $S(\gamma_2)$ exist, and $\gamma_1\diamond \gamma_2$ is their concatenation, then $S(\gamma_1\diamond \gamma_2) = S(\gamma_1)+S(\gamma_2)$.
  \item Let $\gamma\from I\to\mathbb R^{2n}$ be a closed curve with finite length. Then 
  \[
  |S(\gamma)|\le \ell(\gamma)^2,
  \]
  where $\ell(\gamma)$ is the length of $\gamma$.
  \item Let $\gamma,\gamma_i\from I\to\mathbb R^{2n}$ be curves with finite length, and assume that $\gamma_i\to\gamma$ uniformly on $I$, and that $\sup_{i\in\mathbb N}\ell(\gamma_i)<\infty$. Then the symplectic area of $\gamma$ exists and
  \[
  \lim_{i\to\infty}S(\gamma_i)=S(\gamma).
  \]
\item If $\alpha\from I\to \R^{2n}$ is a closed curve, $v\in \R^{2n}$, and $\beta(t) = \alpha(t) + v$, then $S(\alpha) = S(\beta)$.
  \end{enumerate}
\end{lemma}
\begin{proof}
  Items (1) and (2) directly follow from Definition \ref{def:LevySimplecticArea}.
    
  Let us prove Item (3). Up to reparametrization, and using Item (1), we can assume $\gamma\from [0,\ell]\to \mathbb R^{2n}$ is Lipschitz and $|\gamma'(t)|=1$ a.e., where $\ell=\ell(\gamma)$. Since $\gamma$ is closed, $S(\gamma)=S(\gamma-a)$ for every $a\in\mathbb R^{2n}$, as an immediate consequence of \eqref{eqn:SymplecticArea0}. We translate $\gamma$ so that $\gamma(0) = 0$. Then
  \begin{equation}\label{eq:cone-ineq}
    |S(\gamma)|=\frac{1}{2}\left|\int_0^\ell \sum_{i=1}^n \left(\gamma_{x_i}(s)\gamma_{y_i}'(s) - \gamma_{y_i}(s)\gamma_{x_i}'(s)\right)\ud s\right| \leq \int_0^\ell |\gamma'(s)||\gamma(s)|\ud s.
  \end{equation}
  Since $\gamma(0) = 0$, $|\gamma(s)| \le \ell$ for all $s$, so $|S(\gamma)|\le \ell^2$, as desired.

  Next, we prove Item (4). Suppose that $\gamma_i\to \gamma$ uniformly and $L$ is such that $\ell(\gamma_i) < L$. Since the length functional is lower semicontinuous, we have $\ell(\gamma) < L$ as well. We first claim that if $S(\gamma_i) \to A$, then $S(\gamma) = A$.  

  Suppose that $S(\gamma_i) \to A$ and let $\epsilon > 0$. Let $i$ be such that $\|\gamma - \gamma_i\|_\infty < \epsilon$ and $|S(\gamma_i) - A| < \epsilon$. Let $\delta > 0$ be such that if $\mesh(P)<\delta$, then $|S((\gamma_{i})_P)-S(\gamma_{i})|<\epsilon$, where $(\gamma_i)_P$ is defined as in \eqref{eqn:SymplecticArea}.
  
  Let $P$ be a partition of $I$ with $\mesh(P)<\delta$. By our choice of $\delta$ and $i$, $|S((\gamma_i)_P) - A| < 2\epsilon$. Using the triangle inequality, it is readily shown that 
  \[
    |S(\gamma_P)-S((\gamma_i)_P)|\lesssim \|\gamma-\gamma_i\|_{L_\infty(I)}\left(\ell(\gamma)+\ell(\gamma_i)\right) \le 2L\epsilon.
    \]
  Therefore, by the triangle inequality, $|S(\gamma_P)- A| \le (2L+2)\epsilon$ for any $P$ with $\mesh(P)<\delta$. Letting $\epsilon$ go to zero, we get $\lim_{\mesh(P) \to 0}S(\gamma_P) = A$, as desired.

  Next, let $M = \sup_{i,t} |\gamma_i(t)|$; since $\gamma_i \to \gamma$ uniformly, we have $M < \infty$. By \eqref{eq:cone-ineq}, $|S(\gamma_i)| \le \ell(\gamma_i) M \le LM$. That is, $(S(\gamma_i))_i$ is a bounded sequence. By the above, any convergent subsequence of $(S(\gamma_i))_i$ converges to $S(\gamma)$, so in fact, $\lim_{i\to \infty} S(\gamma_i) = S(\gamma)$.

  Finally, to prove (5), note that if $\gamma$ is a Lipschitz curve and $\lambda(t) = v + \gamma(t)$, then $\hat{S}(\gamma) = \hat{S}(\lambda)$. It follows that for any partition $P$, $\hat{S}(\alpha_P) = \hat{S}(\beta_P)$, so $S(\alpha) = S(\beta)$, as desired.
\end{proof}

Finally, the following identity relates the symplectic area of a curve to the product formula in the Heisenberg group.
\begin{lemma}\label{lem:discrete-area-z}
  Let $v_1,\dots, v_k\in \HH_n$ and let $w_i=v_1\dots v_i$ for $i=1,\dots,k$. Let $\gamma= \overline{\pi(w_1),\dots,\pi(w_k)}$. Then 
  $$
  z(w_k) = S(\gamma) + \sum_{i=1}^k z(v_i).
  $$
\end{lemma} 
\begin{proof}
  We will show by induction that
  \begin{equation}\label{eq:discrete-area-inductive}
    z(w_j) = S(\overline{\pi(w_1),\dots,\pi(w_j)}) + \sum_{i=1}^j z(v_i)
  \end{equation}
  for all $j\ge 1$. This is trivial for $j=1$. If \eqref{eq:discrete-area-inductive} holds for some $j\ge 1$, one can use the definition of the group operation \eqref{eq:group-law} to verify that 
    \[
    z(w_{j+1}) = z(w_j v_{j+1}) = z(w_j) + z(v_{j+1}) + S(\overline{\pi(w_j),\pi(w_{j+1})}).
  \]
  By Lemma~\ref{lem:symplectic-area-props-v2}.(2), this implies \eqref{eq:discrete-area-inductive} for $j+1$.
\end{proof}

\subsection{Area formula for vertical fibers}\label{sec:AreaKoz}
The work of Kozhevnikov \cite{Kozhevnikov} implies the following area formula
for fibers of $C^1_{\mathrm{H}}$ maps $\HH_n\to\mathbb R^{2n}$.

\begin{thm}\label{thm:S-cH2}
  Let $\HH_n$ be endowed with an arbitrary left-invariant homogeneous distance. There exists $\eta<1$ such that the following holds. Let $F\from \HH_n\to\mathbb R^{2n}$ be a $C^1_{\mathrm{H}}$ map, and $B\subset\HH_n$ a ball such that $\|DF_q-\mathrm{id}\|_{\mathrm{Op}}<\eta$ for every $q\in 5B$. Let $w\in\mathbb R^{2n}$, and let $K$ be a connected subset of $F^{-1}(w)\cap B$ with $\diam(K)>0$. 
  
  Then there exists a monotone increasing (in the sense of Lemma~\ref{lem:vertical-properties}) parameterization $\alpha\from [0,1]\to\HH_n$ of $K$. Calling $\gamma:=\pi\circ \alpha$, we also have \begin{equation}\label{eqn:PreAreaFormulaVerticalFibers}
  \cH^2(K)=z(\alpha(1))-z(\alpha(0))-\limsup_{\mesh(P)\to 0}\hat S(\gamma_P),
  \end{equation}
  where $\hat S(\gamma_P)$ is defined in \eqref{eqn:SymplecticArea}. In particular, if $S(\gamma)$ exists,
  \begin{equation}\label{eqn:AreaFormulaVerticalFibers}
  \begin{aligned}
  \cH^2(K)&=z(\alpha(1))-z(\alpha(0))-S(\gamma) \\
  &=z(\alpha(0)^{-1}\alpha(1))-S(\hat\gamma),
  \end{aligned}
  \end{equation}
  where $\hat\gamma$ is the closed curve obtained by joining the endpoints of $\gamma$ with a line segment.
\end{thm}
\begin{proof}
    Let $\lambda=2$ and $\eta=c$ as in Corollary~\ref{cor:fibers-are-vertical}, then $F^{-1}(w)\cap B$ is a $2$--vertical curve for every $w\in\mathbb R^{2n}$. Thus every connected subset $K\subset F^{-1}(w)\cap B$ with $\diam(K)>0$ is parameterized by a monotone increasing $2$--vertical map $\alpha\from [0,1]\to\HH_n$ by Lemma~\ref{lem:vertical-properties}(2).

    Hence, taking into account \cite[Remark 5.3.9]{Kozhevnikov}, $K$ satisfies the two properties in \cite[Notation 5.3.8]{Kozhevnikov}. Thus, applying \cite[Corollary 5.4.16 and Remark 5.4.19]{Kozhevnikov}, we have
    \[
    \begin{aligned}
    \cH^2(K)&=z(\alpha(1))-z(\alpha(0))+\liminf_{\mesh(P)\to 0}(-\hat S(\gamma_P)) \\
    &=z(\alpha(1))-z(\alpha(0))-\limsup_{\mesh(P)\to 0}\hat S(\gamma_P),
    \end{aligned}
    \]
    as desired\footnote{In the statement of \cite[Corollary 5.4.16]{Kozhevnikov} the term $-\limsup_{\mesh(P)\to 0}\hat S(\gamma_P)$ in \eqref{eqn:PreAreaFormulaVerticalFibers} is replaced by $+4\liminf_{\mesh(P)\to 0}\hat S(\gamma_P)$. This happens because of the different definition of the group operation on $\HH_n$ in \cite{Kozhevnikov}: compare \cite[Pages 19-20]{Kozhevnikov} with our \eqref{eq:group-law}.}. In particular, if $S(\gamma)$ exists, we have 
    \[
    \limsup_{\mesh(P)\to 0} \hat S(\gamma_P) = \lim_{\mesh(P)\to 0} \hat S(\gamma_P) = S(\gamma).
    \]
    Hence, by using \eqref{eqn:PreAreaFormulaVerticalFibers}, \eqref{eq:group-law}, and Lemma \ref{lem:symplectic-area-props-v2}(2),  we conclude
    \[
    \begin{aligned}
        \cH^2(K) &= z(\alpha(1))-z(\alpha(0))-S(\gamma) \\
        & = z(\alpha(0)^{-1}\alpha(1))+\frac{1}{2}\omega(\pi(\alpha(0)),\pi(\alpha(1)))-S(\gamma)\\
        &=z(\alpha(0)^{-1}\alpha(1))-S(\overline{\pi(\alpha(1)), \pi(\alpha(0))})-S(\gamma)\\
        &=z(\alpha(0)^{-1}\alpha(1))-S(\hat\gamma),
    \end{aligned}
    \]
    as desired.
\end{proof}

\section{Reduction to \texorpdfstring{$C^1_{\mathrm{H}}$}{C\^{}1\_{}H} maps}\label{sec:ReductionToC1H}

In this section we begin the proof of Theorem~\ref{thmCoareaLip} by showing that if Theorem~\ref{thmCoareaLip} holds for $C^1_{\mathrm{H}}$ maps, then it holds for Lipschitz maps. Here and in the following sections, we fix some $n>0$ and let $\HH = \HH_n$.

That is, we prove Theorem~\ref{thmCoareaLip} modulo the following result.
\begin{prop}\label{propCoareaC1h}
    Let $f\from \HH\to\mathbb R^{2n}$ be a $C^1_{\mathrm{H}}$ function. Then for every measurable set $U\subset\HH$ we have
    \[
    \int_U |\JfH(x)|\ud x = \int_{\mathbb R^{2n}}\mathcal{H}^2(f^{-1}(v)\cap U)\ud v.
    \]
\end{prop}
Proposition~\ref{propCoareaC1h} is a direct consequence of Proposition \ref{prop:density}, which we will prove in Sections \ref{sec:CoAreaC1H} and \ref{sec:density-proof}.

We recall two results from the literature. The first is a coarea inequality due to Magnani \cite{MagnaniCoareaInequality}. We state it for maps from $\HH$ to $\R^{2n}$, but it holds more generally for Lipschitz maps between Carnot groups.
\begin{thm}[{\cite[Theorem 2.6]{MagnaniCoareaInequality}}]\label{thm:CoareaInequality}
    Let $f\from \HH\to \mathbb{R}^{2n}$ be a Lipschitz map. Then for every measurable set $U\subset \mathbb H$
\begin{equation}\label{eqn:CoareaInequality}
     \int_U |\JfH(x)|\ud x \geq \int_{\mathbb R^{2n}}\mathcal{H}^2(f^{-1}(v)\cap U)\ud v.
    \end{equation}
\end{thm}
The second result is an extension theorem proved in \cite{FSSCMathAnn}.
\begin{thm}[{\cite[Theorem 6.8]{FSSCMathAnn}}]\label{thmWithney}
    Let $H(\HH,\mathbb R)$ be the space of homogeneous homomorphisms $\HH\to\mathbb R$ endowed with the distance coming from the operator norm. Let $F\subset \HH$ be a closed set and let $f\from F\to\mathbb R$ and $k\from F\to H(\HH,\mathbb R)$ be continuous functions. For every compact set $K\subset F$ and $\delta>0$ define 
    \begin{equation}\label{eqn:rhokwhit}
    \rho_K(\delta):=\sup\left\{\frac{|f(q)-f(p)-k(p)(p^{-1}\cdot q)|}{d(p,q)}:p,q\in K, 0<d(p,q)<\delta\right\}.
    \end{equation}
    If for every compact $K\subset F$ we have that $\rho_K(\delta)\to 0$ as $\delta\to 0$, then there is $\tilde f\from \HH\to\mathbb R$ such that $\tilde f\in C^1_{\mathrm{H}}(\HH;\mathbb R)$ and 
    \[
    \tilde f|_F = f, \qquad D\tilde f|_F=k.
    \]
\end{thm}
\subsection{Proof of Theorem \ref{thmCoareaLip}}

The following proof is inspired by \cite[Theorem 3.5]{MagnaniMathNachr}. Let $f\from \HH=\HH_n\to\mathbb R^{2n}$ be a Lipschitz function. By Theorem \ref{thm:CoareaInequality} it is enough to show that for every compact set $U\subset \HH$,
\begin{equation}\label{eqn:Sought}
    \int_U |\JfH(x)|\ud x \leq \int_{\mathbb R^{2n}}\mathcal{H}^2(f^{-1}(v)\cap U)\ud v.
    \end{equation}
Let $U\subset \HH$ be a compact set. By the Pansu--Rademacher Theorem \cite{Pansu, MagnaniPhD}, $f$ is differentiable almost everywhere. Let $Df\from \HH\to H(\HH,\mathbb R^{2n})$ be its differential. Let $\epsilon>0$. By using Lusin and Egoroff's theorems we can find a compact (since $U$ is compact) set $F_\epsilon\subset U$ such that $\cH^{2n+2}(U\setminus F_\epsilon)<\epsilon$, $Df$ is continuous on $F_\epsilon$, and
\[
    \lim_{q\to p, q\in F_\epsilon} \frac{|f(q)-f(p)-Df(p)(p^{-1}\cdot q)|}{d(p,q)}=0,
\]
where the last limit is uniform for $p\in F_\epsilon$. We thus apply Theorem~\ref{thmWithney} to each component of $f$ with $F:=F_\epsilon$ and $k:=Df|_{F_\epsilon}$. This produces a function $\tilde f_\epsilon\from \HH\to\mathbb R^{2n}$ that is $C^1_{\mathrm{H}}$--regular such that 
\begin{equation}\label{eqn:Fepsilon}
  \tilde f_\epsilon = f \text{ and } D\tilde f_\epsilon=Df \qquad \text{on $F_\epsilon$.}
\end{equation}
Therefore $J_Hf = J_H\tilde{f}_\epsilon$ on $F_\epsilon$. Further, for any $v\in \R^{2n}$, if $x\in \tilde f_\epsilon^{-1}(v)\cap F_\epsilon$, then $x\in U$ and $f(x) = \tilde f_{\epsilon}(x) = v$. That is, $\tilde f_\epsilon^{-1}(v)\cap F_\epsilon \subset f^{-1}(v)\cap U$.

By Proposition~\ref{propCoareaC1h} and the fact that $f$ is Lipschitz,
\begin{align*}
  \int_{U} |J_{\mathrm{H}}f(x)|\ud x & = \int_{F_\epsilon} |J_{\mathrm{H}}\tilde{f}_\epsilon(x)|\ud x + \int_{U\setminus F_\epsilon} |J_{\mathrm{H}}f(x)|\ud x\\
  & \le \int_{\mathbb R^{2n}} \cH^2(\tilde f_\epsilon^{-1}(v)\cap F_\epsilon)\ud v + \Lip(f)^{2n} \epsilon \\
  & \le \int_{\mathbb R^{2n}} \cH^2(f^{-1}(v)\cap U)\ud v + \Lip(f)^{2n} \epsilon.
\end{align*}
Letting $\epsilon$ go to zero, we get 
\[
\int_U |J_{\mathrm H}f(x)|\ud x \le \int_{\mathbb R^{2n}} \cH^2(f^{-1}(v)\cap U)\ud v,
\]
as desired.

\section{The density of the coarea measure}\label{sec:CoAreaC1H}

As we saw in the previous section, the coarea formula for Lipschitz maps in Theorem \ref{thmCoareaLip} can be reduced to the following proposition.
\begin{prop}\label{prop:density}
  Let $f\from \HH = \HH_n\to \R^{2n}$ be a $C^1_{\mathrm{H}}$ map. Let $\mu_f$ be the measure such that
  $$
  \mu_f(U) := \int_{\R^{2n}} \cH^2(U\cap f^{-1}(w))\ud w
  $$
  for any measurable $U\subset \HH$.
  For almost every $p\in \HH$,
  $$\lim_{r\to 0} \frac{\mu_f(B_r(p))}{\cH^{2n+2}(B_r(p))} = |\JfH(p)|.$$
  Therefore, by the Radon--Nikodym theorem, $\mu_f(U) = \int_U |\JfH(p)|\ud x$ for any measurable $U\subset \HH$.
\end{prop}
Note that the function $w\mapsto \cH^2(U\cap f^{-1}(w))$ is $\cL^{2n}$--measurable for every measurable set $U\subset \HH$ by the results in \cite[¶2.10.26]{FedererBook}. We also have $\mu_f\ll \cH^{2n+2}$ due to Theorem \ref{thm:CoareaInequality}.

In the next two subsections, we will state some estimates that will be necessary to prove Proposition~\ref{prop:density}, leaving some proofs to later sections. Then, in Section~\ref{sec:density-proof}, we will use these estimates to prove Proposition~\ref{prop:density}.

As noted in the introduction, if $f\from \HH\to \R^{2n}$ is a $C^1_{\mathrm{H}}$ map with $\JfH(p)\ne 0$ for all $p$, the fibers of $f$ are \emph{vertical curves} in the sense of \cite{Kozhevnikov} or \cite{AY-Vertical}. While vertical curves are $C^0$--close to vertical lines, there can be a large difference between the measure of a vertical curve and the measure of a nearby vertical line. By Theorem~\ref{thm:S-cH2}, this difference is determined by the symplectic area of the projection of the curve to $\R^{2n}$. Thus, to prove Proposition~\ref{prop:density}, we need to bound the symplectic area of projections of fibers of $f$.

Let $w\in \mathbb R^{2n}$ and $B\subset \HH$ be a ball such that $f^{-1}(w)\cap B$ is nonempty. In general, $f^{-1}(w)\cap B$ may be disconnected; we let $K\subset f^{-1}(w)\cap B$ be a connected subset with $\diam(K)>0$. In Section~\ref{sec:approximations}, we will show that we can approximate the projection of $K$ by a sequence of piecewise-linear curves $g_i$ that converge uniformly. Furthermore, we can control the geometry of the approximations in terms of the $\beta$--numbers of $f$, which bound how well $f$ is approximated by an affine map. In Section~\ref{sec:symplectic-area}, we state a geometric criterion for when a sequence $(g_i)_i$ of curves satisfies $S(\lim g_i) = \lim_i S(g_i)$. In Section~\ref{sec:density-proof}, we will show that for almost any $w\in \R^{2n}$ and any choice of $K$ as above, this criterion is satisfied, and we will use it to prove Proposition~\ref{prop:density}.

\subsection{Approximations of \texorpdfstring{$f$ and $f^{-1}(w)$}{f and f\^{}-1(w)}}\label{sec:approximations}
In this section, we approximate the projection of $K$ by a sequence of piecewise-linear curves $g_i$. The main goal of this section is to prove Proposition~\ref{prop:construct-Lambda}, which lets us construct sequences of points that approximate $K$ and bound their geometry, but before we can state the proposition, we will need to introduce some notation.

The $g_i$'s will be based on a sequence of nested partitions of $K$ (similar to dyadic intervals), which we call a \emph{dyadic patchwork}. We will write this patchwork as a map that associates each vertex of a rooted tree $\cT$ to a subinterval of $K$.

We treat a \emph{rooted tree} $\cT$ as a collection of vertices with a distinguished root $v_0$ and a \emph{parent map} $\cP\from \cT\setminus\{v_0\}\to \cT$ such that for every $v\in \cT\setminus\{v_0\}$, there is some $i\ge 0$ such that $\cP^i(v)=v_0$. For $v\in \cT$, we denote the \emph{children} of $v$ by $\cC(v)=\cP^{-1}(v)$. For $i\ge 0$, we denote the \emph{$i$th-generation descendants} of $v$ by $\cC^i(v)= \cP^{-i}(v)$ and let $\cT^i=\cC^i(v_0)$. (In particular, $\cT^0=\{v_0\}$.) Let $\gen\from \cT\to \Z$ be the function such that $\gen(v) = i$ for all $v\in \cT^i$. For $k\ge 0$, let $\cT^{>k} = \bigcup_{i>k}\cT^i$ (and likewise for $\cT^{\ge k}$, etc.).
The tree structure induces a partial order on $\cT$, with $\cP(v)\succ v$ for all $v\in \cT\setminus\{v_0\}$ and $v_0=\max \cT$.

\begin{defn}\label{def:patchwork}
For $\mu > 1$, a \emph{$\mu$--dyadic patchwork} for $K$ is a tuple $$(\cT, \{I_v\}_{v\in \cT})$$ such that:
\begin{enumerate}
\item $\cT$ is a rooted tree with degree at most $\mu$.
\item Each set $I_v$ is a closed subinterval of $K$ with distinct endpoints, and if $v_0=\max \cT$, then $I_{v_0} = K$.
\item For each $v\in \cT$, the intervals $\{I_w\}_{w\in \cC(v)}$ partition $I_v$, i.e., $I_v=\bigcup_{w\in \cC(v)} I_w$ and the $I_w$'s overlap only at their endpoints.
\item For all $v\in \cT^i$,
  \begin{equation}\label{eq:patch-diam}
    \diam I_v\in [\mu^{-1} 2^{-i}\diam K, \mu 2^{-i}\diam K].
  \end{equation}
\end{enumerate}
\end{defn}
This is essentially a set of Christ cubes for $K$ (see \cite[Sec.\ 3]{ChristA-Tb-Theorem} or \cite[Appendix I]{DavidWaveletsAndSingular}), with the additional condition that the intervals $I_v$ are connected. 

Let $\zeta\from [0,1]\to K$ be a monotone increasing parameterization of $K$ (see Lemma \ref{lem:vertical-properties}) and let $\gamma= \pi\circ \zeta$. For each $v\in \cT$, let $J_v = \zeta^{-1}(I_v)\subset [0,1]$. This is a closed interval, and for any $i\ge 0$, the set of intervals $\{J_v:v\in \cT^i\}$ forms a partition of $[0,1]$. For any $v\in \cT$, let $m_v\in [0,1]$ be the midpoint of $J_v$. If $\gen(w) = \gen(w')$ and $J_w \cap J_{w'}\ne \emptyset$, then we say that $w$ and $w'$ are \emph{weak neighbors} (note that $w$ may equal $w'$).

If $C>0$ and $\Lambda \from \cT \to \R^{2n}$ is a map such that for any $v\in \cT$, 
\begin{equation}\label{eq:approx-points}
  |\Lambda(v) - \gamma(m_v)| \le C 2^{-\gen(v)} \diam(K),
\end{equation}
then we call $\Lambda$ a set of \emph{$C$--approximating points} or just \emph{approximating points} for $\cT$. This depends \emph{a priori} on the choice of $\zeta$, but by \eqref{eq:patch-diam}, we have $\diam(\pi(I_v))\le \mu 2^{-\gen(v)} \diam K$, so a different choice of $\zeta$ would only change the choice of $C$.

We will construct $\Lambda$ so that the geometry of $\Lambda$ is bounded by the $\beta$--numbers of $f$, which measure how well $f$ can be approximated by affine maps. We say that $\lambda\from \HH\to \R^{2n}$ is \emph{affine} if there is a linear map $h \from \R^{2n}\to \R^{2n}$ and a $w_0\in \R^{2n}$ such that $\lambda(p) =h(\pi(p)) + w_0$ for all $p$. Let $\Aff$ be the set of affine functions. (These are sometimes called \emph{vertical affine} functions.)

For $x\in \HH$ and $r>0$, we let $B_r(x)$ be the open metric ball of radius $r$ around $x$. If $B$ is a ball, we let $\rad(B)$ be its radius and for $\rho>0$, we let $\rho B$ be the ball of radius $\rho \rad(B)$ with the same center. For a map $\phi\from \HH \to \R^{2n}$ and a ball $B\subset \HH$, let
\begin{equation}\label{eqn:BetaPhiAlpha}
\beta_\phi(B) = \inf_{\alpha\in \Aff} \sqrt{\fint_{B} \frac{|\phi(p) - \alpha(p)|^2}{\rad(B)^2} \ud p}.
\end{equation}
This is a scale-invariant measure of how well $\phi$ can be approximated by an affine function on $B$. We will denote $\beta_\phi(p,r):=\beta_\phi(B_r(p))$.

Now we can state the following proposition.
\begin{prop}\label{prop:construct-Lambda}
  There are $c_0, C, \mu, \rho>0$ with the following property. Let $f\from \HH\to \R^{2n}$ be a $C^1_{\mathrm{H}}$ map and let $B\subset \HH$ be a ball such that $\|\DfH_q - \id\|_{\mathrm{Op}} < c_0$ for all $q\in \rho B$. Let $K$ be a connected subset of $B\cap f^{-1}(w)$ with $\diam(K)>0$. Then there is a $\mu$--dyadic patchwork $(\cT, \{I_v\}_{v\in \cT})$, a parameterization $\zeta\from [0,1]\to K$, and a set of $C$--approximating points $\Lambda$ such that for any $v, v'\in \cT$, if $v'$ is a weak neighbor of $v$ or of a child of $v$, then
  \begin{equation}\label{eq:children-neighbors}
    |\Lambda(v) - \Lambda(v')| \lesssim \rad(\cB(v)) \beta_f(5 \cB(v)),
  \end{equation}
  where for $v\in \cT$,
  \begin{equation}\label{eq:def-cbv}
    \cB(v) := B_{\mu 2^{-\gen(v)}\diam K}(\zeta(m_v)).
  \end{equation}
    
\end{prop}

Given such a map $\Lambda$ and any $i\ge 0$, we can define a piecewise-linear curve $g_i = g_{\Lambda, i}\from [0,1]\to \R^{2n}$ as follows. We label the elements of $\cT^i$ by $w_1,\dots, w_{k}$ so that the intervals $J_{w_1},J_{w_2},\dots$ are in increasing order, let $s_j = m_{w_j}$ for $j=1,\dots, k$, and let $s_0=0$, $s_{k+1}=1$. Let $g_i$ be the piecewise-linear curve with
\begin{equation}\label{eq:def-gi-Lambda}
  g_i(s_j)= \begin{cases}
    \Lambda(w_1) & j=0\\
    \Lambda(w_j) & j = 1,\dots, k\\
    \Lambda(w_k) & j = k+1
  \end{cases}
\end{equation}
such that $g_i$ is affine on each interval $[s_{j-1},s_j]$. When $i=0$, we have $\cT^0=\{v_0\}$, so $g_0$ is the constant curve with value $\Lambda(v_0)$.

In the rest of this section, we will explain the proof of Proposition~\ref{prop:construct-Lambda}; in the next section, we will use the bounds in Proposition~\ref{prop:construct-Lambda} to bound the symplectic area of the $g_i$'s.

The first step is to construct the dyadic patchwork. We suppose that $c_0 < \frac{1}{20}$. Then, by Corollary~\ref{cor:fibers-are-vertical}, $K$ is a $2$--vertical curve. In Appendix~\ref{app:vertical-curves}, we will prove the following.
\begin{lemma}\label{lem:patchwork-exist}
  For any $\lambda>0$, there is a $\mu>0$ such that if $K$ is a $\lambda$--vertical curve, then it admits a $\mu$--dyadic patchwork.
\end{lemma}
We take $\mu$ so that Lemma~\ref{lem:patchwork-exist} holds with $\lambda=2$, so that $K$ admits a $\mu$--dyadic patchwork.

Next, we construct affine approximations of $f$ and use them to construct $\Lambda$.
For a function $\phi \from \HH\to \R^{2n}$ and a ball $D=B_r(p)$, let $\widehat{L}_2(D)$ be the Hilbert norm
$$
\|g\|_{\widehat{L}_2(D)}^2 = \fint_{D} g^2 \ud \cL,
$$
corresponding to the uniform probability measure on $D$. Here $\cL$ is the Lebesgue $(2n+1)$-dimensional measure on $\HH\equiv \mathbb R^{2n+1}$. Let
$$
\langle g, h\rangle_D = \fint_{D} g h \ud \cL,
$$
denote the inner product on ${\widehat{L}_2(D)}$. We define  $\alpha_{\phi,D}\from \R^{2n}\to \R^{2n}$ to be the unique affine function such that $\alpha_{\phi,D}\circ \pi$ is the orthogonal projection of $\phi$ to $\Aff\subset \widehat{L}_2(D)$. That is,
$$
\|\phi - \alpha_{\phi,D}\circ \pi\|_{\widehat{L}_2(D)} \le \|\phi - \alpha\circ \pi\|_{\widehat{L}_2(D)},
$$
for every affine function $\alpha\from \R^{2n}\to \R^{2n}$, and thus
\begin{equation}\label{eqn:BetaPhiAlpha-approx}
\beta_\phi(D) = \frac{\|\phi - \alpha_{\phi,D}\circ\pi\|_{\widehat{L}_2(D)}}{\rad(D)}.
\end{equation}

We first note that if $\DfH$ is close to the identity near $D$, then $\alpha_{f,D}$ is invertible.
\begin{lemma}\label{lem:bilipschitz-affine-v2}
  Let $f\from \HH\to \R^{2n}$ be a $C^1_{\mathrm{H}}$ map, let $c>0$, and let $D=B_R(p)\subset \HH$ be a ball such that $\|\DfH_q - \id\|_{\mathrm{Op}}<c$ for all $q\in 5D$. For every $x\in D$,
  \begin{equation}\label{eq:ControlAlpha2}
    |(\alpha_{f,D}(\pi(x))-\pi(x))-(f(p)-\pi(p))|\lesssim cR.
  \end{equation}
  Consequently,
  \begin{equation}\label{eq:ControlAlpha3}
    |\alpha_{f,D}(\pi(p))-f(p)|\lesssim cR,
  \end{equation}
  and if $c$ is sufficiently small, then $\alpha_{f,D}$ is 2--bilipschitz, i.e.,
\begin{equation}\label{eqnBilipControl}
    \frac{1}{2}|x-y|\leq |\alpha_{f,D}(x)-\alpha_{f,D}(y)| \leq 2|x-y|,
  \end{equation}
  for every $x,y\in \mathbb R^{2n}$.
\end{lemma}
We defer the proof to Appendix~\ref{app:vertical-curves}.

For $v\in \cT$ and $\cB(v)$ defined as in Proposition~\ref{prop:construct-Lambda}, if $\alpha_{f,\cB(v)}$ is invertible, we define
\begin{equation}\label{eq:def-lambda}
  \Lambda(v) = \alpha_{f,\cB(v)}^{-1}(w)\in \R^{2n}.
\end{equation}
It remains to show that $\Lambda$ satisfies the desired bounds. 

We need two lemmas. First, we note the following equivalences.
\begin{lemma}\label{lem:affine-facts}
  Let $D$ be a ball in $\HH$. For any affine function $a\from \mathbb R^{2n}\subset \HH\to \R$ and any $\rho>0$,
  \begin{equation}\label{eqn:EquivalentNorms}
  \|a\|_{\widehat{L}_2(D)} \approx \|a\|_{L_\infty(D)}
  \end{equation}
  and
  \begin{equation}\label{eqn:EquivalentNorms2}
  \|a\|_{\widehat{L}_2(D)} \approx_\rho \|a\|_{\widehat{L}_2(\rho D)}.
  \end{equation}
\end{lemma}
\begin{proof}
  Let $\Aff$ be the subspace of affine functions and let $B_1=B_1(\zero)$ be the unit ball. Then $\|\cdot \|_{\widehat{L}_2(B_1)}$, $\|\cdot \|_{\widehat{L}_2(\rho B_1)}$, and $\|\cdot \|_{L_\infty(B_1)}$ are all norms on $\Aff$. Since $\Aff$ is finite-dimensional, all three norms are equivalent, with constant depending on $\rho$. 

  Let $g\from \HH\to \HH$ be the translation and dilation such that $g(B_1) = D$. For any $a\in \Aff$, we have
  $\|a\|_{L_\infty(D)} = \|a\circ g\|_{L_\infty(B_1)}$, $\|a\|_{\widehat{L}_2(D)} = \|a\circ g\|_{\widehat{L}_2(B_1)}$, and $\|a\|_{\widehat{L}_2(\rho D)} = \|a\circ g\|_{\widehat{L}_2(\rho B_1)}.$ Since $a\circ g \in \Aff$, this implies that $L_\infty(D)$, $\widehat{L}_2(D)$, and $\widehat{L}_2(\rho D)$ are also equivalent with a constant depending on $\rho$.
\end{proof}

Second, we note the following bound on $\beta_\phi$.
\begin{lemma}\label{lem:beta-compare}
  Let $\rho \in (0,1)$. Suppose that $D, D'\subset \HH$ are balls such that $D'\subset D$ and $\rad(D') \ge \rho \rad(D)$. For any $\phi\from \HH \to \R^{2n}$, we have
  \begin{equation}\label{eqn:CompareBeta}
  \beta_\phi(D') \lesssim_\rho \beta_\phi(D)
  \end{equation}
  and
  \begin{equation}\label{eqn:CompareAlphaPhiB}
  \|\alpha_{\phi,D} - \alpha_{\phi, D'}\|_{L_\infty(\pi(D))} \lesssim_\rho \beta_\phi(D) \rad(D).
  \end{equation}
\end{lemma}
\begin{proof}

  Since $\rho\rad(D)\leq \rad(D')$, we have
  \begin{equation}\label{eqn:L2EstimatesProofRad}
  \|g\|_{\widehat{L}_2(D')} = \cL(D')^{-\frac{1}{2}} \|g\|_{L_2(D')} \lesssim_{\rho} \cL(D)^{-\frac{1}{2}} \|g\|_{L_2(D)} = \|g\|_{\widehat{L}_2(D)},
  \end{equation}
  for any function $g$. In particular, using the latter inequality and \eqref{eqn:BetaPhiAlpha-approx}, we have
  \begin{multline}\label{eqn:InequalityRad}
  \rad(D')\beta_\phi(D') = \|\phi-\alpha_{\phi,D'}\circ \pi\|_{\widehat{L}_2(D')} 
  \leq \|\phi-\alpha_{\phi,D}\circ \pi\|_{\widehat{L}_2(D')}  \\
  \lesssim_{\rho} \|\phi-\alpha_{\phi,D}\circ \pi\|_{\widehat{L}_2(D)} 
  =\rad(D)\beta_\phi(D) \le \rho^{-1} \rad(D')\beta_\phi(D).
  \end{multline}
  We divide both sides by $\rad(D')$ to get \eqref{eqn:CompareBeta}.
    Finally, using \eqref{eqn:BetaPhiAlpha-approx}, \eqref{eqn:L2EstimatesProofRad}, and \eqref{eqn:InequalityRad}, we have
  \begin{equation}\label{eqn:ININ}
  \begin{aligned}
      \|\alpha_{\phi,D}\circ \pi - \alpha_{\phi,D'}\circ \pi\|_{\widehat{L}_2(D')} &\le \|\phi-\alpha_{\phi,D}\circ \pi\|_{\widehat{L}_2(D')}
      + \|\phi-\alpha_{\phi,D'}\circ \pi\|_{\widehat{L}_2(D')} \\
      &\lesssim_{\rho} \rad(D)\beta_\phi(D) + \rad(D')\beta_\phi(D')\\
      &\lesssim_{\rho} \rad(D)\beta_\phi(D).
      \end{aligned}
  \end{equation}
  Since $D'\subset D$ and $\rad(D')\geq\rho\rad(D)$ we have $D\subset 2\rho^{-1}D'$. Thus $\pi(D)\subset \pi(2\rho^{-1}D')=2\rho^{-1}\pi(D')$, and by Lemma~\ref{lem:affine-facts}, we have
  \[
  \begin{aligned}
  \|\alpha_{\phi,D} - \alpha_{\phi,D'}\|_{L_\infty(\pi(D))} &\le
    \|\alpha_{\phi,D} - \alpha_{\phi,D'}\|_{L_\infty(2\rho^{-1}\pi(D'))}\\
  &\lesssim_{\rho}
  \|\alpha_{\phi,D} - \alpha_{\phi,D'}\|_{L_\infty(\pi(D'))}  \\
  &\lesssim_{\rho} \|\alpha_{\phi,D}\circ \pi - \alpha_{\phi,D'}\circ \pi\|_{\widehat{L}_2(D')}.
  \end{aligned}
  \]
  Using the latter inequality and \eqref{eqn:ININ} we have
  \[
    \|\alpha_{\phi,D} - \alpha_{\phi,D'}\|_{L_\infty(\pi(D))} \lesssim_{\rho} \rad(D)\beta_\phi(D),
  \]
  as desired.
\end{proof}

Finally, we prove Proposition~\ref{prop:construct-Lambda}.
\begin{proof}[Proof of Proposition~\ref{prop:construct-Lambda}]
  Let $\mu$ be as in Lemma~\ref{lem:patchwork-exist}. Suppose that $c_0 < \frac{1}{20}$ and that $c_0$ is small enough that the conclusion of Lemma~\ref{lem:bilipschitz-affine-v2} holds. Let $\rho = 10\mu + 5$. 

  Let $f\in C^1_{\mathrm{H}}(\HH, \R^{2n})$ and let $B\subset \HH$ be a ball such that $\|\DfH_q - \id\|_{\mathrm{Op}} < c_0$ for all $q\in \rho B$. Let $R=\rad(B)$. Since $c_0 < \frac{1}{20}$, if $K$ is a connected subset of $B\cap f^{-1}(w)$ with $\diam(K)>0$, then $K$ is a $2$--vertical curve, so it admits a monotone increasing parameterization $\zeta\from [0,1]\to K$ and a $\mu$--dyadic patchwork $(\cT, \{I_v\}_{v\in \cT})$. 

  Let $\cB(v)$ be as in \eqref{eq:def-cbv}.
  In order to define $\Lambda$, we need to check that $\alpha_{f,\cB(v)}$ is invertible. Let $v\in \cT$. Then, since $\zeta(m_v)\in B$ and
  $$\rad(\cB(v)) = \mu 2^{-\gen(v)}\diam K \le \mu \diam K \le 2\mu R,$$
  we have $\cB(v) \subset (2\mu + 1) B$ and $5\cB(v) \subset \rho B$. Thus, by Lemma~\ref{lem:bilipschitz-affine-v2}, $\alpha_{f,\cB(v)}$ is invertible, so for all $v\in \cT$, we define $\Lambda(v) = \alpha_{f,\cB(v)}^{-1}(w)$ as in \eqref{eq:def-lambda}.

  We claim that $\Lambda$ is a set of approximating points for $\cT$. Let $v\in \cT$. By \eqref{eq:patch-diam}, $I_v \subset \cB(v)$. Let $q = \zeta(m_v)$ be the center of $\cB(v)$; then $f(q)=w$.  By \eqref{eq:ControlAlpha3},
  $$|\alpha_{f,\cB(v)}(\pi(q)) - w| = |\alpha_{f,\cB(v)}(\pi(q)) - f(q)| \lesssim c_0 \rad(\cB(v)).$$
  We suppose that $c_0$ is small enough that $|\alpha_{f,\cB(v)}(\pi(q)) - w| \le \frac{1}{2}\rad(\cB(v))$.
  Since $\alpha_{f,\cB(v)}$ is $2$--bilipschitz by \eqref{eqnBilipControl}, 
  \begin{equation}\label{eqn:ApproximatingRad}
  |\gamma(m_v) - \Lambda(v)| = |\pi(q) - \alpha_{f,\cB(v)}^{-1}(w)| \le \rad(\cB(v)),
  \end{equation}
  so $\Lambda$ is a set of approximating points for $\gamma$. Furthermore, $\Lambda(v) \in \pi\left(\cB(v)\right)$ for all $v$.

  Finally, we check \eqref{eq:children-neighbors}. Let $v'\in \cT$ be a weak neighbor of $v$ or a weak neighbor of a child of $v$. Let $D=\cB(v)$ and let $D'=\cB(v')$. Then $\frac{1}{2}\rad(D) \le \rad(D')\le \rad(D)$ and 
  $$D \cap D' \supset I_v\cap I_{v'}\ne \emptyset,$$
  so $D' \subset 3D$. 

  Since $\alpha_{f,D}$ is $2$--bilipschitz, and $\alpha_{f,D}(\Lambda(v))=\alpha_{f,D'}(\Lambda(v'))=w$,
  \begin{align*}
    |\Lambda(v) -\Lambda(v')| & \le 2|\alpha_{f,D}(\Lambda(v))-\alpha_{f,D}(\Lambda(v'))| \\
    & = 2 |\alpha_{f,D'}(\Lambda(v'))-\alpha_{f,D}(\Lambda(v'))|.
  \end{align*}
  Since $\Lambda(v') \in \pi(D') \subset \pi(5D)$,
  \begin{align*}
    |\Lambda(v) -\Lambda(v')| & \lesssim \|\alpha_{f,D'} - \alpha_{f,D}\|_{L_\infty(\pi(5D))} \\
    & \le \|\alpha_{f,D'} - \alpha_{f,5D}\|_{L_\infty(\pi(5D))} + \|\alpha_{f,5D} - \alpha_{f,D}\|_{L_\infty(\pi(5D))} \\
    & \lesssim \rad(D)\beta_f(5D),
  \end{align*}
  where the last inequality follows from Lemma~\ref{lem:beta-compare}.
\end{proof}

\subsection{Symplectic areas of fibers}\label{sec:symplectic-area}
Given a map $\zeta\from [0,1]\to \HH$ that positively parameterizes a $2$--vertical curve $K$, a $\mu$--dyadic patchwork $\cT$ for $K$, and a set $\Lambda$ of $C$--approximating points for $\cT$, \eqref{eq:def-gi-Lambda} gives a sequence of curves $g_i=g_{\Lambda,i}$ that converges uniformly to $\gamma = \pi\circ \zeta$. In this section, we give a criterion for $S(g_i)$ to converge to $S(\gamma)$. 

(Note that while the construction in Section~\ref{sec:approximations} relies on $\zeta$ parameterizing a fiber of a $C^1_{\mathrm{H}}$ map, our criterion applies to any $2$--vertical curve and any set of approximating points.)

For $i\ge 0$ and $v\in \cT^i$, let
\begin{equation}\label{eq:def-delta-lambda}
  \delta_\Lambda(v) := \diam(g_{i}(J_v) \cup g_{i+1}(J_v)),
\end{equation}
and let
\begin{equation}\label{eq:def-sigma-lambda}
  \sigma(\Lambda) = \sum_{v\in \cT} \delta_\Lambda(v)^2.
\end{equation}
For any curve $g$, let $\hat{g}$ be the closed curve obtained by connecting the endpoints of $g$ by a line segment. 
\begin{prop}\label{prop:vertical-curve-area}
  With notation as above, if $\sigma(\Lambda)<\infty$, then
  $S(\hat{\gamma})$ exists and 
  $S(\hat{\gamma}) = \lim_{i\to \infty} S(\hat{g}_i)$. Furthermore, 
  \begin{equation}\label{eq:sgamma-bound}
    |S(\hat{\gamma})| \lesssim_{\mu,C} \max\left\{\sigma(\Lambda), \sigma(\Lambda) \log \frac{(\diam K)^2}{\sigma(\Lambda)}\right\}.
  \end{equation}
  
  If, in addition, $g_i(0)=g_i(1)$ for all $i$, then $|S(\hat{\gamma})| \lesssim_{\mu,C} \sigma(\Lambda)$.
\end{prop}

We will prove Proposition~\ref{prop:vertical-curve-area} in Section~\ref{sec:area-formula}.
\section{Proof of Proposition~\ref{prop:density}}\label{sec:density-proof}

In this section, we use the results stated in Sections~\ref{sec:approximations} and \ref{sec:symplectic-area} to prove Proposition~\ref{prop:density}, i.e., if $f\from \HH\to \R^{2n}$ is a $C^1_{\mathrm{H}}$ map and
$$\mu_f(U) = \int_{\R^{2n}} \cH^2(f^{-1}(w)\cap U)\ud w,$$
then for almost every $p\in \HH$,
$$\lim_{r\to 0} \frac{\mu_f(B_r(p))}{\cH^{2n+2}(B_r(p))} = |\JfH(p)|.$$

Our strategy is to relate the quantity $\sigma(\Lambda)$ in Proposition~\ref{prop:vertical-curve-area} to the bounds from the Dorronsoro inequality (Theorem~\ref{thm:dorronsoro}). We first introduce a new function $T(\Gamma)$ that describes the $\beta$--numbers of $f$ on a subset $\Gamma\subset \HH$. 

For a metric space $X$ and $r>0$, an \emph{$r$--net} $N\subset X$ is a subset such that for any $p,q\in N$, $d(p,q) > r$. For every $i \in \Z$, we fix a maximal $2^{-i}$--net $\cN_i$ on $\HH$. By the maximality of $\cN_i$,
$$\bigcup_{p\in \cN_i} B_{2^{-i}}(p) = \HH.$$

Given a subset $\Gamma\subset \HH$ and $i\in \Z$, let
\begin{equation}\label{eq:def-Q-Gamma}
  Q_{\Gamma, i} = \{v\in \cN_i : \Gamma \cap B_{2^{-i}}(v) \ne \emptyset\}.
\end{equation}
Let $A = 10\mu + 1$. If $\diam(\Gamma)=0$, we define $T(\Gamma)=0$. Otherwise, 
\begin{equation}\label{eq:def-T}
  T(\Gamma) = T(\Gamma;f) := \sum_{i=-\lfloor \log_2(\diam \Gamma)\rfloor}^\infty \sum_{q\in Q_{\Gamma, i}} 2^{-2i} \beta_f(q, A 2^{-i})^2.
\end{equation}

We claim the following proposition holds. 
\begin{prop}\label{prop:levy-beta-fiber}
  Let $c_0$ and $\rho$ be as in Proposition~\ref{prop:construct-Lambda}. Let $B\subset \HH$ be a ball and let $f\from \HH\to \R^{2n}$ be a $C^1_{\mathrm{H}}$ map such that $\|\DfH_q - \id\|_{\mathrm{Op}} < c_0$ for all $q\in \rho B$.

  Let $w\in \R^{2n}$ be a point such that $T_{w,B} := T(B\cap f^{-1}(w)) < \infty$. Let $K\subset B\cap f^{-1}(w)$ be a connected subset with $\diam(K)>0$. Let $\zeta$ parameterize $K$, and let $\gamma= \pi\circ \zeta$. 
  Let $\hat\gamma$ be the closed curve obtained by connecting the endpoints of $\gamma$ by a line segment. Then the symplectic area $S(\gamma)$ exists and
  \begin{equation}\label{eq:levy-beta-fiber-small-T}
    |S(\hat{\gamma})| \lesssim \max\left\{T_{w,B}, T_{w,B} \log \frac{(\diam K)^2}{T_{w,B}}\right\}.
  \end{equation}
\end{prop}
An alternate version of \eqref{eq:levy-beta-fiber-small-T} will sometimes be convenient. Let 
$$\nu(t) = \begin{cases}
  0 & t = 0\\
  t (2 + \log(t^{-1})) & 0 < t \le 1\\
  t + 1 & t > 1.
\end{cases}$$
Then $\nu$ is an increasing concave function with $\nu(t) \approx \max\{t, t \log t^{-1}\}$ for $t>0$, so we can write \eqref{eq:levy-beta-fiber-small-T} as
\begin{equation}\label{eq:alt-levy-beta}
  |S(\hat{\gamma})| \lesssim \diam(K)^2 \nu\left(\frac{T_{w,B}}{\diam(K)^2}\right).
\end{equation}
\begin{proof}[Proof of Proposition~\ref{prop:levy-beta-fiber}]
  By Proposition~\ref{prop:construct-Lambda}, there is a set of approximating points $\Lambda$ satisfying
  \begin{equation}\label{eq:ch-ne-2}
    |\Lambda(v) - \Lambda(v')| \lesssim \rad(\cB(v)) \beta_f(5 \cB(v)),
  \end{equation}
  whenever $v$ and $v'$ are weak neighbors or $v'$ is a weak neighbor of a child of $v$.

  Let $g_i = g_{\Lambda,i}$ as in \eqref{eq:def-gi-Lambda} and let $\delta_\Lambda$ be as in Section~\ref{sec:symplectic-area}. Then $g_{i}(J_v) \cup g_{i+1}(J_v)$ is contained in the convex hull of the set
  $$\{\Lambda(v') : \text{$v'$ is a weak neighbor of $v$ or a weak neighbor of a child of $v$}\},$$
  so by \eqref{eq:ch-ne-2}, 
  $$\delta_\Lambda(v) = \diam\left(g_{i}(J_v) \cup g_{i+1}(J_v)\right) \lesssim \rad(\cB(v)) \beta_f(5 \cB(v)).$$

  Let $\sigma(\Lambda) = \sum_{v\in \cT} \delta_\Lambda(v)^2$ as in \eqref{eq:def-sigma-lambda} and let $\Gamma = B\cap f^{-1}(w)$ so that $T_{w,B}=T(\Gamma)$.
  By Proposition~\ref{prop:vertical-curve-area}, it suffices to show that $\sigma(\Lambda) \lesssim T(\Gamma)$.

  Let $i_0 = -\lfloor \log_2\diam(K)\rfloor$. For each $i\ge 0$ and $v\in \cT^i$, the maximality of $\cN_{i_0 + i}$ implies that there is a $N(v)\in \cN_{i_0+i}$ such that $\zeta(m_v) \in B_{2^{-i_0 - i}}(N(v))$; note that $N(v) \in Q_{\Gamma, i_0 + i}$. 
  Then 
  $$\rad(\cB(v)) = \mu 2^{-i}\diam K \le 2 \mu 2^{-i_0 - i},$$
  so
  \begin{equation}\label{eq:5cbv}
    5\cB(v) \subset B_{(10\mu + 1)2^{-i_0 - i}}(N(v)) \subset B_{A 2^{-i_0 - i}}(N(v)).
  \end{equation}
  Therefore, by Lemma~\ref{lem:beta-compare},
  $$\delta_\Lambda(v)\lesssim \rad(\cB(v)) \beta_f(5 \cB(v)) \lesssim 2^{-i_0 - i} \beta_f(N(v), A 2^{-i_0 - i}).$$
  
  Thus,
  $$
  \sigma(\Lambda) = \sum_{v\in \cT} \delta_\Lambda(v)^2 \lesssim \sum_{i=0}^\infty \sum_{v\in \cT^i} 2^{-2(i_0 + i)} \beta_f(N(v), A 2^{-i_0 - i})^2.
  $$
  Since $\cT^i$ is a finite set and $N(v) \in Q_{\Gamma, i_0 + i}$ for all $v\in \cT^i$, we can rearrange this sum to get
  \begin{equation}\label{eq:sigma-lambda-first}
    \sigma(\Lambda) \lesssim \sum_{i=0}^\infty \sum_{q\in Q_{\Gamma,i_0+i}} \#N^{-1}(q) \cdot 2^{-2(i_0 + i)} \beta_f(q, A 2^{-i_0 - i})^2 
  \end{equation}

  Suppose that $q\in Q_{\Gamma, i_0 + i}$ and let $N^{-1}(q) = \{v_1,\dots, v_k\}$. For each $j$, let $s_j$ be the lower endpoint of $I_{v_j}$; we order the $v_j$ so that $s_1 < \dots < s_k$. Then, using \eqref{eq:patch-diam},
  $$
  \diam(\zeta([s_j,s_{j+1}]))\ge \diam I_{v_j} \ge \frac{1}{2}\mu^{-1} 2^{- i_0 - i}.
  $$
  By \eqref{eq:patch-diam} and \eqref{eq:5cbv},
  $$I_{v_j} \subset 5\cB(v_j) \subset B_{A 2^{- i_0 - i}}(q),$$
  so $\diam \bigcup_j I_{v_j} \le 2 A 2^{- i_0 - i}$. By  Lemma~\ref{lem:chains}, there is an $N$ depending only on $\mu$, $A$, and $n$ such that $k\le N$. That is, $\#N^{-1}(q)\lesssim 1$. By \eqref{eq:sigma-lambda-first}, substituting $j=i_0+i$,
  $$\sigma(\Lambda) \lesssim \sum_{j=i_0}^\infty \sum_{q\in Q_{\Gamma, j}} 2^{-2j} \beta_f(q, A 2^{- j})^2 = T(\Gamma),$$
  as desired.
\end{proof}

We will use Proposition~\ref{prop:levy-beta-fiber} to prove the following lemmas. First, we bound the difference between the measure of $f^{-1}(w)\cap B$ and the measure of a vertical line in terms of $T$.
\begin{lemma} \label{lem:error-bound-pointwise}
  Let $c>0$. There are $\epsilon \in (0,1)$ and $\rho > 1$ with the following property. Let $B = B_R(p) \subset \HH$ be a ball and let $f\from \HH\to \R^{2n}$ be a $C^1_{\mathrm{H}}$ map such that $\|\DfH_q - \id\|_{\mathrm{Op}} < \epsilon$ for all $q\in \rho B$.

  Let $\hat{f}(q) = f(p) + \pi(p^{-1}q)$.
  Then for any $w\in \R^{2n}$, if $T_{w,2B}=T(f^{-1}(w)\cap 2B)<\infty$, then
  \begin{equation}\label{eq:error-equation-f-pi}
    |\cH^2(f^{-1}(w)\cap B) - \cH^2(\hat{f}^{-1}(w)\cap B)| \lesssim R^2 \nu\left(R^{-2} T_{w,2B}\right) + c R^2.
  \end{equation}
\end{lemma}
We will prove this lemma in Section~\ref{sec:proof-error-bound-pointwise}.

Second, we prove a bound on $T_{w,B}$. Since $\cH^2(f^{-1}(w)\cap B)$ can be infinite, $T_{w,B}$ can be infinite as well. Nonetheless, it is bounded on average. Since $T_{w,B}$ depends on a choice of dyadic patchwork, the function $w\mapsto T_{w,B}$ may not be measurable, so we recall that for $g\from\mathbb R^{2n}\to[0,+\infty]$, the \emph{upper integral} of $g$ is given by
\[
\int_{\mathbb R^{2n}}^*g :=\inf\left\{\int_{\mathbb R^{2n}}h:g\leq h,\,\,\text{$h$ measurable}\right\}.
\]

\begin{lemma}\label{lem:fubini}
  Let $f\from \HH \to \R^{2n}$ be a $C^1_{\mathrm{H}}$ map, let $L>0$, and let $B\subset \HH$ be a ball of radius $R$. There is a $c > 0$ such that if $f$ is $L$--Lipschitz on $5 B$, then 
  \begin{equation}\label{eq:fubini}
    \int^*_{\R^{2n}} T_{w,B}\ud w \lesssim_L \int_0^{c R} \int_{5B}\beta_f(v, \rho)^2 \ud v \frac{\mathrm d\rho}{\rho}.
  \end{equation}
\end{lemma}

The following version of Dorronsoro's inequality, due to \cite{FO-Dorronsoro}, implies that the right hand side of \eqref{eq:fubini} is finite.  
\begin{thm}[{\cite[Theorem 6.1]{FO-Dorronsoro}}]\label{thm:dorronsoro}
    Let $f\in L_2(\HH;\mathbb R^{2n})$ be a map such that $\|D_{\mathrm H}f\|\in L_2(\HH)$.
    Then 
    \[
    \int_{\HH} \int_0^\infty \beta_f(B_r(x))^2\frac{\ud r}{r}\ud x \lesssim \int_{\HH} \|D_{\mathrm H}f(x)\|^2\ud x.
    \]
\end{thm}
(Theorem 6.1 of \cite{FO-Dorronsoro} is stated for $g\in L_2(\HH;\R)$. Applying it to each component of $f$ gives Theorem~\ref{thm:dorronsoro}.)

This lets us bound $T(B_r(p) \cap f^{-1}(w))$ as $r$ goes to $0$.
\begin{lemma}\label{lem:error-density}
  Let $f\from \HH \to \R^{2n}$ be a $C^1_{\mathrm{H}}$ map. Then for almost every $p\in \HH$, 
  \begin{equation}\label{eq:error-density}
    \lim_{r\to 0} r^{-2n-2}\int^*_{\R^{2n}} T_{w,B_r(p)} \ud w = 0
  \end{equation}
\end{lemma}
We will prove Lemma~\ref{lem:fubini} and Lemma~\ref{lem:error-density} in Section~\ref{sec:proof-bounds-T}. Before that, we will use the lemmas to prove Proposition~\ref{prop:density}.
\begin{proof}[Proof of Proposition~\ref{prop:density}]
  Let $f\from \HH\to \R^{2n}$ be a $C^1_{\mathrm{H}}$ map. By Lemma~\ref{lem:error-density}, almost every $p\in \HH$ satisfies \eqref{eq:error-density}, so we suppose that $p$ satisfies \eqref{eq:error-density}. We consider two cases, depending on whether $\JfH(p) = 0$.

  First, if $\JfH(p)=0$, then by Theorem~\ref{thm:CoareaInequality},
  $$\frac{\mu_f(B_r(p))}{\cH^{2n+2}(B_r(p))} \le \frac{1}{\cH^{2n+2}(B_r(p))} \int_{B_r(p)} |\JfH(x)| \ud x \le \sup_{x\in B_r(p)} |\JfH(x)|.$$
  Since $\JfH$ is continuous, 
  $$\lim_{r\to 0} \frac{\mu_f(B_r(p))}{\cH^{2n+2}(B_r(p))} = 0 = |\JfH(p)|,$$
  as desired.

  Otherwise, suppose that $\JfH(p) \ne 0$.  Then $D_{\mathrm H}f_p$ is invertible; let $$\alpha(v) = D_{\mathrm H}f_p(v)$$ and 
  let $g(q) = \alpha^{-1}(f(q))$ for all $q$. Then $D_{\mathrm H}g_{p} = \id_{\R^{2n}}$ and $\beta_f(B) \approx_{\alpha} \beta_g(B)$ for any ball $B$, so $g^{-1}(v)=f^{-1}(\alpha(v))$ and
  $$T(B \cap g^{-1}(v);g) \approx_\alpha T(B \cap f^{-1}(\alpha(v));f)$$
  for any $v\in \R^{2n}$. By substituting $w=\alpha(v)$ in  \eqref{eq:error-density}, we get
  \begin{multline}\label{eq:g-error-density}
    \lim_{r\to 0} r^{-2n-2}\int^*_{\R^{2n}} T(B_r(p) \cap g^{-1}(v);g) \ud v   \\ \approx_\alpha \lim_{r\to 0} r^{-2n-2} |\JfH(p)|^{-1} \int^*_{\R^{2n}} T(B_r(p) \cap f^{-1}(w);f) \ud w = 0.
  \end{multline}
  Similarly, $\mu_g(U) = |\JfH(p)|^{-1} \mu_f(U)$ for any measurable $U\subset \HH$.
  
  Let $c>0$. Let $\epsilon, \rho>0$ be as in Lemma~\ref{lem:error-bound-pointwise}. If $r > 0$ is sufficiently small and $B = B_r(p)$, then $\|D_{\mathrm H}g_q - \id\|_{\mathrm{Op}} < \epsilon$ for all $q\in \rho B$. For $v\in \R^{2n}$, let
  $$\tau(v) = T(g^{-1}(v)\cap 2B).$$
  Lemma~\ref{lem:error-bound-pointwise} implies that if $\hat{g}(q) = g(p) + \pi(p^{-1}q)$ and $v\in \R^{2n}$, then
  \begin{equation}\label{eq:g-error-pointwise}
    |\cH^2(g^{-1}(v)\cap B) - \cH^2(\hat{g}^{-1}(v)\cap B)| \lesssim r^2 \nu(r^{-2} \tau(v)) + c r^2.
  \end{equation} 

  Let
  $$\eta_r = \int_{\R^{2n}} \cH^2(g^{-1}(v)\cap B) \ud v - \cH^{2n+2}(B) = \mu_g(B) - \cH^{2n+2}(B).$$
  By Lemma~\ref{lem:ControlF}, if $E = B_{2r}(g(p);\R^{2n})$, then $g(B)\subset E$. That is, if $v\not \in E$, then $\hat{g}^{-1}(v)\cap B = g^{-1}(v)\cap B = \emptyset$. By the Euclidean coarea formula,
  $$\cH^{2n+2}(B) = \int_{E} \cH^2(\hat{g}^{-1}(v)\cap B) \ud v,$$
  so 
  \begin{align*}
    |\eta_r|
    &=\left|  \int_{E} \left(\cH^2(g^{-1}(v)\cap B) - \cH^2(\hat{g}^{-1}(v)\cap B) \right)\ud v \right| \\ 
    &\stackrel{\eqref{eq:g-error-pointwise}}{\lesssim} \int^*_E \left(r^2 \nu(r^{-2} \tau(v)) + c r^2\right) \ud v \\ 
    &\lesssim \int^*_E r^2 \nu(r^{-2} \tau(v)) \ud v + c r^{2n+2}.
  \end{align*}

  Since $\nu$ is concave and increasing, by using Jensen's inequality we have
  $$r^{-2n-2} |\eta_r| \lesssim \frac{1}{\cL(E)} \int^*_E \nu(r^{-2} \tau(v)) \ud v + c \le \nu\left(\frac{1}{\cL(E)} \int^*_E r^{-2}\tau(v) \ud v \right) + c.$$
  By \eqref{eq:g-error-density},
  $$\lim_{r\to 0} r^{-2n} \int^*_E r^{-2}\tau(v) \ud v = 0,$$
  so
  $$\limsup_{r\to 0} r^{-2n-2} |\eta_r| \lesssim c.$$
  This holds for any $c>0$, so $\lim_{r\to 0} r^{-2n-2} \eta_r = 0$ and
  $$\lim_{r\to 0} \frac{\mu_g(B)}{\cH^{2n+2}(B)} = \lim_{r\to 0} 1 + \frac{\eta_r}{\cH^{2n+2}(B)} = 1.$$
  Therefore,
  $$\lim_{r\to 0} \frac{\mu_f(B_r(p))}{\cH^{2n+2}(B_r(p))} = |\JfH(p)| \lim_{r\to 0} \frac{\mu_g(B_r(p))}{\cH^{2n+2}(B_r(p))}  = |\JfH(p)|,$$
  as desired.
\end{proof}

\subsection{Estimating \texorpdfstring{$\cH^2(f^{-1}(w)\cap B)$}{H\^{}2(f\^{}-1(w))}}\label{sec:proof-error-bound-pointwise}
In this subsection, we prove Lemma~\ref{lem:error-bound-pointwise}, assuming Proposition~\ref{prop:vertical-curve-area}.
Let $c>0$, let $\epsilon>0$ be a small number to be chosen later, and let $\rho>0$ be a large number to be chosen later. Let $f \from \HH\to \R^{2n}$ be a $C^1_{\mathrm{H}}$ map. After composing with translations, we may suppose that $p=\zero$, $f(\zero) = \zero$, and $\|\DfH_q - \id\|_{\mathrm{Op}} < \epsilon$ for all $q\in \rho B$, where $B:=B_R(\zero)$.  

We suppose that $\rho > 10$, so that by Lemma~\ref{lem:ControlF}, $|f(q) - \pi(q)| \lesssim \epsilon R$ for all $q\in 2B$. It follows that for any $w\in \R^{2n}$, $f^{-1}(w)\cap B$ lies in a small cylinder around $\pi^{-1}(w)$. For any connected subset $K\subset f^{-1}(w)\cap B$, we can approximate $\cH^2(K)$ using Proposition~\ref{prop:levy-beta-fiber} and Theorem~\ref{thm:S-cH2}, so we will use the following lemma to approximate $f^{-1}(w)\cap B$ by connected vertical curves. By Lemma~\ref{lem:vertical-properties}, if $\Gamma$ is a closed connected vertical curve with $0<\diam(\Gamma)<\infty$, then we can label its endpoints $g_-$ and $g_+$ so that $z(g_-^{-1}g_+)>0$. For $q\in \HH$, let $V_q = q \langle Z\rangle \cap B$. We define
$$\height(\Gamma) = z(g_-^{-1}g_+).$$

\begin{lemma}\label{lem:sandwich}
  Let $B$, $f$, and $c$ be as above. If $\epsilon >0$ is sufficiently small, then for any $q\in B$ and $w=f(q)$, there are connected subsets $E_q$ and $F_q$ of $f^{-1}(w)$ with positive diameter such that
  $$E_q \subset B\cap f^{-1}(w) \subset F_q\subset 2B\cap f^{-1}(w)
  ,
  $$
  $|\height(F_q) - \height(V_q)| < c R^2$, and $|\height(E_q) - \height(V_q)| < cR^2$. 
\end{lemma} 
\begin{proof}
  Without loss of generality, we may suppose that $c < \frac{1}{16}$. After scaling and translating, it suffices to consider the case that $B=B_1(\zero)$.
  Let $\lambda\from \R^{2n}\to \R$,
  \begin{equation}\label{eq:def-lambda-fn}
    \lambda(v) = \frac{1}{4}\sqrt{1 - \min\{1,|v|^4\}}
  \end{equation}
  so that for any $r > 0$, 
  $$B_r(\zero) = \{(v,z)\in \HH : |z| < r^2 \lambda(r^{-1} v)\}.$$
  
  Let $0<\kappa < \frac{c}{16}$ be a small number such that $|\lambda(v_1) - \lambda(v_2)| < \frac{c}{8}$ for all $v_1,v_2 \in B_1(\zero;\R^{2n})$ with $|v_1 - v_2| < \kappa$.
  Let
  $$C = \{g\in \HH : |\pi(g) - \pi(q)| < \kappa\} = \pi^{-1}(B_\kappa(\pi(q);\R^{2n})).$$
  For $0 < s \le \frac{1}{2}$, let $C_s = C\cap z^{-1}([-s,s]) \subset 2B$. If $p\in C_{\frac{1}{2}}$, then $|\pi(p)| < 1 + \kappa$, so
  $$\|p\|_{\Kor} < \sqrt[4]{(1+\kappa)^4 + 16 z(p)^2} \le \sqrt[4]{2 + 4} < 2.$$
  That is, $C_{\frac{1}{2}}\subset 2B$.

  Suppose that $g, h \in C$. Then
  \[
  \begin{aligned}
  z(g^{-1}h) &= z(h) - z(g) + \frac{1}{2} \omega(-\pi(g), \pi(h)) \\
  &= z(h) - z(g) + \frac{1}{2} \omega(\pi(h)-\pi(g), \pi(g)).
  \end{aligned}
  \]
  We have $|\pi(h)-\pi(g)| < 2\kappa$ and $|\pi(g)|< 2$, so
  \begin{equation}\label{eq:gh-height}
    |z(g^{-1}h) - (z(h)-z(g))| < 2\kappa.
  \end{equation}

  By Lemma~\ref{lem:ControlF}, and Corollary~\ref{cor:fibers-are-vertical} if $\epsilon>0$ is small enough, then $\Gamma = f^{-1}(w)\cap 2B$ is a (possibly disconnected) $2$--vertical curve and
  \begin{equation}\label{eq:displace}
    |(f(g) - \pi(g)) - (f(q) - \pi(q))| = |f(g) - (w + \pi(g) - \pi(q))| < \kappa, 
  \end{equation}
  for all $g\in 2B$. 
  In particular, if $g\in \Gamma$, then $f(g) = f(q) = w$, so $|\pi(g) - \pi(q)| < \kappa$. Thus $\Gamma \subset C$.
  Furthermore, by Lemma~\ref{lem:vertical-properties}, the relation
  $$g \prec h \qquad \text{if $z(g^{-1}h) > 0$}$$
  is a total ordering on $\Gamma$.

  By Corollary~\ref{cor:fibers-are-vertical}, $\Gamma$ is a $1$--manifold which is relatively closed in $2B$. Let $\partial_-C_s = C\cap z^{-1}(-s)$ and $\partial_+C_s = C\cap z^{-1}(s)$ be the discs capping $C_s$. Then $\Gamma\subset C$, so
  \begin{equation}\label{eq:gamma-partialcs}
    \Gamma \cap \partial C_s \subset \partial_- C_s \cup \partial_+ C_s.
  \end{equation}

  We claim that if $s\in (\frac{c}{8},\frac{1}{2}]$, then exactly one component of $\Gamma \cap C_s$ connects $\partial_-C_s$ and $\partial_+C_s$. We will call this component $G_s$ and let $E_q = G_s$ and $F_q = G_{s'}$ for suitable $s$ and $s'$.

  Suppose that $s\in (\frac{c}{8},\frac{1}{2}]$. 
  For any $t$ with $|t| < \frac{1}{2}$, we have $C\cap z^{-1}(t)\subset 2B$. Then $\pi$ sends $\partial C\cap z^{-1}(t)$ to a sphere around $w$ with radius $\kappa$. By \eqref{eq:displace}, $\pi$ is $\kappa$--close to $f$ on $\partial C\cap z^{-1}(t)$, so $w$ is inside of $f(\partial C\cap z^{-1}(t))$. Therefore, $w\in f(C\cap z^{-1}(t))$ and thus $\Gamma \cap C\cap z^{-1}(t) \ne\emptyset$.

  Let $p_0 \in \Gamma \cap C \cap z^{-1}(0)$ and let $G_s$ be the connected component of $\Gamma \cap \closure(C_s)$ that contains $p_0$. By Corollary~\ref{cor:fibers-are-vertical}, $\Gamma$ is relatively closed in $2B$, so the compactness of $\closure(C_s)$ implies that $G_s$ is compact. Since $p_0\in \inter(C_s)$, $G_s$ is a nontrivial closed interval in $\Gamma$ connecting two points on $\partial C_s$. Let $g_-$ and $g_+$ be the endpoints of $G_s$.

  By Lemma~\ref{lem:vertical-properties}, $\prec$ is a total ordering on $G_s$, so we can label $g_-$ and $g_+$ so that $g_- \prec p_0 \prec g_+$. These lie in $\Gamma\cap \partial C_s$, so by \eqref{eq:gamma-partialcs}, $\{g_-,g_+\}\subset\partial_+C_s\cup \partial_-C_s$.

  For any $p_+\in \partial_+ C_s$, \eqref{eq:gh-height} implies that
  $$z(p_0^{-1}p_+) \ge (z(p_+) - z(p_0)) - 2\kappa = s - 2\kappa > 0,$$
  so $p_0 \prec p_+$. Similarly, $p_- \prec p_0$ for all $p_-\in \partial_- C_s$. Therefore, $g_-\in \partial_- C_s$ and $g_+ \in \partial_+ C_s$.
  By \eqref{eq:gh-height}, 
  \begin{equation}\label{eq:gs-height}
    |\height(G_s) - 2s| \le 2\kappa.
  \end{equation}
  
  Finally, if $H_s\ne G_s$ is another component of $\Gamma\cap C_s$ with endpoints $h_-\prec h_+$, then either $p_0 \prec h_- \prec h_+$ or $h_- \prec h_+\prec p_0$. In the first case, $\{h_-,h_+\}\subset \partial_+ C$, and in the second case, $\{h_-,h_+\}\subset \partial_- C$. 

  Now we define $E_q$ and $F_q$. Let $\lambda_0 = \lambda(\pi(q))$ so that $\height V_q=2\lambda_0$. We first consider $E_q$. If $\lambda_0 < \frac{c}{2}$, then we take $E_q$ to be the connected component of $f^{-1}(w)\cap B$ that contains $q$. Since $\Gamma$ is a $1$--manifold, $\diam (E_q)>0$. Further, $E_q \subset B\cap f^{-1}(w)\subset B\cap C$. By our choice of $\kappa$, we have $|z(g)| \le \lambda_0 + \frac{c}{8}$ for all $g\in B\cap C$, so if $e_-$ and $e_+$ are the endpoints of $E_q$, then by \eqref{eq:gh-height},
  $$0 < \height(E_q) \le z(e_+) - z(e_-) + 2\kappa \le 2\lambda_0 + \frac{c}{4} + 2\kappa < 2\lambda_0 + c,$$
  so $|\height(E_q) - \height(V_q)| = |\height(E_q) - 2\lambda_0| < c.$

  If $\lambda_0 \ge \frac{c}{2}$, we let $E_q=G_{\lambda_0 - \frac{c}{8}}$. Then $E_q$ is a nontrivial closed interval. Moreover, if $g\in C_{\lambda_0 - \frac{c}{8}}$, then $|\pi(g) - \pi(q)| < \kappa$, so
  $$|z(g)| \le \lambda_0 - \frac{c}{8} < \lambda(\pi(g)).$$
  That is, $g\in B$. 
  Thus $E_q\subset B \cap f^{-1}(w)$, and by \eqref{eq:gs-height}
  $$|\height(E_q) - 2\lambda_0| \le \frac{c}{4}+2\kappa < c,$$
  so $E_q$ satisfies the desired properties.

  We let $F_q = G_{\lambda_0 + \frac{c}{4}}$ and claim that this satisfies the desired properties.   By \eqref{eq:gs-height}, 
  $$|\height(F_q) - 2\lambda_0| \le \frac{c}{2}+2\kappa < c.$$
  Let $s=\lambda_0 + \frac{c}{4}$ and let $g_-$ and $g_+$ be the endpoints of $F_q$, so that $z(g_\pm) = \pm s$. Suppose that $p\in B\cap f^{-1}(w)$. Then $p \in B\cap C$, so our choice of $\kappa$ implies that $|z(p)| < \lambda_0 + \frac{c}{8}$, and by \eqref{eq:gh-height},
  $$z(p^{-1} g_+) \ge z(g_+) - z(p) - 2\kappa > \frac{c}{8} - 2\kappa > 0.$$
  That is, $p \prec g_+$. Similarly, $g_- \prec p$. By Lemma~\ref{lem:vertical-properties},
  $$F_q = \{p \in f^{-1}(w) \cap 2B : g_- \prec p \prec g_+\},$$
  so $p\in F_q$. Therefore, $f^{-1}(w)\cap B\subset F_q$, as desired.
\end{proof}

Given this lemma, we prove  Lemma~\ref{lem:error-bound-pointwise} as follows.
\begin{proof}[Proof of Lemma~\ref{lem:error-bound-pointwise}]
  After a translation, we can consider the case that $p=\zero$ and $f(\zero) = \zero$, so that $\hat{f}(q) = \pi(q)$ for all $q$.
  Let $\lambda$ be as in \eqref{eq:def-lambda-fn} and let $\lambda_R\from \R^{2n}\to \R$, $\lambda_R(v)=R^2\lambda(R^{-1}v)$, so that $B_R = B_R(\zero) = \{(v,z)\in \HH : |z| < \lambda_R(v)\}$. Let $\kappa>0$ be such that $|\lambda(v_1)-\lambda(v_2)|<\frac{c}{4}$ for all $v_1,v_2\in \R^{2n}$ with $|v_1-v_2|<\kappa$. Then 
  $$|\lambda_R(v_1)-\lambda_R(v_2)|<\frac{c}{4}R^2$$
  for all $v_1,v_2\in \R^{2n}$ with $|v_1-v_2|<\kappa R$. In particular, if $|v| \ge (1-\kappa)R$, then $\lambda_R(v) \le \frac{c}{4}R^2$.

  Let $\eta$ be the constant in Theorem \ref{thm:S-cH2}. We take $\rho > \max\{2C, 10\}$, where $C$ is as in Proposition~\ref{prop:levy-beta-fiber}. By Lemma~\ref{lem:ControlF}, this implies $|f(q) - \pi(q)| \lesssim \epsilon R$ for all $q\in 2B$. We choose $\epsilon <\eta$ so that  Lemma~\ref{lem:sandwich} holds and $|f(q) - \pi(q)| < \kappa R$ for all $q\in 2B$. 

  We consider two cases, depending on whether $w\in f(B)$.
  First, suppose that $w\not \in f(B)$. Since $|f(q) - \pi(q)| < \kappa R$ for all $q\in B$,
  $$B_{(1-\kappa)R}(\zero;\R^{2n})\subset f(B),$$
  so $|w| \ge (1 - \kappa)R$. By our choice of $\kappa$, we have $\height(\pi^{-1}(w) \cap B) = 2\lambda_R(w)\le \frac{c}{2} R^2$, so
  $$|\cH^2(f^{-1}(w)\cap B) - \cH^2(\pi^{-1}(w)\cap B)| = |\cH^2(\pi^{-1}(w)\cap B)| < c R^2.$$

  Now suppose that $w=f(q)$ for some $q\in B$.
  Let $E_q$ and $F_q$ be as in Lemma~\ref{lem:sandwich} and let $\zeta_E, \zeta_F\from [0,1] \to \HH$ be monotone increasing parameterizations of $E_q$ and $F_q$. Since $E_q \subset F_q \subset f^{-1}(w)\cap 2B$, it follows from the definition of $T$ and our choice of $w$ that $T(E_q) \le T(F_q) \le T_{w,2B}<\infty$.

  Let $\gamma_E = \pi\circ \zeta_E$. By Proposition~\ref{prop:levy-beta-fiber}, $S(\gamma_E)$ exists. By Theorem~\ref{thm:S-cH2}, $\cH^2(E_q) = \height(E_q) - S(\hat{\gamma}_E)$ and $\cH^2(V_q) = \height(V_q)$,
  where $\hat{\gamma}_E$ is the closed curve obtained by connecting the endpoints of $\gamma_E$ by a line segment. Therefore, 
  \[
  \begin{aligned}
  |\cH^2(E_q) - \cH^2(V_q)| &\le |\height(E_q) - \height(V_q)| + |S(\hat{\gamma}_E)| \\
  &\le c R^2 + |S(\hat{\gamma}_E)|.
  \end{aligned}
  \]
  By the remark after Proposition~\ref{prop:levy-beta-fiber},
  $$|S(\hat{\gamma}_{E})| \lesssim \diam(E_q)^2 \nu\left(T_{w,2B}\diam(E_q)^{-2}\right).$$
  We have $\diam(E_q)\le \diam(2B) = 4R$, so 
  $$|S(\hat{\gamma}_{E})| \lesssim 16 R^2  \nu\left(\frac{1}{16}R^{-2}T_{w,2B} \right) \lesssim  R^2 \nu\left(R^{-2}T_{w,2B} \right).$$
  (The first inequality uses the fact that $\nu$ is concave and $\nu(0)=0$ and the second uses the fact that $\nu$ is nondecreasing.)

  Let $\lambda_0=\lambda_R(\pi(q))$ so that $\mathcal{H}^2(V_q)=2\lambda_0$. Then
  \begin{equation}\label{eqn:Estimatee}
  |\cH^2(E_q) - 2\lambda_0| \lesssim c R^2 +  R^2 \nu(R^{-2}T_{w,2B}).
  \end{equation}
  The same argument and inequality hold with $E_q$ replaced by $F_q$, so since $E_q \subset f^{-1}(w)\cap B \subset F_q$, 
  $$|\cH^2(f^{-1}(w)\cap B) - 2\lambda_0| \lesssim cR^2 +  R^2\nu(R^{-2} T_{w,2B}).$$
  
  Since $|w-\pi(q)|=|f(q)-\pi(q)|<\kappa R$, we have $|\lambda_R(w)-\lambda_0|<\frac{c}{4}R^2$. Therefore,
  \begin{align*}
    |\cH^2(f^{-1}(w)\cap B) & - \cH^2(\pi^{-1}(w)\cap B)| \\
    & \le |\cH^2(f^{-1}(w)\cap B) - 2\lambda_0| + 2 |\lambda_R(w)-\lambda_0| \\
    & \lesssim c R^2 +  R^2\nu(R^{-2}T_{w,2B}).
  \end{align*}
  This proves \eqref{eq:error-equation-f-pi}, as desired.
\end{proof}

\subsection{Bounds on \texorpdfstring{$T$}{T}}\label{sec:proof-bounds-T}
In this section, we prove Lemma~\ref{lem:fubini} and Lemma~\ref{lem:error-density}. 

\begin{proof}[Proof of Lemma~\ref{lem:fubini}]
  Recall that $f\from \HH \to \R^{2n}$ is a $C^1_{\mathrm{H}}$ map, $B\subset \HH$ is a ball of radius $R$, and $f$ is $L$--Lipschitz on $5B$. We claim that 
  $$
  \int^*_{\R^{2n}} T(f^{-1}(w)\cap B)\ud w \lesssim_L \int_0^{cR} \int_{5B}\beta_f(v, \rho)^2 \ud v \frac{\mathrm d\rho}{\rho},
  $$
  for some constant $c>0$ which will be specified later.

  For $i\in \Z$, let $\cN_i$ be the maximal $2^{-i}$--net used to define $T$. For $\Gamma\subset \HH$, let $Q_{\Gamma,i}$ be as in \eqref{eq:def-Q-Gamma}. For $w\in \R^{2n}$ and $i\in \Z$, recall that
  $$
  Q_{w,i}:=Q_{f^{-1}(w)\cap B,i} = \{v\in \cN_i : f^{-1}(w)\cap B \cap B_{2^{-i}}(v) \ne \emptyset\}.
  $$
  Let $i_0 = -\lfloor \log_2(2R)\rfloor$ so that $R < 2^{-i_0} \le 2R$ and, by \eqref{eq:def-T}, 
  \begin{equation}\label{eq:twb-bound}
    T_{w,B} := T(f^{-1}(w)\cap B) \le \sum_{i=i_0}^\infty \sum_{v\in Q_{f^{-1}(w)\cap B,i}} 2^{-2i} \beta_f(v, A 2^{-i})^2.
  \end{equation}
  If $v\in Q_{f^{-1}(w)\cap B,i}$ for some $i\ge i_0$, then $B_{2R}(v)$ intersects $B$, so $v \in 3B$. Therefore,
  $Q_{w,i}\subset R_{w,i}$, where
  $$R_{w,i} = \{v \in \cN_i \cap 3B : w\in f(B_{2^{-i}}(v))\}.$$

  By the subadditivity of the upper integral,
  $$\int^*_{\R^{2n}} T_{w,B} \ud w
  \stackrel{\eqref{eq:twb-bound}}{\le} \sum_{i=i_0}^\infty \int^*_{\R^{2n}} \sum_{v\in R_{w,i}} 2^{-2i} \beta_f(v, A 2^{-i})^2 \ud w.$$
  If $v\in R_{w,i}$, then $v\in \cN_i\cap 3B$ and $w\in f(B_{2^{-i}}(v))$, so by Fubini's Theorem,
  \begin{align*}
    \int^*_{\R^{2n}} T_{w,B} \ud w & \le \sum_{i=i_0}^\infty \sum_{v \in \cN_i \cap 3B} \int_{f(B_{2^{-i}}(v))} 2^{-2i} \beta_f(v, A 2^{-i})^2 \ud w \\
    & = \sum_{i=i_0}^\infty \sum_{v \in \cN_i \cap 3B} \cL\left(f(B_{2^{-i}}(v))\right) 2^{-2i} \beta_f(v, A 2^{-i})^2,
  \end{align*}
  where $\cL$ is Lebesgue measure on $\R^{2n}$.
  Since $f$ is $L$--Lipschitz on $5B$, we have $\cL(f(B_{2^{-i}}(v))) \lesssim L^{2n} 2^{-2n i}$. Therefore,
  \begin{equation}\label{eq:fubini-interm}
    \int^*_{\R^{2n}} T_{w,B} \ud w \lesssim_L \sum_{i=i_0}^\infty \sum_{v\in \cN_i\cap 3B} 2^{-(2n+2)i} \beta_f(v, A 2^{-i})^2.
  \end{equation}

  For any $v\in \HH$ and any $i\ge i_0$, if $p \in B_{2^{-i}}(v)$ and $t\in [1,2]$, then Lemma~\ref{lem:beta-compare} implies
  $$\beta_f(v, A 2^{-i}) \lesssim \beta_f(p, 2 t A 2^{-i}),$$
  and $\cH^{2n+2}(B_{2^{-i}}(v)) \approx 2^{-(2n+2)i}$,
  so
  \[
  \begin{aligned}
    2^{-(2n+2)i} \beta_f(v, A 2^{-i})^2 & 
    \lesssim 
    \int_1^2 \int_{B_{2^{-i}}(v)} \beta_f(p, 2 t A 2^{-i})^2 \ud p \frac{\ud t}{t} \\
    &\approx \int_{2 A 2^{-i}}^{4 A 2^{-i}} \int_{B_{2^{-i}}(v)} \beta_f(p, r)^2 \ud p \frac{\ud r}{r}.
    \end{aligned}
  \]
  
  Let $\lambda = \ud p \otimes \frac{\ud r}{r}$ be the corresponding measure on $\HH\times \R$. Then 
  $$2^{-(2n+2)i} \beta_f(v, A 2^{-i})^2 \lesssim \int_{B_{2^{-i}}(v)\times (2 A 2^{-i}, 4 A 2^{-i}]} \beta_f^2 \ud \lambda.$$
  Applying this to \eqref{eq:fubini-interm}, we find
  $$\int^*_{\R^{2n}} T_{w,B} \ud w \lesssim_L \sum_{i=i_0}^\infty \sum_{v\in \cN_i\cap 3B} \int_{B_{2^{-i}}(v)\times (2 A 2^{-i}, 4 A 2^{-i}]} \beta_f^2 \ud \lambda.$$
  If $v\in \cN_i\cap 3B$ and $i\ge i_0$, then $B_{2^{-i}}(v)\times (2 A 2^{-i}, 4 A 2^{-i}]$ is contained in $5B\times [0, 8AR]$. Since $\cN_i$ is a $2^{-i}$--net and $\HH$ is a doubling metric space, these sets have bounded multiplicity, and
  $$\int^*_{\R^{2n}} T_{w,B} \ud w \lesssim \int_{5B \times [0, 8 AR]} \beta_f^2\ud \lambda,$$
  as desired.
\end{proof}

We use this bound and the bounds on $\beta_f$ from Theorem~\ref{thm:dorronsoro} to prove Lemma~\ref{lem:error-density}.

\begin{proof}[Proof of Lemma~\ref{lem:error-density}]
  Let $f\from \HH \to \R^{2n}$ be a $C^1_{\mathrm{H}}$ map, and let $B=B_R(q) \subset \HH$ be a ball. We claim that
  \begin{equation}\label{eq:error-density-conclusion}
    \lim_{r\to 0} r^{-2n-2}\int^*_{\R^{2n}} T(B_r(p) \cap f^{-1}(w)) \ud w = 0
  \end{equation}
  for almost every $p \in B$.

  Let $L = \Lip(f|_{10B})$. 
  For $p\in \R^{2n}$ and $r>0$, let
  $$\tau(p, r) = r^{-2n-2}\int^*_{\R^{2n}} T(B_r(p) \cap f^{-1}(w)) \ud w.$$
  Let $c>0$ be as in Lemma~\ref{lem:fubini} and let 
  $$I(p,r) = \int_0^{cr} \int_{B_{5r}(p)}\beta_f(v, \rho)^2 \ud v \frac{\mathrm d\rho}{\rho}.$$
  Theorem~\ref{thm:dorronsoro} implies that $I(p,r) < \infty$ for any $p$ and $r$. If $p\in B$ and $0<r<R$, then $f$ is $L$--Lipschitz on $B_{5r}(p)$, so by Lemma~\ref{lem:fubini},
  \begin{equation}\label{eq:fubini-app}
    r^{2n+2} \tau(p,r) \lesssim_L I(p,r).
  \end{equation}
  
  For $\epsilon > 0,$ let
  $$S_\epsilon = \{p \in B : \limsup_{r\to 0} \tau(p,r) > \epsilon\}.$$
  We claim that $S_\epsilon$ has measure zero. For any $0 < \delta < R$ and $p\in S_\epsilon$, there is a radius $R(p,\delta)$ such that $0 < R(p,\delta) < \delta$ and $\tau(p, R(p,\delta)) > \epsilon$.

  Let $0 < \delta < \frac{R}{5}$ be small and for each $p\in S_\epsilon$, let $D_p = B_{5 R(p,\delta)}(p)$. The $D_p$'s cover $S_\epsilon$, so by the Vitali Covering Lemma, there are countably many $p_i\in S_\epsilon$ such that the $D_{p_i}$'s are disjoint and $S_\epsilon \subset \bigcup_i 5D_{p_i}$. 
  Let $R_i =  R(p_i,\delta)$, so that
  $$\cH^{2n+2}(S_\epsilon)\lesssim \sum_i \cH^{2n+2}(5D_{p_i}) \lesssim \sum_i R_i^{2n+2}.$$
  By \eqref{eq:fubini-app} and our choice of $R_i$,
  $$\sum_i I(p_i,R_i) \gtrsim_L \sum_i R_i^{2n+2} \tau(p_i,R_i) > \epsilon \sum_i R_i^{2n+2} \gtrsim \epsilon \cH^{2n+2}(S_\epsilon).$$
  
  Since $5R_i < 5 \delta < R$ for all $i$, the $D_{p_i}$'s are disjoint subsets of $2B$. Therefore,
  \[
  \begin{aligned}
  \epsilon \cH^{2n+2}(S_\epsilon)&\lesssim_L \sum_i I(p_i,R_i) = \sum_i \int_0^{cR_i} \int_{D_{p_i}}\beta_f(v, \rho)^2 \ud v \frac{\mathrm d\rho}{\rho} \\
  &\le \int_0^{c\delta} \int_{2B}\beta_f(v, \rho)^2 \ud v \frac{\mathrm d\rho}{\rho}.
  \end{aligned}
  \]
  The integral on the right is bounded by $I(q,R)$, which is finite.
  Taking the limit as $\delta \to 0$, we have
  $$\epsilon \cH^{2n+2}(S_\epsilon)\lesssim \lim_{\delta \to 0} \int_0^{c\delta} \int_{2B}\beta_f(v, \rho)^2 \ud v \frac{\mathrm d\rho}{\rho} = 0$$
  by dominated convergence, so $\cH^{2n+2}(S_\epsilon)=0$. Thus, for almost every $p\in B$, we have $p\not \in \bigcup_{k} S_{\frac{1}{k}}$, so $\lim_{r\to 0} \tau(p,r) = 0$, as desired. 
\end{proof}

\section{Area formula based on approximations}\label{sec:area-formula}

In this section, we prove Proposition~\ref{prop:vertical-curve-area}. We adopt the notation of Section~\ref{sec:symplectic-area}, so that $K\subset \HH$ is a compact connected $2$--vertical curve with $\diam K>0$, $\zeta\from [0,1]\to K$ is a monotone increasing parameterization of $K$, $\gamma = \pi\circ \zeta$, and $\cT$ is a $\mu$--dyadic patchwork for $K$.

Let $C>0$ and let $\Lambda\from \cT \to \R^{2n}$ be a set of $C$--approximating points for $\cT$ as in \eqref{eq:approx-points}. Let $g_i = g_{\Lambda,i}$ be as in \eqref{eq:def-gi-Lambda}, and let $\delta_\Lambda$ and $\sigma(\Lambda)$ be as in \eqref{eq:def-delta-lambda} and \eqref{eq:def-sigma-lambda}. We suppose that $\sigma(\Lambda)<\infty$. Throughout this section, the implicit constants in $\lesssim, \gtrsim$, and $\approx$ will depend on $\mu$, $C$, and $n$.

We first establish some notation and construct a map $G\from [0,1]\times [0,1] \to \R^{2n}$ that interpolates between the $g_i$'s.
For points $p_1,\dots, p_k$ let $\overline{p_1,\dots,p_k}$ denote the curve that connects $p_1,\dots,p_k$ by line segments. If $h_1$ and $h_2$ are paths with $h_1(1)=h_2(0)$, let $h_1\diamond h_2$ be the concatenation of $h_1$ and $h_2$. For a Lipschitz curve $\alpha$, let $\fcl{\alpha}$ represent the fundamental class of $\alpha$, viewed as a Lipschitz chain; likewise, for a polygon $U\subset \R^2$, let $\fcl{U}$ be its fundamental class. For a locally Lipschitz map $g:X\to Y$, we let $g_\sharp$ be the corresponding pushforward map, which takes Lipschitz chains on $X$ to Lipschitz chains on $Y$. The symplectic area $S(\alpha)$ is defined for any Lipschitz curve $\alpha$, and we extend it to the set of Lipschitz $1$--chains in $\R^{2n}$ by linearity.

Let $G\from [0,1]\times [0,1] \to \R^{2n}$ be the map such that for any $i\ge 0$ and any $(t, u) \in [0,1]\times [2^{-i-1}, 2^{-i}]$, 
\begin{equation}\label{eq:def-G}
  G(t, u) = g_{i+1}(t) + \frac{u - 2^{-i-1}}{2^{-i} - 2^{-i-1}} (g_i(t) - g_{i+1}(t))
\end{equation}
and $G(t,0)=\gamma(t)$ for all $t$.
That is, $G(t, 2^{-i}) = g_i(t)$ for all $i\ge 0$, and $G$ interpolates linearly between $g_i$ and $g_{i+1}$ on $[0,1]\times [2^{-i-1}, 2^{-i}]$.

We note the following bounds. 
\begin{lemma}\label{lem:nbhd-diam}
  For every $v\in \cT^i$,
  \begin{equation}\label{eq:delta-lambda-bound}
    \delta_\Lambda(v) \lesssim 2^{-i}\diam K.
  \end{equation}
  For every $(t,u) \in [0,1]\times [0,1]$,
  \begin{equation}\label{eq:G-gamma-dist}
    |G(t,u) - \gamma(t)| \lesssim u \diam(K),
  \end{equation}
  so $G$ is continuous.
\end{lemma}
\begin{proof}
  Let $i\ge 0$ and $v\in \cT^i$. If $v'$ is a weak neighbor of $v$, then $J_v\cap J_{v'}\ne \emptyset$, so \eqref{eq:patch-diam} and \eqref{eq:approx-points} imply that
  \begin{equation}\label{eqn:Lambdavv'}
  \begin{aligned}
  |\Lambda(v) - \Lambda(v')| &\le |\Lambda(v) - \gamma(m_v)| + |\gamma(m_v) - \gamma(m_{v'})| + |\Lambda(v') - \gamma(m_{v'})|\\
  &\le 2 C 2^{-i} \diam(K) + \diam(\gamma(J_v)) + \diam(\gamma(J_{v'})) \\
  &\lesssim 2^{-i} \diam(K).
  \end{aligned}
  \end{equation}

  For any $t\in [0,1]$, there are $v,v'\in \cT^i$ such that $v$ and $v'$ are weak neighbors and $g_i(t)$ lies on $\overline{\Lambda(v), \Lambda(v')} = \overline{g_i(m_v), g_i(m_{v'})}$. Thus, by \eqref{eq:patch-diam}, \eqref{eq:approx-points}, and \eqref{eqn:Lambdavv'},
  \begin{equation}\label{eq:gamma-G-intermediate}
  \begin{aligned}
    |\gamma(t) - g_i(t)| & \lesssim |\gamma(t) - \gamma(m_v)| + |\gamma(m_v) - g_i(m_v)| + |g_i(m_v) - g_i(t)|\\
    & \le \diam(\gamma(J_v)) + C 2^{-i}\diam(K) + |\Lambda(v) - \Lambda(v')| \\
    & \lesssim 2^{-i}\diam(K).
    \end{aligned}
  \end{equation}

  This proves \eqref{eq:G-gamma-dist} when $u=2^{-i}$; it is trivial when $u=0$. If $u>0$, then $u\in [2^{-i},2^{-i+1}]$ for some $i$, and $G(t,u)$ is on the line segment from $G(t,2^{-i})$ to $G(t,2^{-i+1})$. Therefore, by the triangle inequality and \eqref{eq:gamma-G-intermediate},
  $$
  |\gamma(t) - G(t,u)| \lesssim 2^{-i+1}\diam(K) \lesssim u\diam(K).
  $$
  That is, \eqref{eq:G-gamma-dist} holds for every $(t,u)\in [0,1]\times [0,1]$.

  Finally, we prove \eqref{eq:delta-lambda-bound}. Let $v\in \cT^i$. By \eqref{eq:gamma-G-intermediate} and \eqref{eq:patch-diam}, for any $j\ge i$ and $s\in J_v$,
  \begin{multline*}
    |\gamma(m_v) - g_j(s)| \le |\gamma(m_v) - \gamma(s)| + |\gamma(s) - g_j(s)| \\ \lesssim \diam(\gamma(J_v)) + 2^{-j}\diam(K) \lesssim 2^{-i} \diam(K).
  \end{multline*}
  In particular, if $y \in g_{i}(J_v) \cup g_{i+1}(J_v)$, then $|\gamma(m_v) - y| \lesssim 2^{-i}\diam(K)$. It follows that
  $$\delta_\Lambda(v) = \diam(g_{i}(J_v) \cup g_{i+1}(J_v)) \lesssim 2^{-i}\diam(K),$$
  proving \eqref{eq:delta-lambda-bound}, as desired.
\end{proof}

Equation \eqref{eq:G-gamma-dist} with $u=2^{-i}$ implies that $|g_i(t) - \gamma(t)| \lesssim 2^{-i}\diam(K)$, so $g_i$ converges uniformly to $\gamma$ as $i\to \infty$. Each curve $g_i$ has finite length, so $S(g_i)$ exists for all $i$, but $\ell(g_i)$ may be unbounded. We will show that even if $\ell(g_i)$ is unbounded, $S(g_i)$ converges to $S(\gamma)$ as $i\to \infty$. We will need several lemmas, which we will state here and prove in the following subsections.

For every $i\ge 0$ and $v\in \cT^i$, let $R_v$ be the rectangle
\begin{equation}\label{eq:def-rv}
  R_v=J_v\times [2^{-i-1}, 2^{-i}] \subset [0,1]^2
\end{equation}
and let $\theta_v := G|_{\partial R_v}$. This is a piecewise-linear closed curve constructed by connecting $g_i|_{J_v}$ and $g_{i+1}|_{J_v}$ by line segments between their endpoints.

We use the $\theta_v$'s to calculate the limit of $S(g_i)$ as $i\to \infty$. Recall that for any curve $g$, $\hat{g}$ denotes the closed curve obtained by connecting the endpoints of $g$ by a line segment.
\begin{lemma}\label{lem:sgi-converge}
  For any $v\in \cT$, 
  \begin{equation}\label{eq:ell-theta-v}
    \ell(\theta_v) \lesssim \delta_\Lambda(v) \lesssim 2^{-\gen(v)}\diam(K).
  \end{equation}
  Therefore,
  $$\sum_{v\in \cT} |S(\theta_v)| \lesssim \sum_{v\in \cT} \delta_\Lambda(v)^2 = \sigma(\Lambda)<\infty.$$
  It follows that $\sum_{v\in \cT} S(\theta_v)$ converges absolutely. Let
  \begin{equation}\label{eq:def-psi}
    \psi = (G\circ \overline{(1,0),(1,1)}) \diamond (G\circ \overline{(0,1),(0,0)}).
  \end{equation}
  Since $G$ is constant on the top edge of $[0,1]^2$, this is a curve from $G(1,0) = \gamma(1)$ to $G(0,0) = \gamma(0)$.
  Let $\Phi = \sum_{v\in \cT} S(\theta_v)$. 
  Then $\ell(\psi)<\infty$ and
  $$\lim_{i\to \infty} S(g_i) = \Phi - S(\psi).$$
\end{lemma}
\begin{proof}
  For any $v\in \cT^i$, $\theta_v$ is a piecewise-linear curve made of up to $\mu+5$ segments (two for the top edge of $R_v$, two for the sides, up to $\mu+1$ for the bottom edge). The endpoints of each segment lie on $g_i(J_v)$ or $g_{i+1}(J_v)$, so each segment has length at most $\delta_\Lambda(v)$. Thus, $\ell(\theta_v) \le (\mu + 5)\delta_\Lambda(v)$. With \eqref{eq:delta-lambda-bound}, this proves \eqref{eq:ell-theta-v}. 

  By Lemma~\ref{lem:symplectic-area-props-v2}.(3),
  $$
  |S(\theta_v)| \leq \ell(\theta_v)^2 \lesssim \delta_\Lambda(v)^2,
  $$
  so $\sum_{v\in \cT} |S(\theta_v)| \lesssim \sigma(\Lambda)<\infty$. Therefore, $\sum_{v\in \cT}S(\theta_v)$ converges absolutely.

  For any $i>0$, $G$ sends $\partial([0,1]\times [2^{-i},1])$ to the closed curve
  $$h_i = (G\circ \overline{(1,2^{-i}),(1,1)}) \diamond (G\circ \overline{(0,1),(0,2^{-i})}) \diamond g_i.$$
  (We omit the segment $G\circ \overline{(1,1),(0,1)}$ because $G$ is constant on this segment.)
  Since   
  $$[0,1]\times [2^{-i},1] = \bigcup_{v\in \cT^{<i}}R_v,$$
  we have
  $$
  \fcl{h_i} = \partial G_\sharp\Big(\big\llbracket[0,1]\times [2^{-i},1]\big\rrbracket\Big) = \sum_{v\in \cT^{<i}} \partial G_\sharp(\fcl{R_v}) = \sum_{v\in \cT^{<i}}\fcl{\theta_v}.
  $$
  By the linearity of $S$,
  $$
  S(h_i) = \sum_{v\in \cT^{<i}}S(\theta_v).
  $$

  This lets us calculate $\lim S(g_i)$. Let
  $$\psi_i = \left(G\circ \overline{(1,2^{-i}),(1,1)}\right) \diamond \left(G\circ \overline{(0,1),(0,2^{-i})}\right),$$
  so that $h_i = \psi_i \diamond g_i$ and thus $S(h_i) = S(\psi_i) + S(g_i)$ by Lemma \ref{lem:symplectic-area-props-v2}.(2).
  By \eqref{eq:G-gamma-dist}, for all $i$ and for $a=0,1$
  $$\ell(G \circ \overline{(a,2^{-i-1}),(a,2^{-i})}) = |g_{i+1}(a) - g_i(a)| \lesssim 2^{-i}\diam K,$$
  so
  $$\ell(\psi_i) \lesssim \sum_{j=0}^{i-1} 2^{-j}\diam K \lesssim \diam K.$$
  That is, $\ell(\psi_i)$ is uniformly bounded. Since $\psi_i\to \psi$ uniformly and $\ell(\psi)<\infty$, we can use Lemma~\ref{lem:symplectic-area-props-v2} to show that
  $$
  \lim_{i\to \infty} S(g_i) = \lim_{i\to \infty} S(h_i) - \lim_{i\to \infty} S(\psi_i) = \sum_{v\in \cT}S(\theta_v) - S(\psi),
  $$
  as desired.
\end{proof}

This shows that $S(g_i)$ converges as $i\to \infty$. It remains to show that $S(\gamma) = \lim_{i\to \infty} S(g_i)$. For any partition
$P=\{t_0,\dots,t_k\}$ of $[0,1]$ with $0=t_0<\dots<t_k=1$, let $\gamma_P$ be the piecewise-linear curve connecting $\gamma(t_0),\dots, \gamma(t_{k}) \in \R^{2n}$, parameterized so that $\gamma_P(t_i) = \gamma(t_i)$ for all $i$. 
We will use $G$ to construct a curve $\alpha_P\from [0,1]\to \R^{2n}$ such that $\alpha_P(t_j) = \gamma(t_j)$ for all $j$, then prove Proposition~\ref{prop:vertical-curve-area} by comparing $\alpha_P$ and $\gamma_P$.

Let $v_0$ denote the root of $\cT$ and let
\begin{equation}\label{eq:def-cU}
  \cU = \cU_P = \{v\in \cT : J_{v}\cap P\ne \emptyset\}.
\end{equation}
For any subset $\cS\subset \cT$, let
\begin{equation}\label{eq:def-R}
  R(\cS) = \bigcup_{v\in \cS} R_v \subset [0,1]^2.
\end{equation}
(See Figure~\ref{fig:coherent} for an example.)

In Section~\ref{sec:proof-arch-lengths}, we will prove the following lemma.
\begin{lemma}\label{lem:boundary-arches}
  The set $R(\cU_P)$ is a simply-connected noncompact planar manifold, homeomorphic to a closed disc with $k+1$ boundary points removed. For each $i=1,\dots,k$, $R(\cU_P)$ has a boundary component connecting $(t_{i-1},0)$ to $(t_{i},0)$, and this component can be parameterized by a map $u=(\alpha_1,\alpha_2)\from [0,1]\to \partial R(\cU_P)$ such that $\alpha_1$ is nondecreasing and $\alpha_2(x)>0$ for all $x\in (0,1)$.
\end{lemma}
Let $A_j$ be the boundary component connecting $(t_{j-1},0)$ and $(t_j,0)$. Then $R(\cU_P)$ is bounded by the $A_j$'s and the left, top, and right sides of $[0,1]^2$.

\begin{figure}
  \begin{center}
    \includegraphics[width=.7\textwidth]{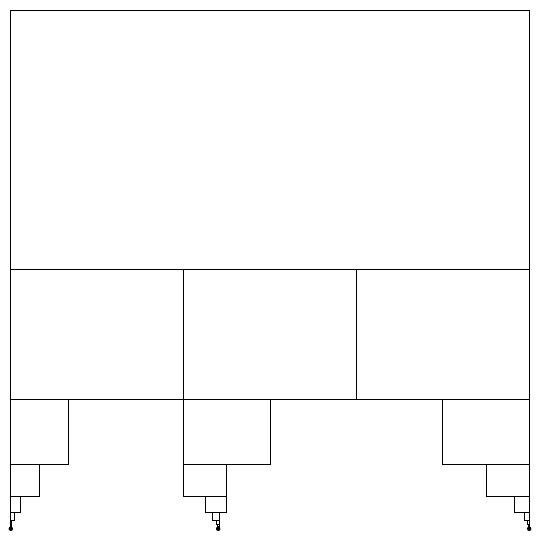}
    \caption{\label{fig:coherent}
      Example of $R(\cU_P)$ for a three-point partition. The dots mark the points in the partition, and the lower boundary of $R(\cU_P)$ can be divided into two arches connecting the dots.}
  \end{center}
\end{figure}

For each $j$, let $u_j\from [0,1]\to [0,1]^2$ parameterize $A_j$, such that $u_j(0) = (t_{j-1},0)$ and $u_j(1) = (t_{j},0)$. Let $a_j = G\circ u_j$ and let $\alpha_P = a_1\diamond \dots \diamond a_k$.

\begin{lemma}\label{lem:arch-lengths}
  For each $j$, let
  \begin{equation}\label{eq:def-cwj}
    \cW_j = \{w\in \cT : J_w\cap \{t_{j-1},t_j\} \ne \emptyset \text{ and } \diam I_w \le 2 \mu^2 \diam \zeta([t_{j-1},t_j])\}.
  \end{equation}
  Then 
  \begin{equation}\label{eq:boundary-graph}
    A_j\setminus \{(t_{j-1},0), (t_{j},0)\} \subset \bigcup_{w\in \cW_j} \partial R_w
  \end{equation}
  and 
  \begin{equation}\label{eq:alpha-length}
    \ell(a_j) \lesssim \sum_{w\in \cW_j} \delta_\Lambda(w) \lesssim  d(\zeta(t_{j-1}),\zeta(t_j)).
  \end{equation}
  Furthermore, the $\cW_j$'s satisfy the following.
  \begin{enumerate}
  \item\label{it:W-multiplicity}
    There is a $c = c(\mu,n) >0$ such that for any $v\in \cT$, $|\{j : v\in \cW_j\}|\le c$.
  \item\label{it:arch-count}
    For all $i\ge 0$ and all $j$, $|\cW_j \cap \cT^i| \le 4$.
  \end{enumerate}
\end{lemma}

We defer the proof of Lemma~\ref{lem:arch-lengths} to Section~\ref{sec:proof-arch-lengths}.

For any partition $P$, let
\begin{equation}\label{eqn:RelationPhiPAlphaP}
\phi_P = G|_{\partial R(\cU_P)} = \alpha_P \diamond \psi,
\end{equation}
where $\psi = (G\circ \overline{(1,0),(1,1)}) \diamond (G\circ \overline{(0,1),(0,0)})$ as in Lemma~\ref{lem:sgi-converge}.
By Lemma~\ref{lem:sgi-converge} and Lemma~\ref{lem:arch-lengths}, $\ell(\phi_P) = \ell(\psi) + \ell(\alpha_P) < \infty$. Thus, up to reparameterization, $\phi_P$ is a Lipschitz curve. Therefore, $S(\phi_P)$ exists and is finite by Lemma~\ref{lem:symplectic-consistent} and Lemma~\ref{lem:symplectic-area-props-v2}.(1). The following lemma is a consequence of the absolute convergence of $\sum_{v\in \cT} S(\theta_v)=\Phi$; see Section~\ref{sec:proof-dyadic}. 
\begin{lemma}\label{lem:lim-dyadic}
  For any partition $P$, we have $|S(\phi_P)| \lesssim \sigma(\Lambda)$ and
  \begin{equation}\label{eq:lim-dyadic}
    \lim_{\mesh(P)\to 0} S(\phi_P) = \Phi.
  \end{equation}
  Therefore, by \eqref{eqn:RelationPhiPAlphaP} and Lemma \ref{lem:symplectic-area-props-v2}.(2),
  $$\lim_{\mesh(P)\to 0} S(\alpha_P) = \lim_{\mesh(P)\to 0} (S(\phi_P) - S(\psi)) = \Phi - S(\psi).$$
\end{lemma}
Next, we will show that $S(\alpha_P) - S(\gamma_P) \to 0$ as $\mesh(P)\to 0$. Recall that $\alpha_P = a_1\diamond \dots \diamond a_k$, where $a_j$ is the arch connecting $\gamma(t_{j-1})$ to $\gamma(t_j)$. Then $\hat{a}_j = a_j \diamond \overline{\gamma(t_j), \gamma(t_{j-1})}$ and
$$\fcl{\hat{a}_j} = \fcl{a_j} - \fcl{\gamma_P|_{[t_{j-1},t_j]}}.$$
Therefore,
\begin{equation}\label{eq:alpha-P-gamma-P}
\fcl{\alpha_P} - \fcl{\gamma_P} = \sum_j \fcl{\hat{a}_j}.
\end{equation}
We claim that $\sum_j S(\hat{a}_j) \to 0$.

By Lemma~\ref{lem:arch-lengths}, 
$$\ell(a_j) \lesssim d(\zeta(t_{j-1}),\zeta(t_j)),$$ so by Lemma~\ref{lem:symplectic-area-props-v2}(3) and Lemma~\ref{lem:vertical-diameter-v2},
\begin{equation}\label{eq:base-alpha-bound}
  S(\hat{a}_j) \leq \ell(a_j)^2 \lesssim d(\zeta(t_{j-1}),\zeta(t_j))^2 \lesssim z(\zeta(t_{j-1})^{-1} \zeta(t_j)),
\end{equation}
for all $j$. To show that $\sum_j S(\hat{a}_j) \to 0$ as $\mesh P\to 0$, we need to strengthen \eqref{eq:base-alpha-bound} in two ways. We will show the following two lemmas. For each $j$, let $\Delta z_j = z(\zeta(t_{j-1})^{-1} \zeta(t_j))$.

\begin{lemma}\label{lem:arch-area}
  For any $\epsilon>0$, there is a $\rho$ such that if $\mesh P<\rho$, then
  \begin{equation}\label{eq:arch-area}
    \sum_j \ell(a_j)^2 \le \epsilon + \epsilon \sum_j \Delta z_j.
  \end{equation}
\end{lemma}
One can compare this inequality to the bound $\ell(a_j)\lesssim \sqrt{\Delta z_j}$ from Lemma~\ref{lem:arch-lengths}. The stronger inequality comes from using the assumption that $\sigma(\Lambda)<\infty$.

\begin{lemma}\label{lem:z-sum}
  There is a $\tau >0$ depending on $K$ such that for any sufficiently fine partition $P$, 
  \begin{equation}\label{eq:z-sum}
    \sum_j \Delta z_j < \tau.
  \end{equation}
\end{lemma}

Like the previous lemma, this lemma relies on the assumption that $\sigma(\Lambda)<\infty$. Without this assumption, $\zeta$ could have Hausdorff dimension greater than $2$, as in \cite{AY-Vertical}, and the sum in \eqref{eq:z-sum} could be arbitrarily large. 

These lemmas imply Proposition~\ref{prop:vertical-curve-area}.
\begin{proof}[Proof of Proposition~\ref{prop:vertical-curve-area}]
  Throughout this proof, we use the bounds in Lemma~\ref{lem:symplectic-area-props-v2} without comment.
  We first show that $\lim_{\mesh P\to 0} S(\gamma_P) = \Phi - S(\psi)$, where $\psi = (G\circ \overline{(1,0),(1,1)}) \diamond (G\circ \overline{(0,1),(0,0)})$.
  By \eqref{eq:alpha-P-gamma-P}, for any $P$,
  \begin{equation}\label{eq:vca-pf-s-gamma}
    S(\gamma_P) = S(\alpha_P) - \sum_j S(\hat{a}_j).
  \end{equation}

  Let $\tau$ be as in Lemma~\ref{lem:z-sum}. By Lemma~\ref{lem:arch-area} and Lemma~\ref{lem:z-sum}, for any $\epsilon>0$, if $\mesh P$ is sufficiently small,
  $$\sum_j S(\hat{a}_j) \lesssim \sum_j \ell(a_j)^2 \le \epsilon + \epsilon \sum_j \Delta z_j \le \epsilon + \epsilon \tau.$$
  Therefore, 
  $$\lim_{\mesh P\to 0} \sum_j S(\hat{a}_j) = 0,$$
  and by Lemma~\ref{lem:lim-dyadic},
  $$
  \lim_{\mesh P\to 0} S(\gamma_P) = \lim_{\mesh P\to 0} S(\alpha_P) - \lim_{\mesh P\to 0} \sum_j S(\hat{a}_j) = \Phi - S(\psi).
  $$
  Therefore, $S(\gamma) = \Phi - S(\psi)$.

  Since $\hat{\gamma} = \gamma \diamond \overline{\gamma(1),\gamma(0)}$ and $\hat{\psi} = \psi \diamond \overline{\gamma(0),\gamma(1)}$,
  \begin{equation}\label{eq:s-gamma-s-psi}
    S(\hat{\gamma}) = \Phi - S(\psi) + S(\overline{\gamma(1),\gamma(0)}) = \Phi - S(\hat{\psi}).
  \end{equation}
  We will bound both of these terms. 

  By the definition of $\Phi$, see Lemma \ref{lem:sgi-converge},
  $$|\Phi| \le \sum_{v\in \cT} |S(\theta_v)| \lesssim \sigma(\Lambda).$$
  
  It remains to bound $S(\hat{\psi})\leq \ell(\hat\psi)^2\lesssim \ell(\psi)^2$. We will show that 
  \begin{equation}\label{eq:lgamma-target}
    \ell(\psi)^2 \lesssim \max\left\{\sigma(\Lambda), \sigma(\Lambda) \log \frac{(\diam K)^2}{\sigma(\Lambda)}\right\}.
  \end{equation}
  
  Let $\psi_1 = G\circ \overline{(1,1),(1,0)}$ and $\psi_0 = G\circ \overline{(0,1),(0,0)}$, so that $\psi = \tilde{\psi}_1 \diamond \psi_0$, where $\tilde{\psi}_1$ is the reverse of $\psi_1$. For $t=0,1$, $\psi_t$ consists of segments from $g_i(t)$ to $g_{i+1}(t)$ and
  $$\ell(\psi_t) = \sum_{i=0}^\infty |g_{i+1}(t) - g_i(t)|.$$
  By \eqref{eq:G-gamma-dist} and the triangle inequality, 
  \begin{equation}\label{eq:pointwise-psi}
    |g_{i+1}(t) - g_i(t)| \lesssim 2^{-i}\diam(K),
  \end{equation}
  for all $i$, so $\ell(\psi) \lesssim \diam(K)$. We consider two cases. If $\sigma(\Lambda) > \frac{1}{16} \diam(K)^2$, then $\ell(\psi)^2 \lesssim \diam(K)^2 \lesssim \sigma(\Lambda)$, implying \eqref{eq:lgamma-target}.

  We thus consider the case $\diam(K)^2 \ge 16 \sigma(\Lambda)$. Let
  $$i_0=\left\lfloor \frac{1}{2} \log_2 \frac{\diam(K)^2}{\sigma(\Lambda)}\right\rfloor$$
  so that $i_0\ge 2$.
  For all $i$ and for $t=0,1$, let $w_{t,i}\in \cT^i$ be the unique vertex such that $t\in J_{w_{t,i}}$. Then $|g_{i+1}(t) - g_i(t)| \le \delta_\Lambda(w_{t,i}),$
  and, using \eqref{eq:pointwise-psi},
  \begin{multline*}
    \ell(\psi_t) = \sum_{i=0}^\infty |g_{i+1}(t) - g_i(t)| \lesssim \sum_{i=0}^{i_0-1} \delta_\Lambda(w_{t,i}) + \sum_{i=i_0}^{\infty}  2^{-i}\diam(K) \\ \lesssim  \sum_{i=0}^{i_0-1} \delta_\Lambda(w_{t,i}) + 2^{-i_0} \diam(K).
  \end{multline*}
  By Cauchy--Schwarz, for any $x_i$, 
  $$\left(\sum_{i=1}^{k} x_i \right)^2 \le k\sum_{i=1}^k x_i^2,$$
  so 
  \[
  \begin{aligned}
    \ell(\psi_t)^2 &\lesssim (i_0 + 1) \left(( 2^{-i_0} \diam(K))^2 + \sum_{i=0}^{i_0-1} \delta_\Lambda(w_{t,i})^2\right)\\ 
    &\lesssim i_0 \left(\diam(K)^2 2^{-2i_0} + \sigma(\Lambda)\right) 
    \lesssim i_0 \sigma(\Lambda).
    \end{aligned}
  \]
  Therefore,
  $$\ell(\psi)^2 \le 2 \ell(\psi_0)^2 + 2 \ell(\psi_1)^2\lesssim i_0 \sigma(\Lambda) \approx \sigma(\Lambda) \log \frac{(\diam K)^2}{\sigma(\Lambda)},$$
  as desired.

  Finally, if $g_i(0)=g_i(1)$ for all $i$, then $\psi_0=\psi_1$, so $S(\psi) = S(\psi_0)-S(\psi_1) = 0$ and
  $$|S(\hat{\gamma})| = |\Phi - S(\psi)| = |\Phi| \lesssim \sigma(\Lambda).$$
\end{proof}

In the following subsections, we will prove Lemmas~\ref{lem:arch-lengths}--\ref{lem:z-sum}.

\subsection{Proof of Lemma~\ref{lem:arch-lengths}}\label{sec:proof-arch-lengths}

In this section, we prove Lemmas~\ref{lem:boundary-arches}
and~\ref{lem:arch-lengths}. We first state a generalization of Lemma~\ref{lem:boundary-arches}. A set $\cS\subset \cT$ is \emph{semicoherent} if 
\begin{itemize}
\item $\cS$ has a unique maximal element,
\item if $v \in \cS$ is not maximal, then $\cP(v)\in \cS$, where $\cP(v)$ is the parent vertex of $v$.
\end{itemize}
In particular, $\cU_P$ is a semicoherent set that contains the root $v_0$ of $\cT$. Recall that $R(\cS) = \bigcup_{v\in \cS} R_v$. By the semicoherence of $\cS$, if $(s,t)\in R(\cS)$, then $(s,t')\in R(\cS)$ for any $t'\in [t,1]$. It follows that $R(\cS)$ is contractible.

For $i\ge 0$, let
$$F_i(\cS) := \bigcup_{v\in \cT^i\cap \cS} J_v.$$
By the semicoherence of $\cS$, $F_{i+1}(\cS) \subset F_i(\cS)$ for all $i$. Let
\begin{equation}\label{eq:def-F-boundary}
  F(\cS) := \bigcap_{i=1}^\infty F_i(\cS).
\end{equation}
This is a closed subset, and one can check that
$$\closure R(\cS) = R(\cS) \cup (F(\cS)\times \{0\}).$$
We will show that $R(\cS)$ satisfies the following lemma. Since $\cU$ is semicoherent, this implies Lemma~\ref{lem:boundary-arches}.

\begin{lemma}\label{lem:coherent-boundary}
  Let $\cS\subset \cT$ be a semicoherent set such that $v_0\in \cS$. If $\cS$ is finite, then $R(\cS)$ is homeomorphic to a closed disc.

  Otherwise, if $\cS$ is infinite, then $R(\cS)$ is a simply-connected noncompact planar manifold. If $I=(a,b)$ is a connected component of $[0,1]\setminus F(\cS)$, then $R(\cS)$ has a boundary component connecting $(a,0)$ to $(b,0)$, and this component can be parameterized by a map $\alpha=(\alpha_1,\alpha_2)\from [0,1]\to \partial R(\cS)$ such that $\alpha_1$ is nondecreasing and $\alpha_2(x)>0$ for all $x\in (0,1)$.
\end{lemma}

\begin{proof}
  The rectangles $R_v$ cover $[0,1] \times (0,1]$, overlapping only at their boundaries. Indeed, they form the $2$--cells of a cellulation of $[0,1] \times (0,1]$, which we denote by $\tau$. The edges of $\tau$ are each the left, right, or top edge of some $R_v$; since a vertex $v$ may have multiple children, the bottom edge of an $R_v$ may be subdivided into multiple edges.
  
  Then $R(\cS)$ is a union of $2$--cells of $\tau$. Let $C = (\partial R(\cS)) \cap ([0,1] \times (0,1])$ be the relative boundary of $R(\cS)$, which is a union of edges of $\tau$. Suppose that $p\in C$ is a vertex of $\tau$. Up to four rectangles $R_v$ intersect at $p$, so either two or four edges of $C$ meet at $p$. The only way that four edges of $C$ can meet at $p$ is if $R(\cS)$ forms a checkerboard pattern around $p$, with two diagonally opposite rectangles in $R(\cS)$ and the other two not in $R(\cS)$. Since $\cS$ is semicoherent, this is impossible. Therefore, $R(\cS)$ is a planar manifold with boundary. By the remarks before the lemma, $R(\cS)$ is simply connected.

  If $\cS$ is finite, then $R(\cS)$ is a simply-connected compact planar manifold, so it is homeomorphic to a closed disc. This proves the first part of the lemma.

  Suppose that $\cS$ is infinite. Then $F(\cS)$ is nonempty, so $R(\cS)$ is noncompact. A simply-connected planar manifold with an $S^1$ boundary component is a closed disc and thus compact, so every boundary component of $R(\cS)$ is homeomorphic to $\R$.

  Let $I=(a,b)$ be a connected component of $[0,1]\setminus F(\cS)$; this implies $a,b\in F(\cS)$. For $s\in [0,1]$, let $I_s = \{t : (s,t)\in R(\cS)\}$; this is an interval containing $1$. Let $s_0\in (a,b)$ and let $t_0 = \inf I_{s_0}$. Since $s_0\not \in F(\cS)$, we have $t_0>0$ and $(s_0,t_0)\in \partial R(\cS)$. Let $\gamma=(\gamma_1,\gamma_2) \from\R\to \partial R(\cS)$ parameterize the boundary component of $R(\cS)$ containing $(s_0,t_0)$.

  We claim that $\gamma(x)$ converges to $(a,0)$ and $(b,0)$ as $x$ goes to $\infty$ or $-\infty$. Since
  $$R(\cS) \supset ([a,b]\times [\tfrac{1}{2},1]) \cup (\{a,b\}\times (0,1]),$$
  we have $\gamma(\R) \subset [a,b] \times [0,\frac{1}{2}]$ and thus $\gamma_1(\R)\subset [a,b]$.

  By the semicoherence of $\cS$, if $t<1$ and $(s,t)\in \inter R(\cS)$, then $(s,t')\in \inter R(\cS)$ for all $t'\in [t,1)$. If $t<1$ and
  $$(s,t)\in \inter (\R^2\setminus R(\cS)),$$
  then $(s,t')\in \inter (\R^2\setminus R(\cS))$ for all $t' \in [0,t]$. Therefore, for any $s\in [0,1]$, the intersection $\gamma(\R)\cap \{s\}\times \R$ is either empty or an interval. That is, $\gamma$ crosses $\{s\}\times \R$ at most once. It follows that $\gamma_1$ is either nondecreasing or nonincreasing.  After possibly reversing the parameterization, suppose that $\gamma_1$ is nondecreasing. 

  Since $\tau$ is locally finite, $\gamma(x)$ leaves every compact subset of $[0,1]\times (0,1]$ as $x\to \pm \infty$, so
  $$\lim_{x\to -\infty} \gamma_2(x)=\lim_{x\to \infty} \gamma_2(x)=0.$$
  Since $\gamma_1$ is monotone, $\lim_{x\to -\infty} \gamma(x)$ and $\lim_{x\to \infty} \gamma(x)$ exist and lie in
  $$([0,1]\times \{0\}) \cap \closure R(\cS) = F(\cS)\times \{0\}.$$
  The only two such points with $s$--coordinates between $a$ and $b$ are $(a,0)$ and $(b,0)$, so $\lim_{x\to -\infty} \gamma(x) = (a,0)$ and $\lim_{x\to \infty} \gamma(x) = (b,0)$. We obtain the desired $\alpha$ by letting $\alpha|_{(0,1)}$ be a parametrization of $\gamma$ and extending to $[0,1]$ by continuity.
\end{proof} 

Now we prove Lemma~\ref{lem:arch-lengths}. We recall that for a partition $P=\{t_0,\dots,t_k\}$, we let $\cU_P = \{v\in \cT : J_{v}\cap P\ne \emptyset\}$ and let $A_j$ be the arch of $\partial R(\cU_P)$ from $(t_{j-1},0)$ to $(t_j,0)$. We take $u_j$ to be a parameterization of $A_j$ with $u_j(0) = (t_{j-1},0)$ and $u_j(1) = (t_{j},0)$ and we let $a_j = G\circ u_j$.

\begin{proof}[Proof of Lemma~\ref{lem:arch-lengths}]
  Let $\cW_j$ be as \eqref{eq:def-cwj}.
  We claim that the $\cW_j$'s satisfy the lemma.
  
  We first show \eqref{eq:boundary-graph}. Suppose that $p = (s,t)\in A_j\setminus \{(t_{j-1},0), (t_{j},0)\}$. Since $p \in \partial R(\cU_P)$, there are $v\in \cU_P$ and $w\not \in \cU_P$ such that $p\in \partial R_v\cap \partial R_w$. We claim that $v\in \cW_j$. 

  Our choice of $v$ and $w$ implies that $s\in J_v\cap J_w$ and $|\gen(v) - \gen(w)| \le 1$. By Lemma~\ref{lem:coherent-boundary}, $s\in [t_{j-1},t_j]$ and $t>0$. Thus $J_v$ intersects $[t_{j-1},t_j]$. Since $v\in\cU_P$, we have $J_v\not\subset (t_{j-1},t_j)$, so $J_v$ contains $t_{j-1}$ or $t_j$.

  To bound $\diam I_v$, note that $J_w$ intersects $[t_{j-1},t_j]$ but not $P$. Thus $I_w\subset \zeta([t_{j-1},t_j])$, and by \eqref{eq:patch-diam}, 
  $$\diam I_v \le 2\mu^2 \diam I_w \le 2 \mu^2 \diam \zeta([t_{j-1},t_j]),$$
  so $v\in \cW_j$. Since $p\in\partial R_v$, this proves \eqref{eq:boundary-graph}.

  Before proving \eqref{eq:alpha-length}, we prove the bounds on $\cW_j$. We first prove part \ref{it:W-multiplicity}. Let $v\in \cT$. If $v\in \cW_j$, then $\{t_{j-1},t_j\}\cap J_v\ne \emptyset$ and $$\diam I_v \le 2 \mu^2 \diam \zeta([t_{j-1},t_j]).$$

  If $\{t_{j-1},t_j\}\cap J_v\ne \emptyset$, then $[t_{j-1},t_j] \subset J_v$ or $[t_{j-1},t_j]$ contains one of the endpoints of $J_v$. There are at most four $j$'s such that $[t_{j-1},t_j]$ contains one of the endpoints of $J_v$, so it is enough to bound the cardinality of
  $$N_v = \{j : [t_{j-1},t_j] \subset J_v\text{ and }\diam \zeta([t_{j-1},t_j]) \ge (2\mu^2)^{-1}\diam I_v\}.$$

  For any $j\in N_v$, we have $$\diam \zeta([t_{j-1},t_j]) \ge (2 \mu^2)^{-1}\diam I_v,$$ so by
  Lemma~\ref{lem:chains}, there is an $N>0$ (depending on $\mu$, $\lambda$, and $n$) such that $|N_v| < N$ for all $v$. This proves part \ref{it:W-multiplicity}.

  For part \ref{it:arch-count}, note that for any $i\ge 0$, there are at most two $v\in \cT^i$ such that $t_{j-1}\in J_v$ (depending on whether $t_{j-1}$ is one of the endpoints of $J_v$) and at most two $v\in \cT^i$ such that $t_{j}\in J_v$. Therefore, $|\cW_j \cap \cT^i|\le 4$.

  Finally, we use these bounds to prove \eqref{eq:alpha-length}. By \eqref{eq:boundary-graph} and \eqref{eq:ell-theta-v}, we have
  $$\ell(a_j) \stackrel{\eqref{eq:boundary-graph}}{\le} \sum_{w\in \cW_j} \ell(\theta_w) \stackrel{\eqref{eq:ell-theta-v}}{\lesssim} \sum_{w\in \cW_j} \delta_\Lambda(w).$$
  This proves the first part of \eqref{eq:alpha-length}.

  For the second part, let $n_j = \min_{w\in \cW_j} \gen(w)$. Let $w\in \cT^{n_j}\cap \cW_j$. Then
  $$2^{-n_j} \diam K = 2^{-\gen(w)} \diam K \stackrel{\eqref{eq:patch-diam}}{\le} \mu \diam I_w \stackrel{\eqref{eq:def-cwj}}{\le} 2\mu^3 \diam \zeta([t_{j-1},t_j]).$$
  Since $|\cW_j \cap \cT^i| \le 4$ for all $i$,
  \[
  \begin{aligned}
    \sum_{w\in \cW_j} \delta_\Lambda(w) &\leq \sum_{i=n_j}^\infty |\cW_j \cap \cT^i| 2^{-i}\diam(K)\\ 
    &\lesssim 2^{-n_j}\diam(K)
    \lesssim \diam \zeta([t_{j-1},t_j])\lesssim d(\zeta(t_{j-1}),\zeta(t_j)),
  \end{aligned}
\]
where the last inequality follows from Lemma~\ref{lem:vertical-diameter-v2}.
  \end{proof}

\subsection{Proof of Lemma~\ref{lem:lim-dyadic}}\label{sec:proof-dyadic}
Lemma~\ref{lem:lim-dyadic} is a consequence of the following lemma. As in \eqref{eqn:RelationPhiPAlphaP}, let $\phi_P = G|_{\partial R(\cU_P)}$.
\begin{lemma}\label{lem:area-phi-P}
  For any partition $P$, $S(\phi_P) = \sum_{v\in \cU_P} S(\theta_v)$.
\end{lemma}
  
\begin{proof}
  Let $P=(t_0,\dots, t_k)$ be a partition. For any $i\ge 0$, let $\cU_P^{\le i } = \cU_P\cap \cT^{\le i}$ and $\cU_P^{> i } = \cU_P\cap \cT^{> i}$. Let $D_i = R(\cU_P^{\le i})$; note that
  $$D_i = R(\cU_P) \cap \big([0,1]\times [2^{-i-1},1]\big).$$
  Let $\eta_{i} = G|_{\partial D_i}$. 
  Since $\cU_P^{\le i}$ is finite, Lemma~\ref{lem:coherent-boundary} implies that
  $D_i$ is homeomorphic to a closed disc and the additivity of $S$ implies
  \begin{equation}\label{eq:eta-i-sum}
    S(\eta_i) = \sum_{v\in \cU_P^{\le i}} S(\theta_v).
  \end{equation}
  
  Since $D_i \subset R(\cU_P)$ and 
  $$R(\cU_P)\setminus D_i = R(\cU_P) \cap ([0,1]\times [0,2^{-i-1})),$$
  the difference $\fcl{\phi_P} - \fcl{\eta_i}$ has support
  $$
  \supp(\fcl{\phi_P} - \fcl{\eta_i}) \subset G(P\times \{0\}) \cup \bigcup_{v\in \cU_P^{>i}} G(\partial R_v).
  $$
  In fact, $\fcl{\phi_P} - \fcl{\eta_i}$ can be written as a sum of closed curves with total length $L_i$, where
  $$L_i \le \sum_{v\in \cU_P^{>i}} \ell(\theta_v).
  $$

  By Lemma~\ref{lem:sgi-converge},
  \begin{equation}\label{eq:area-phi-P-sum}
    \sum_{v\in \cU_P} \ell(\theta_v) \lesssim \sum_{v\in \cU_P} 2^{-\gen(v)} \diam(K) \le \sum_{i=0}^\infty 2^{-i} \diam(K) |\cU_P\cap \cT^i|.
  \end{equation}
    
  For each $i$, $|\cU_P\cap \cT^i|\le 2(k+1)$, so $\sum_{v\in \cU_P} \ell(\theta_v) < \infty$. It follows that $\lim_{i\to \infty} L_i = 0$. By Lemma~\ref{lem:symplectic-area-props-v2}, $|S(\phi_P) - S(\eta_i)| \lesssim L_i^2$, so by \eqref{eq:eta-i-sum}, 
  $$S(\phi_P) = \lim_{i\to\infty} S(\eta_i) = \lim_{i\to \infty} \sum_{v\in \cU_P^{\le i}} S(\theta_v).$$
  By Lemma~\ref{lem:sgi-converge}, $\sum_{v\in \cT} S(\theta_v)$ converges absolutely, so
  $$S(\phi_P) = \sum_{v\in \cU_P} S(\theta_v),$$
  as desired.
\end{proof}

\begin{proof}[Proof of Lemma~\ref{lem:lim-dyadic}] 
  Let $P=(t_0,\dots,t_k)$ be a partition. By Lemma~\ref{lem:area-phi-P}, $S(\phi_P) = \sum_{v\in \cU_P} S(\theta_v)$. For any $v\in \cT$, if $\mesh P < \ell(J_v)$, then $P\cap J_v\ne \emptyset$, so $v\in \cU_P$. It follows that $\cU_P\to \cT$ as $\mesh P \to 0$. Therefore, since $\sum_{v\in \cT} S(\theta_v)$ converges absolutely,
  $$\lim_{\mesh P \to 0} S(\phi_P) = \lim_{\mesh P\to 0}\sum_{v\in \cU_P} S(\theta_v) = \sum_{v\in \cT} S(\theta_v) = \Phi,$$
  as desired.

  Finally, notice that by Lemma~\ref{lem:area-phi-P} and Lemma \ref{lem:sgi-converge},
  \[
  |S(\phi_P)|\le \sum_{v\in \cU_P} |S(\theta_v)| \leq \sum_{v\in \cT} |S(\theta_v)|\lesssim \sigma(\Lambda),
  \]
  as desired.
\end{proof}

\subsection{Proof of Lemma~\ref{lem:arch-area}}
Recall that for each $j$, $a_j$ is the segment of $\alpha_P$ from $\gamma(t_{j-1})$ to $\gamma(t_j)$, see Lemma~\ref{lem:arch-lengths}. Let $\Delta z_j = z(\zeta(t_{j-1})^{-1} \zeta(t_j))$. We claim that for any $\epsilon>0$, there is a $\rho>0$ such that if $\mesh P<\rho$, then
  $$
  \sum_j \ell(a_j)^2 \le \epsilon + \epsilon \sum_j \Delta z_j.
  $$

  Let $m>0$ be a large integer and let $\delta > 0$ be a small number, both to be chosen later. Let $\eta = 2 \mu^2$. Let $i_0>0$ be such that
  $$\sum_{v\in \cT^{> i_0}}\delta_\Lambda(v)^2 \le \delta.$$
  Since $\sigma(\Lambda)<\infty$, such an $i_0$ exists.
  Let $\rho>0$ be such that
  $$\eta \diam \zeta([a,b]) < \mu^{-1} 2^{-i_0} \diam K$$
  for all $a,b\in [0,1]$ such that $|a-b|<\rho$.

  Suppose that $\mesh P < \rho$.
  Let $\cW_j\subset \cT$ be as in Lemma~\ref{lem:arch-lengths} and let $n_j = \min_{w\in \cW_j} \gen(w)$ so that $\cW_j \subset \cT^{\ge n_j}$. By our choice of $\eta$, we have $\diam I_w \le \eta \diam \zeta([t_{j-1},t_j])$ for all $w\in \cW_j$. If $w\in \cW_j\cap \cT^{n_j}$, then 
  $$\mu^{-1}2^{-n_j} \diam K \stackrel{\eqref{eq:patch-diam}}{\le} \diam I_w \le \eta \diam \zeta([t_{j-1},t_j]) < \mu^{-1} 2^{-i_0}\diam K,$$
  so $n_j > i_0$. Furthermore, by this inequality and Lemma~\ref{lem:vertical-diameter-v2},
  \begin{equation}\label{eq:nj-size-bound}
    \mu^{-1}2^{-n_j} \diam K \lesssim \diam \zeta([t_{j-1},t_j]) \lesssim \sqrt{\Delta z_j}.
  \end{equation}
  
  By \eqref{eq:alpha-length},
  \[
  \begin{aligned}
    \ell(a_j)^2 &\lesssim \left(\sum_{w\in \cW_j} \delta_\Lambda(w)\right)^2 \\ 
    &\lesssim \left(\sum_{w\in \cW_j(<n_j + m)}\delta_\Lambda(w)\right)^2 + \left(\sum_{w\in \cW_j(\ge n_j + m)}\delta_\Lambda(w)\right)^2,
\end{aligned}
  \]
  where $\cW_j(< k) = \cW_j\cap \cT^{< k}$ and $\cW_j(\ge k) = \cW_j\cap \cT^{\ge k}$.
  We bound these terms separately. On one hand, by Lemma~\ref{lem:arch-lengths}.(2), we have $|\cW_j\cap\cT^k|\le 4$ for all $k$, so $|\cW_j(<n_j + m)| \le 4m$, and
  $$\left(\sum_{w\in \cW_j(<n_j + m)}\delta_\Lambda(w)\right)^2 \le 4m \sum_{w\in \cW_j(<n_j + m)}\delta_\Lambda(w)^2,$$
  by Cauchy--Schwarz. On the other hand, by Lemmas~\ref{lem:sgi-converge} and \ref{lem:arch-lengths}.(2),
  $$
  \begin{aligned}
  \sum_{w\in \cW_j(\ge n_j + m)}\delta_\Lambda(w) &\stackrel{\eqref{eq:ell-theta-v}}{\lesssim} \sum_{k=n_j+m}^\infty 4 \cdot 2^{-k}\diam(K) \\
  & \le 8 \cdot 2^{-m} 2^{-n_j}\diam(K) 
\\ &\stackrel{\eqref{eq:nj-size-bound}} 
  {\lesssim} 2^{-m} \sqrt{\Delta z_j}.
  \end{aligned}
  $$
  That is,
  $$\ell(a_j)^2\lesssim 4m \sum_{w\in \cW_j(<n_j + m)}\delta_\Lambda(w)^2 + 2^{-2m} \Delta z_j,$$
  and
  $$\sum_j \ell(a_j)^2\lesssim 4m \sum_j \sum_{w\in \cW_j(<n_j + m)}\delta_\Lambda(w)^2 + 2^{-2m}\sum_j \Delta z_j.$$
  
  By Lemma~\ref{lem:arch-lengths}.(1), each $v\in \cT$ appears in at most $c$ of the $\cW_j$'s, so
  $$4m \sum_j \sum_{w\in \cW_j(<n_j + m)}\delta_\Lambda(w)^2 \le 4cm \sum_{w\in \bigcup_j \cW_j}\delta_\Lambda(w)^2.$$
  Since $\cW_j \subset \cT^{>i_0}$ for all $j$,
  $$4m \sum_j \sum_{w\in \cW_j \cap \cT^{<n_j + m}}\delta_\Lambda(w)^2 \le 4cm \sum_{w\in \cT^{>i_0}}\delta_\Lambda(w)^2 \le 4cm\delta.$$

  Therefore, 
  $$\sum_j \ell(a_j)^2\lesssim 4cm\delta + 2^{-2m} \sum_j \Delta z_j,$$
  and if $2^{-2m}$ and $m \delta$ are sufficiently small, then 
  $$\sum_j \ell(a_j)^2 \le \epsilon + \epsilon  \sum_j \Delta z_j,$$
  as desired.

\subsection{Proof of Lemma~\ref{lem:z-sum}}
In this section, we show that $\sum_j \Delta z_j$ is uniformly bounded as $\mesh(P)\to 0$. That is, there is some $c>0$ such that for any sufficiently fine partition $P$, $\sum_j \Delta z_j < c$. Recall that $\gamma_P$ is the piecewise-linear approximation of $\gamma$ corresponding to $P$,
 $$\psi = (G\circ \overline{(1,0),(1,1)}) \diamond (G\circ \overline{(0,1),(0,0)}),$$
 $\alpha_P = a_1\diamond \dots \diamond a_k$, and $\phi_P = \alpha_P \diamond \psi.$

  By Lemma~\ref{lem:arch-area} and Lemma \ref{lem:symplectic-area-props-v2}.(3), if $P$ is sufficiently fine, then
  \begin{equation}\label{eq:half-bound}
    \sum_j |S(\hat{a}_j)| \le \sum_j \ell(a_j)^2 \le 1 + \frac{1}{2} \sum_j \Delta z_j.
  \end{equation}
  Then, by \eqref{eq:alpha-P-gamma-P},
  \begin{equation}\label{eq:half-gamma-P}
    |S(\gamma_P)| = \bigg|S(\alpha_P)-\sum_j S(\hat{a}_j)\bigg| \le |S(\alpha_P)| + \frac{1}{2} \sum_j \Delta z_j +  1.
  \end{equation}
  
  We can write
  $$\zeta(1) = \zeta(0)\cdot \left(\zeta(t_0)^{-1}\zeta(t_1)\right)\dots \left(\zeta(t_{k-1})^{-1} \zeta(t_{k})\right),$$
  so by Lemma~\ref{lem:discrete-area-z} and \eqref{eq:half-gamma-P},
  \begin{equation}\label{eq:zj-first}
    \sum_j \Delta z_j  = z(\zeta(1)) - z(\zeta(0)) - S(\gamma_P) \le z(\zeta(1)) - z(\zeta(0)) + |S(\alpha_P)| + \frac{1}{2} \sum_j \Delta z_j + 1
  \end{equation}
  and thus
  $$\sum_j \Delta z_j \le 2(z(\zeta(1)) - z(\zeta(0))) + 2 |S(\alpha_P)| + 2.$$

  By Lemma~\ref{lem:lim-dyadic}, there
  is a $c_0>0$ depending only on $K$ such that $|S(\phi_P)|< c_0$ for all $P$ and thus $|S(\alpha_P)| = |S(\phi_P) - S(\psi)| < c_0 + |S(\psi)|$.
  Then 
  $$\sum_j \Delta z_j < 2(z(\zeta(1)) - z(\zeta(0))) + 2 c_0 + 2 |S(\psi)| + 2,$$
  as desired. 

\appendix

\section{Proofs of lemmas from Section~\ref{sec:CoAreaC1H}}\label{app:vertical-curves}

In this appendix, we give the proofs of some lemmas that were stated in Section~\ref{sec:CoAreaC1H}.

\begin{proof}[Proof of Lemma~\ref{lem:bilipschitz-affine-v2}]
  We identify an affine map $\alpha\from\R^{2n}\to\R^{2n}$ with
$\alpha\circ\pi\from\HH\to\R^{2n}$ when no confusion arises. For $x\in \HH$, let
  $$\tilde f(x)=f(x)-\pi(x)-(f(p)-\pi(p)).$$
  Then $\tilde{f}(p) = 0$ and $\|D_{\mathrm{H}}\tilde{f}_q\|_{\mathrm{Op}}<c$ for all $q\in 5D$. Furthermore, 
  $$\alpha_{\tilde{f},D} = \alpha_{f,D} - \alpha_{\pi,D} - f(p) + \pi(p) = \alpha_{f,D} - \id_{\R^{2n}} - f(p) + \pi(p),$$
  so to prove \eqref{eq:ControlAlpha2}, it suffices to show that for any $x\in D$, $|\alpha_{\tilde{f},D}(\pi(x))| \lesssim cR$.
  
  By Lemma~\ref{lem:ControlF}, $|\tilde{f}(x)| = |\tilde{f}(x) - \tilde{f}(p)| \lesssim cR$ for all $x\in D$. 
  The $L_\infty(D)$ norm and the $\widehat{L}_2(D)$ norm are equivalent on $\Aff$ (with a constant depending on $n$), so 
  $$\|\alpha_{\tilde{f},D}\circ \pi\|_{L_\infty(D)} \approx  \|\alpha_{\tilde{f},D} \circ \pi\|_{\widehat{L}_2(D)} \le \|\tilde{f} \|_{\widehat{L}_2(D)} \lesssim cR,$$
  as desired.
  
  Next, we show \eqref{eqnBilipControl}. Since $\alpha_{\tilde{f},D}$ is affine, there are $M\in \R^{2n\times 2n}$ and $b\in \R^{2n}$ such that $\alpha_{\tilde{f},D}(v) = Mv + b$ for all $v\in \R^{2n}$. Then
  $$\alpha_{f,D}(v) = \alpha_{\tilde{f},D}(v) + v + f(p) - \pi(p) = (M + \id_{\R^{2n}})v + b_0$$
  for $b_0 = b + f(p) - \pi(p)$.
  
  We can decompose $\Aff = \Lin \oplus \Con$, where $\Lin$ and $\Con$ are the subspaces of linear functions and constant functions, respectively. By the symmetry of $D$, $\Lin$ is orthogonal to $\Con$ in $L_2(D)$. Then $M\circ \pi$ is the orthogonal projection of $\alpha_{\tilde{f},D}$ to $\Lin$, so $M\circ \pi$ is the orthogonal projection of $\tilde{f}$ to $\Lin$. Therefore,
  $$\|M\circ \pi\|_{\widehat{L}_2(D)} \le \|\tilde{f} \|_{\widehat{L}_2(D)} \lesssim cR.$$
  If $c$ is small enough that $\|M\|_{\Op} < \frac{1}{2}$, then
  $$|\alpha_{f,D}(v) - \alpha_{f,D}(w)| = |(M+\id_{\R^{2n}})(v-w)| \in \left[\frac{1}{2} |v-w|, \frac{3}{2} |v-w|\right],$$
  as desired.
\end{proof}

\begin{proof}[Proof of Lemma~\ref{lem:patchwork-exist}]  
  For $g,h\in K$ with $g\prec h$, let $[g,h]_K = \{ p \in K : g \preceq p \preceq h\}$.
  
  By Lemma~\ref{lem:compact-intersections-v2} there is a $C=C(\lambda)>1$ such that if $I\subset K$ is a segment with endpoints $q$ and $q'$, then $\diam I \le C d(q,q')$. By Lemma~\ref{lem:Biholder}, there are $\theta<1$ and $N$ depending on $\lambda$ and $n$ such that for any $g,h\in K$ with $g\prec h$ there is a sequence $g=q_0\prec q_1\prec \ldots\prec q_k=h$ such that $k\le N$ and
  \begin{equation}\label{eq:good-subdiv}
    \theta (2C)^{-1}d(g,h) \le d(q_i,q_{i+1})\leq (2C)^{-1} d(g,h)
  \end{equation}
  for all $i$. We call such a sequence a \emph{good subdivision} of $[g,h]_K$.  

  Let 
  \begin{equation}\label{eq:def-mu}
    \mu:=\max\left\{\sqrt{2} C \theta^{-\frac{1}{2}},N\right\}.
  \end{equation}
    
  We iteratively construct $(\cT, \{I_v\}_{v\in \cT})$ as follows. To start, let $v_0$ be the root of $\cT$ and let $I_{v_0} = K$. Suppose by induction that we have constructed the first $i$ generations of $\cT$ and the corresponding $I_v$'s so that Definition~\ref{def:patchwork} is satisfied. Let $v\in \cT^i$ and suppose that $I_v = [g,h]_K$. Then
  \begin{equation}\label{eq:patch-diamNew}
    \diam I_v\in [\mu^{-1} 2^{-i}\diam K, \mu 2^{-i}\diam K].
  \end{equation}
  If $\diam I_v \le \mu 2^{-i-1}\diam K$, we add a single child of $v$ to $\cT$, which we call $v'$, and set $I_{v'} = I_v$. This satisfies Definition~\ref{def:patchwork}.

  Otherwise, we have
  \begin{equation}\label{eq:Iv-bound}
    \mu 2^{-i-1}\diam K < \diam I_v \le \mu 2^{-i}\diam K.
  \end{equation}
  Let $(q_0,\dots, q_k)$ be a good subdivision of $[g,h]_K$, so that the $q_i$'s satisfy \eqref{eq:good-subdiv}.
  We add $k$ children of $v$ to $\cT$, called $c_0,\dots, c_{k-1}$,
  and for each $j=0,\dots,k-1$, we let $I_{c_j} = [q_{j},q_{j+1}]_K$.

  On one hand, our choice of $C$ implies that for all $j$,
  $$\diam I_{c_j} \leq C d(q_j,q_{j+1})\stackrel{\eqref{eq:good-subdiv}}{\le} \frac{1}{2} d(g,h) \le \frac{1}{2} \diam I_v.$$
  On the other hand, 
  $$\diam I_{c_j} \geq d(q_j,q_{j+1})\stackrel{\eqref{eq:good-subdiv}}{\ge} \frac{\theta}{2C} d(g,h)\geq \frac{\theta}{2C^2} \diam I_v \ge \mu^{-2} \diam I_v.$$
  Thus, by \eqref{eq:Iv-bound}, 
  \[
  \begin{aligned}
  \mu^{-1} 2^{-i-1}\diam K &\stackrel{\eqref{eq:Iv-bound}}{\le} \mu^{-2} \diam I_v \le \diam I_{c_j} \\
  &\le \frac{1}{2} \diam I_v \stackrel{\eqref{eq:Iv-bound}}{\le} \mu 2^{-i-1}\diam K,
  \end{aligned}
  \]
  so the $I_{c_j}$'s satisfy \eqref{eq:patch-diam}. The rest of the conditions in Definition~\ref{def:patchwork} are easy to check.
\end{proof}

\section{Signed area of Hölder curves in \texorpdfstring{$\R^{2n}$}{R\^{}2n}}\label{app:rn-curve-area}
\subsection{An area formula for Hölder curves}
In this section, we show how to adapt the proof of Proposition \ref{prop:vertical-curve-area} to prove a similar result for curves in $\R^{2n}$. When $\gamma\from[0,1]\to\mathbb R^{2n}$ is an $\alpha$--H\"older curve with $\alpha>1/2$, the symplectic area of $\gamma$ can be defined by using Young's integral. The following proposition extends this definition of symplectic area to $\frac{1}{2}$--Hölder curves.

\begin{prop}\label{prop:rn-curve-area}
  Let $L>0$. Let $\gamma\from [0,1] \to \R^{2n}$ be a curve such that
  \begin{equation}\label{eq:rn-gamma-holder}
    |\gamma(s) - \gamma(t)| \le L \sqrt{|s-t|},
  \end{equation}
  for all $s,t\in [0,1]$. For $i\ge 0$ and $j=0,\dots, 2^i$, let $\Lambda_{i,j} \in \R^{2n}$ be such that
  \begin{equation}\label{eqn:Lambdaij}
  |\Lambda_{i,j} - \gamma(j2^{-i})| \leq L 2^{-\frac{i}{2}}.
  \end{equation}
  For each $i$, let $h_i$ be the curve
  \begin{equation}\label{eqn:girn}
  h_i =\overline{\Lambda_{i,0},\dots, \Lambda_{i,2^i}},
  \end{equation}
  that connects $\Lambda_{i,0},\dots, \Lambda_{i,2^i}$ by straight line segments. Let
  $$\delta_{i,j} = \diam(\{\Lambda_{i,j}, \Lambda_{i,j+1}, \Lambda_{i+1,2j}, \Lambda_{i+1,2j+1}, \Lambda_{i+1,2j+2}\}).$$
  If
  \begin{equation}\label{eqn:ConditionSigmaRn}
  \sigma := \sum_{i=0}^\infty \sum_{j=0}^{2^i-1} \delta_{i,j}^2 < \infty,
  \end{equation}
  then $S(\hat{\gamma})$ exists, $S(\hat{\gamma}) = \lim_iS(\hat{h}_i)$, and 
  \begin{equation}\label{eq:sgi-bound-rn}
    |S(\hat{\gamma})| \lesssim \sigma \max\left\{1, \log \frac{L^2}{\sigma}\right\}.
  \end{equation}
  Here $\hat\gamma,\hat h_i$ are the closed curves obtained by joining the endpoints of $\gamma,h_i$ with straight line segments. If $h_i(0)=h_i(1)$ for all $i$, or if $h_i(0) = \gamma(0)$ and $h_i(1) = \gamma(1)$ for all $i$, then 
  $|S(\hat{\gamma})| \lesssim \sigma$.
\end{prop}
When $\Lambda_{i,j}=\gamma(j2^{-i})$, this implies Theorem~\ref{thm:ExistenceAreaCurves}. 

This lets us use the argument of Proposition~\ref{prop:vertical-curve-area} to prove Proposition~\ref{prop:rn-curve-area}. There are two main changes. First, we define $G$ so that it interpolates between the $h_i$'s rather than the $g_i$'s constructed in Section~\ref{sec:area-formula}. Second, the $\frac{1}{2}$--Hölder condition on $\gamma$ lets us avoid some of the complications of the proof, particularly around Lemma~\ref{lem:arch-area} and Lemma~\ref{lem:z-sum}.

We start by establishing some notation. We parameterize $h_i$ so that $h_i(j2^{-i}) = \Lambda_{i,j}$. For any $i\ge 0$ and any $(t, u) \in [0,1]\times [2^{-i-1}, 2^{-i}]$, we let
\begin{equation*}
  G(t, u) = h_{i+1}(t) + \frac{u - 2^{-i-1}}{2^{-i} - 2^{-i-1}} (h_i(t) - h_{i+1}(t))
\end{equation*}
and $G(t,0)=\gamma(t)$ for all $t$, as in \eqref{eq:def-G}.

This is a slightly different map than the one used in Section~\ref{sec:area-formula}, but we can still use $G$ to break the $h_i$'s into a sum of loops. Let $\cT$ be the infinite rooted binary tree. We label the vertices by $v_{i,j}$, $i\ge 0$, $0\le j < 2^i$, so that $v_{0,0}$ is the root of $\cT$ and $\cC(v_{i,j}) = \{v_{i+1,2j},v_{i+1,2j+1}\}$ for all $i$ and $j$. Let $J_{v_{i,j}} = [j2^{-i}, (j+1)2^{-i}]$. For every $i\ge 0$ and $v\in \cT^i$, we let $R_v=J_v\times [2^{-i-1}, 2^{-i}]$ as in \eqref{eq:def-rv} and let $\theta_v := G|_{\partial R_v}$.

As in Section~\ref{sec:area-formula}, each curve $\theta_{v_{i,j}}$ is a polygon, and the total area of the $\theta_{v_{i,j}}$'s is bounded. In fact, 
$$\theta_{v_{i,j}} = \overline{\Lambda_{i,j},\Lambda_{i,j+1}, \Lambda_{i+1,2j+2}, \Lambda_{i+1,2j+1}, \Lambda_{i+1,2j}},$$
so $\ell(\theta_{v_{i,j}}) \le 5 \delta_{i,j}$ and $|S(\theta_{v_{i,j}})| \lesssim \delta_{i,j}^2$, so $\sum_{v\in \cT} S(\theta_v)$ converges absolutely. This lets us use the methods of Section~\ref{sec:area-formula} to prove Proposition~\ref{prop:rn-curve-area}.

\begin{proof}[Proof of Proposition~\ref{prop:rn-curve-area}]
  For all $i$ and $j$, let $\delta_\Lambda(v_{i,j}) = \delta_{i,j}$. The argument that proves Lemma~\ref{lem:nbhd-diam} implies
  \begin{equation}\label{eqn:EstimateDeltaLambdarn}
    \delta_\Lambda(v) \lesssim L 2^{-\frac{i}{2}}
  \end{equation}
  for all $v\in \cT^i$ and   
  \begin{equation}\label{eqn:Gtu}
    |G(t,u)-\gamma(t)|\lesssim L \sqrt{u}
  \end{equation}
  for all $t$ and $u$.  These bounds differ slightly from \eqref{eq:delta-lambda-bound} and \eqref{eq:G-gamma-dist} because those inequalities deal with patchworks whose level--$i$ pieces have diameter on the order of $2^{-i}$, whereas in this situation, $\gamma(J_v)$ can have diameter of order $2^{-\frac{\gen(v)}{2}}$. These estimates will still suffice for the argument.

  Let $\Phi = \sum_{v\in\cT} S(\theta_v)$. As noted above, this sum converges absolutely, so the proof of Lemma~\ref{lem:sgi-converge} shows that
  \[
    \lim_{i\to\infty} S(h_i)=\Phi-S(\psi),
  \]
  where
  $$\psi = (G\circ \overline{(1,0),(1,1)}) \diamond \overline{h_0(1), h_0(0)}\diamond (G\circ \overline{(0,1),(0,0)})$$
  traces out three sides of $G([0,1]^2)$.
  (This is slightly different from \eqref{eq:def-psi} because $h_0$ is not necessarily constant.)
  It remains to show that $S(\gamma) = \lim_{i\to \infty} S(h_i)$.

  For any partition $P=\{t_0,\dots,t_k\}$ of $[0,1]$, we define $\gamma_P$, $\cU_P$, and $R(\cU_P)$ as in Section~\ref{sec:area-formula}.
  Lemma~\ref{lem:boundary-arches} holds as before, so $\partial R(\cU_P)$ contains arches $A_j$ connecting $(t_{j-1},0)$ and $(t_j,0)$. We let $u_j$ parameterize $A_j$, let $a_j = G\circ u_j$, and let $\alpha_P = a_1\diamond \dots \diamond a_k$. Then, as in \eqref{eq:vca-pf-s-gamma}, 
  \begin{equation}\label{eq:s-gamma-s-alpha-rn}
    S(\gamma_P) = S(\alpha_P) - \sum_j S(\hat{a}_j),
  \end{equation}
  so it suffices to estimate $S(\alpha_P)$ and the $\ell(a_j)$'s.

  The arguments in Section~\ref{sec:proof-arch-lengths} show that if
  \begin{equation}\label{eq:def-cwjrn}
    \cW_j := \{w\in \cT : J_w\cap \{t_{j-1},t_j\} \ne \emptyset\, \text{ and}\, \diam J_w \leq 2|t_j-t_{j-1}|\},
  \end{equation}
  then \eqref{eq:boundary-graph} holds and 
  \begin{equation}\label{eq:alpha-length-rn}
    \ell(a_j) \lesssim \sum_{w\in \cW_j} \delta_\Lambda(w) \lesssim \sqrt{t_{j}- t_{j-1}}
  \end{equation}
  for all $j$, as in Lemma~\ref{lem:arch-lengths}. Furthermore, parts (\ref{it:W-multiplicity}) and (\ref{it:arch-count}) of Lemma~\ref{lem:arch-lengths} hold verbatim.

  Likewise, the proof of Lemma~\ref{lem:lim-dyadic} (see Section~\ref{sec:proof-dyadic}) still holds, using the inequality $\ell(\theta_v)\lesssim 2^{-\frac{\gen(v)}{2}}$ rather than $\ell(\theta_v)\lesssim 2^{-\gen(v)}$ in \eqref{eq:area-phi-P-sum}. That is,
  \begin{equation}\label{eq:lim-dyadic-conclusion}
    \lim_{\mesh(P)\to 0} S(\alpha_P) = \lim_{\mesh(P)\to 0} (S(\phi_P) - S(\psi)) = \Phi - S(\psi).
  \end{equation}
  
  Furthermore, we can use the argument of Lemma~\ref{lem:arch-area} to prove that if $\mesh(P)$ is sufficiently small, then
  \begin{equation}\label{eqn:ajepsilon}
    \sum_j \ell(a_j)^2\leq \epsilon.
  \end{equation}

  We proceed as follows. Let $\delta > 0$ and $m\in \mathbb{N}$ be numbers to be chosen later. Let $i_0>0$ be such that
  $$
  \sum_{v\in \cT^{\ge i_0}}\delta_\Lambda(v)^2 \le \delta.
  $$
  Since $\sigma<\infty$, such an $i_0$ exists.
  Let $\rho:=2^{-i_0-2}$.

  Suppose that $\mesh P < \rho$.
  Let $\cW_j\subset \cT$ be as in \eqref{eq:def-cwjrn} and let $n_j = \min_{w\in \cW_j} \gen(w)$. For any $w\in \cW_j$,
  \begin{equation}\label{eqn:SqrtEstimate}
  2^{-n_j}=\diam J_w \leq 2|t_j-t_{j-1}|\leq 2\mesh P < 2^{-i_0},
  \end{equation}
  so $n_j > i_0$. Thus, by \eqref{eq:alpha-length-rn},
  \[
  \begin{aligned}
    \ell(a_j)^2 \lesssim \left(\sum_{w\in \cW_j} \delta_\Lambda(w)\right)^2  
    \lesssim \left(\sum_{w\in \cW_j(<n_j + m)}\delta_\Lambda(w)\right)^2 + \left(\sum_{w\in \cW_j(\ge n_j + m)}\delta_\Lambda(w)\right)^2,
\end{aligned}
  \]
  where $\cW_j(< k) := \cW_j\cap \cT^{< k}$ and $\cW_j(\ge k) := \cW_j\cap \cT^{\ge k}$.

  We bound these terms separately. On one hand, as in Lemma~\ref{lem:arch-lengths}.(\ref{it:arch-count}), we have $|\cW_j\cap\cT^k|\le 4$ for all $k$, so $|\cW_j(<n_j + m)| \le 4m$. Therefore,
  $$
  \left(\sum_{w\in \cW_j(<n_j + m)}\delta_\Lambda(w)\right)^2 \le 4m \sum_{w\in \cW_j(<n_j + m)}\delta_\Lambda(w)^2,
  $$
  by Cauchy--Schwarz. Likewise,
  \begin{multline*}
    \sum_{w\in \cW_j(\ge n_j + m)}\delta_\Lambda(w)
    \stackrel{\eqref{eqn:EstimateDeltaLambdarn}}{\lesssim} L \sum_{k=n_j+m}^\infty |\cW_j\cap\cT^k| 2^{-\frac{k}{2}} \\
    \lesssim L \cdot 2^{-\frac{m}{2}} 2^{-\frac{n_j}{2}} \stackrel{\eqref{eqn:SqrtEstimate}}{\lesssim} L2^{-\frac{m}{2}}\sqrt{|t_j-t_{j-1}|}.
  \end{multline*}
  That is,
  $$\ell(a_j)^2\lesssim 4m \sum_{w\in \cW_j(<n_j + m)}\delta_\Lambda(w)^2 + L^22^{-m} |t_j-t_{j-1}|,$$
  and
  $$\sum_j \ell(a_j)^2\lesssim 4m \sum_j \sum_{w\in \cW_j(<n_j + m)}\delta_\Lambda(w)^2 + L^22^{-m}.$$

  Note that $\bigcup_j \cW_j \subset \cT^{>i_0}$ and that each $v\in \cT$ appears in at most $c$ of the $\cW_j$'s (as in Lemma~\ref{lem:arch-lengths}.(\ref{it:W-multiplicity})). Therefore, 
  $$
  \sum_j \sum_{w\in \cW_j(<n_j + m)}\delta_\Lambda(w)^2 \le \sum_{w\in \bigcup_j \cW_j}\delta_\Lambda(w)^2 \le c \sum_{w\in \cT^{>i_0}}\delta_\Lambda(w)^2 \le c\delta,
  $$
  and
  $$\sum_j \ell(a_j)^2\lesssim 4cm\delta + L^22^{-m}.$$
  If we choose $m$ large enough and $\delta$ small enough, this implies \eqref{eqn:ajepsilon}. Thus, by Lemma~\ref{lem:symplectic-area-props-v2}.(3), 
  \begin{equation}\label{eq:s-hat-aj-rn}
    \lim_{\mesh P \to 0}\sum_j S(\hat{a}_j) = 0.
  \end{equation}

  Finally,
  \begin{equation}\label{eq:sgamma-rn}
    S(\gamma) = \lim_{\mesh P \to 0} S(\gamma_P) \stackrel{\eqref{eq:s-gamma-s-alpha-rn}}{=} \lim_{\mesh P \to 0} S(\alpha_P) - \sum_j S(\hat{a}_j)  \stackrel{\eqref{eq:lim-dyadic-conclusion} \wedge \eqref{eq:s-hat-aj-rn}}{=} \Phi - S(\psi).
  \end{equation}

  As in \eqref{eq:s-gamma-s-psi}, we have $S(\hat\gamma)=\Phi-S(\hat\psi)$ and $|\Phi|\lesssim\sigma$. It remains to show that 
  \begin{equation}\label{eq:rn-target}
    S(\hat{\psi}) \lesssim \ell(\psi)^2 \lesssim \sigma\max\left\{1,\log\frac{L^2}{\sigma}\right\}.
  \end{equation}
  This follows from essentially the same argument as in the proof of Proposition~\ref{prop:vertical-curve-area}, but there are a few changes; we briefly sketch the argument here.

  For $t=0,1$, let $\psi_t = G\circ \overline{(t,1),(t,0)}$ and let $\tilde{\psi}_1$ be the reverse of $\psi_1$, so that $\psi = \tilde{\psi}_1 \diamond \overline{h_0(1), h_0(0)}\diamond \psi_0$. 
  Then $\psi_t$ consists of segments from $h_i(t)$ to $h_{i+1}(t)$. By \eqref{eqn:Gtu}, we have $|h_{i+1}(t)-h_i(t)|\lesssim L2^{-\frac{i}{2}}$, and by \eqref{eq:rn-gamma-holder} and \eqref{eqn:Lambdaij}, we have $|h_0(0)-h_0(1)| \lesssim L$, so
  \begin{equation}\label{eq:ell-psi-rn}
    \ell(\psi) \lesssim |h_0(0)-h_0(1)| + \sum_{i=0}^\infty L2^{-\frac{i}{2}} \lesssim L.
  \end{equation}

  If $\sigma > \frac{1}{16} L^2$, then $\ell(\psi)^2 \lesssim L^2 \lesssim \sigma$, implying \eqref{eq:rn-target}. Otherwise, we let
  $$i_0=\left\lfloor \log_2 \frac{L^2}{\sigma}\right\rfloor;$$
  note that $i_0 \ge 4$.  Then $|h_0(0)-h_0(1)| \lesssim \delta_{\Lambda}(v_{0,0})$, $|h_{i+1}(0)-h_i(0)|\le \delta_{\Lambda}(v_{i,0})$, and $|h_{i+1}(1)-h_i(1)|\le \delta_{\Lambda}(v_{i,2^i-1})$, so we can sharpen \eqref{eq:ell-psi-rn} to
  \begin{multline*}
    \ell(\psi) \lesssim \sum_{i=0}^{i_0-1} \left(\delta_\Lambda(v_{i,0}) + \delta_\Lambda(v_{i,2^i})\right) + \sum_{i=i_0}^{\infty} L2^{-\frac{i}{2}}\\
    \lesssim \sum_{i=0}^{i_0-1} \left(\delta_\Lambda(v_{i,0}) + \delta_\Lambda(v_{i,2^i})\right) + L2^{-\frac{i_0}{2}}.
  \end{multline*}
  By Cauchy--Schwarz,
  $$\ell(\psi)^2 \lesssim i_0 \bigg(\sum_{i=0}^{i_0-1} \left(\delta_\Lambda(v_{i,0}) + \delta_\Lambda(v_{i,2^i})\right)^2 + L^22^{-i_0}\bigg) \lesssim
  i_0 (\sigma + L^22^{-i_0}) \lesssim i_0 \sigma,$$
  as desired.

  Finally, if $h_i(0)=h_i(1)$ for all $i$, then $\psi_0 = \psi_1$, and $S(\psi) = S(\tilde{\psi}_1 \diamond \psi_0) = 0$. If $h_i(0) = \gamma(0)$ and $h_i(1) = \gamma(1)$ for all $i$, then $\psi = \overline{\gamma(1),\gamma(0)}$ and $\ell(\psi)^2 \le \delta_\Lambda(v_{0,0})^2 \le \sigma$. In either case, $S(\hat{\psi})\lesssim \sigma$, so \eqref{eq:sgamma-rn} implies that
  $$|S(\hat\gamma)| \le |\Phi| + |S(\hat\psi)| \lesssim \sigma,$$
  as desired.
\end{proof}

We also note the following consequence of this argument for later use.
\begin{lemma}\label{lem:sq-partition}
  Suppose that $\gamma\from [0,1]\to \R^{2n}$ is a curve satisfying the hypothesis of  Proposition~\ref{prop:rn-curve-area}. Then for any $\epsilon>0$, there is a $\delta>0$ such that if $P=\{t_0,\dots, t_k\}$ is a partition of $[0,1]$ with $\mesh P < \delta$, then
  $$\sum_{j=0}^{k-1} |\gamma(t_j)-\gamma(t_{j+1})|^2\leq \epsilon.$$
\end{lemma}
\begin{proof}
  This is a direct consequence of \eqref{eqn:ajepsilon}.
\end{proof}

\subsection{Examples of Hölder curves without signed area}\label{sec:examples}

In this section, we study how the signed area of a curve is related to the signed area of its dyadic approximations. For a curve $f\from [0,1]\to \R^n$ and $k \ge 0$, let $\cD_kf\from [0,1]\to \R^n$ be the $k$th dyadic approximation of $f$. That is, $\cD_k f$ is the piecewise-linear function such that $\cD_k f(j2^{-k}) = f(j 2^{-k})$ for all $j$ and $\cD_k f$ is affine on $[j2^{-k}, (j+1)2^{-k}]$ for all $j$.

Recall that for any partition $P=\{t_1,\dots, t_k\}$, 
$$A_P(\gamma) = \frac{1}{2}\sum_{i=1}^{k-1} (\gamma_1(t_i) \gamma_2(t_{i+1}) - \gamma_2(t_i) \gamma_1(t_{i+1}))$$
and $A(\gamma) = \lim_{\mesh(P)\to 0} A_P(\gamma)$. In particular, if $\cP_k$ is the $k$th dyadic partition of $[0,1]$, then $A_{\cP_k}(\gamma) = A(\cD_k\gamma)$.

When $\gamma$ is an $\alpha$--Hölder curve and $\alpha > \frac{1}{2}$, $A(\gamma)$ exists; in fact, if $P$ and $P'$ are two partitions, then
\begin{equation}\label{eq:uniform-area}
  |A_P(\gamma) - A_{P'}(\gamma)| \lesssim \max(\mesh(P),\mesh(P'))^{2\alpha - 1},
\end{equation}
with constant depending on the $\alpha$--Hölder constant of $\gamma$. 
This is a standard consequence of Young’s theory of integration \cite{YoungIntegral}.

We construct two examples. Our first example shows that \eqref{eq:uniform-area} fails when $\alpha=\frac{1}{2}$. In fact, we construct a curve $\gamma\from [0,1]\to \R^2$ such that $A(\cD_k\gamma) = 0$ for all $k$, but $A(\gamma)$ does not exist.

To explain our second example, let $\delta_{i,j}(\gamma)$ and $\sigma(\gamma)$ be as in Theorem~\ref{thm:ExistenceAreaCurves}. Then the region between $\cD_i\gamma$ and $\cD_{i+1}\gamma$ consists of $2^i$ triangles of diameters $\delta_{i,0}(\gamma),\dots \delta_{i,2^{i}-1}(\gamma)$, so
$$|A(\cD_i\gamma) - A(\cD_{i+1}\gamma)| \le \sum_j \delta_{i,j}(\gamma)^2$$
and
$$|A(\cD_{k+1}\gamma)| \le \sum_{i=0}^k \sum_{j=0}^{2^i-1} \delta_{i,j}(\gamma)^2.$$
That is, $A(\cD_{k+1}\gamma)$ is a sum of areas of triangles. If $\sigma(\gamma) < \infty$, then the corresponding infinite sum converges absolutely.

It is natural to ask whether this absolute convergence condition (without the $\frac{1}{2}$--Hölder condition) implies $A(\gamma) = \lim_k A(\cD_{k+1}\gamma)$. We will show that it does not, by constructing an example of a $(\frac{1}{2}-\epsilon)$--Hölder curve $\lambda\from [0,1]\to \R^2$ such that $\sigma(\lambda) < \infty$, but $A(\lambda)$ does not exist.

Both of these examples are based on the following construction. For $i\ge 0$, let
\begin{equation}\label{eq:def-ab}
  a_i = \frac{1}{3}(1 - 4^{-i})\qquad b_i = \frac{1}{3}(1 + 2\cdot 4^{-i}),
\end{equation}
so that $(a_i)_i = 0, \frac{1}{4}, \frac{5}{16},\dots$ and $(b_i)_i = 1, \frac{1}{2}, \frac{3}{8}, \dots$ are sequences of dyadic rationals converging to $\frac{1}{3}$. Let $\alpha\from [0,\frac{1}{3}) \cup (\frac{1}{3}, 1] \to \R$ be the piecewise-linear function such that $\alpha(a_i) = \alpha(b_i) = i$ and $\alpha$ is affine on $[a_i, a_{i+1}]$ and $[b_{i+1},b_i]$ for all $i$. 

\begin{lemma}\label{lem:cdf}
  For each $i$, the graph of $\cD_{2i}\alpha$ is the piecewise-linear curve
  \begin{equation}\label{eq:cDf-even}
    \{y = \cD_{2i}\alpha(x)\} = \overline{(a_0,0),\dots, (a_i,i),(b_i,i),\dots, (b_0,0)},
  \end{equation}
  i.e., $\cD_{2i}\alpha = \min\{\alpha,i\}$,
  and the graph of $\cD_{2i+1}\alpha $ is the piecewise-linear curve
  \begin{equation}\label{eq:cDf-odd}
    \{y = \cD_{2i+1}\alpha(x)\} = \overline{(a_0,0),\dots, (a_i,i),(b_{i+1},i+1),\dots, (b_0,0)},
  \end{equation}
  where $\overline{p_1,\dots,p_k}$ is the curve that connects $p_1,\dots,p_k$ by line segments. In particular, $\|(\cD_{n}\alpha)'\|\le 2^{n}$ for all $n$.
\end{lemma}
\begin{proof}
  Note that if $f$ is affine on $[a,b]$ and $a,b \in 2^{-k}\Z$ for some $k \ge 0$, then $\cD_k f|_{[a,b]} = f|_{[a,b]}$.
  One can check that for $i\ge 0$, 
  $$\{a_0,a_1,\dots\} \cap 2^{-i}\Z = \{a_0,\dots,a_{\lfloor\frac{i}{2}\rfloor}\}$$
  and 
  $$\{b_0,b_1,\dots\} \cap 2^{-i}\Z =    \{b_0,\dots,b_{\lceil\frac{i}{2}\rceil}\}.$$

  Let $i\ge 0$, let $k = \lfloor\frac{i}{2}\rfloor$, and let $l = \lceil\frac{i}{2}\rceil$. Then
  $$\{a_0,a_1,\dots\}\cup \{b_0,b_1,\dots\} \cap 2^{-i}\Z = \{a_0,\dots, a_k, b_0,\dots, b_l\}$$
  and $b_l - a_k = 2^{-i}$. Therefore, $\cD_i\alpha$ agrees with $\alpha$ on $[0, a_k]$ and $[b_l, 1]$ and $\cD_i\alpha$ is affine on $[a_k,b_l]$. All of the differences $|\alpha(a_j)-\alpha(a_{j+1})|$, $|\alpha(b_j)-\alpha(b_{j+1})|$, and $|\alpha(b_l)-\alpha(a_{k})|$ are at most $1$, so $\|(\cD_i\alpha)'\|_{\infty}\le 2^{i}$, as desired.
\end{proof}

If $Q = \{t_0,\dots, t_q\}$ is a partition with $$a = t_0 < t_1 <\dots  < t_q = b$$ and $f\from [a,b]\to \R^n$ is a function which is affine on each interval $[t_i,t_{i+1}]$, we say that $f$ is \emph{$Q$--PL}. If $f$ is defined on $[0,1]$ and is affine on each interval $[j2^{-i},(j+1)2^{-i}]$, we say that $f$ is \emph{$2^{-i}$--PL}. In particular, for any curve $f\from [0,1]\to \R^n$, the approximation $\cD_if$ is $2^{-i}$--PL.

Let $\beta\from [0,\infty)\to \R^2$ be the map such that $\beta|_{[0,4]}$ is a unit-speed parameterization of the boundary of the unit square and $\beta(0)=\beta(4) = \beta(t) = (0,0)$ for $t \ge 4$. For an integer $k > 0$, let $\theta_k \from [0,1]\to \R^2$, $\theta_k(t) = \beta(k^{-1} \alpha(t))$. Since $\beta$ is constant on $[4,\infty)$, 
\begin{equation}\label{eq:thetak}
  \theta_k(t) = \beta(k^{-1} \alpha(t)) = \beta(k^{-1} \min\{4k,\alpha(t)\}) = \beta(k^{-1} \cD_{8k}\alpha(t)).
\end{equation}
That is, $\theta_k(t)$ traces the unit square forward as $t$ ranges from $0$ to $\frac{1}{3}$, then backward for $t$ from $\frac{1}{3}$ to $1$.

One can check that $\theta_k$ is $2^{-8k}$--PL; in fact it is $Q_k$--PL, where
$$Q_k = \{a_0,\dots, a_{4k}, b_{4k},\dots, b_0\}.$$
If we apply the argument in Lemma~\ref{lem:cdf} to $\theta_k$, we find that
\begin{equation}\label{eq:cdtheta}
  \cD_i \theta_k(t) = \beta(k^{-1} \cD_{\min\{8k,i\}}\alpha(t))
\end{equation}
for all $i$ and $t$. That is, when $i < 8k$, $\cD_i \theta_k$ traces part of the unit square forward, then backward. 

Our two examples will both come from the following construction.
Let $(r_i)_i$ and $(k_i)_i$ be sequences of integers such that $r_0 \ge 0$, $k_i \ge 1$, and $r_{i+1} \ge r_i + 16 k_i + 2$ for all $i$. We define a sequence of $2^{-r_i}$--PL curves $\gamma_i\from [0,1]\to \R^2$ as follows.
Let $\gamma_0$ be the constant curve $0$. Suppose that $i\ge 0$ and that $\gamma_i$ is a $2^{-r_i}$--PL curve. For every 
$j=0,\dots, 2^{r_i}-1$, let $c_{i,j} = j 2^{-r_i}$ and let $m_{i,j} = \frac{c_{i,j}+ c_{i,j+1}}{2}$. If $t\in [c_{i,j},c_{i,j+1}]$, we let
\begin{equation}\label{eqn:defgammai}
\gamma_{i+1}(t) =
\begin{cases}
  \gamma_i(c_{i,j}) + 2^{-\frac{r_i}{2}} \theta_{k_i}((t-c_{i,j}) 2^{r_i + 1}) & t \in [c_{i,j}, m_{i,j}] \\
  \gamma_i(c_{i,j} + 2(t-m_{i,j})) & t \in [m_{i,j}, c_{i,j+1}].
\end{cases}
\end{equation}
Then $\gamma_{i+1}$ is $2^{-r_i-1}2^{-8k_i}$--PL and thus $2^{-r_{i+1}}$--PL. Hence, $\gamma_i(t) = \gamma_{i+1}(t)$ for all $t\in 2^{-r_i}\Z$.

\begin{lemma}\label{lem:gamma-i-no-area}
  The curves $\gamma_i$ converge uniformly to a curve $\gamma$ such that $A(\cD_k\gamma) = 0$ for all $k$, but $A(\gamma)$ does not exist.

  If there is a $c>0$ such that $k_i < c$ for all $i$, then $\gamma$ is $\frac{1}{2}$--Hölder. For any sequence $(k_i)_i$ and any $0<\epsilon<\frac{1}{2}$, there is a sequence $(r_i)_i$ such that $\gamma$ is $(\frac{1}{2}-\epsilon)$--Hölder.
\end{lemma}
\begin{proof} 
  First, we bound $\gamma_i'$ and show that the $\gamma_i$'s converge. Suppose by induction that $i\ge 0$ and $\|\gamma_{i}'\|_\infty\le 2^{\frac{r_{i}}{2}}$. Since $\gamma_0$ is constant, this holds for $i=0$. From the definition, 
  $$\|\gamma_{i+1}'\|_\infty \le \max \left\{2 \|\gamma_{i}'\|_\infty, 2^{\frac{r_i}{2} + 1} \|\theta_{k_i}'\|_\infty \right\}$$
  for all $i$. By \eqref{eq:thetak}, $\theta_{k_i}(t) = \beta(k_i^{-1} \cD_{8k_i}\alpha(t))$, so by Lemma~\ref{lem:cdf},
  $$\|\theta_{k_i}'\|_\infty \le k_i^{-1} \|(\cD_{8k_i}\alpha)'\|_\infty\le 2^{8k_i}.$$

  Therefore, 
  \begin{equation}\label{eq:gamma-deriv}
    \|\gamma_{i+1}'\|_\infty \le \max\{2^{\frac{r_i}{2}+1}, 2^{\frac{r_i}{2} + 1 + 8k_i}\} \le 2^{\frac{r_{i+1}}{2}},
  \end{equation}
  and by induction, $\|\gamma_{i}'\|_\infty\le 2^{\frac{r_{i}}{2}}$ for all $i\ge 0$.
  Therefore, for any $i\ge 0$ and any $0\le j < 2^{r_i}$,
  $$\diam \gamma_i([c_{i,j}, c_{i,j+1}]) \le 2^{-\frac{r_i}{2}}.$$

  The image $\gamma_{i+1}([c_{i,j}, c_{i,j+1}])$ consists of $\gamma_i([c_{i,j}, c_{i,j+1}])$ plus a square of side length $2^{-\frac{r_i}{2}}$, so it is contained in a ball of radius $2\cdot 2^{-\frac{r_i}{2}}$ around $\gamma_i(c_{i,j})$. Therefore, 
  $$\|\gamma_i - \gamma_{i+1}\|_\infty \le 3\cdot 2^{-\frac{r_i}{2}}.$$
  Since $r_{j+1} \ge 16 + r_j$ for all $j$, it follows that $\gamma_i$ converges uniformly to a function $\gamma$, and 
  \begin{equation}\label{eq:gamma-convergence}
    \|\gamma_i - \gamma\|_\infty \le \sum_{j=i}^\infty 3\cdot 2^{-\frac{r_j}{2}} \le 4\cdot 2^{-\frac{r_i}{2}}.
  \end{equation}

  Furthermore, we can estimate $A(\cD_k\gamma)$ and $A(\gamma)$. For all $i$, $j$, and $l$ such that $l\ge i$, we have $\gamma_{l+1}(c_{i,j}) = \gamma_{l}(c_{i,j})$, so $\gamma(c_{i,j}) = \gamma_i(c_{i,j})$. Since $\gamma_i$ is $2^{-r_i}$--PL, we have $\cD_{r_i}\gamma = \cD_{r_i}\gamma_i = \gamma_i$.

  We suppose that $r_i < k \le r_{i+1}$. Then $\cD_k \gamma = \cD_{k} \gamma_{i+1}$. The definition of $\gamma_{i+1}$ breaks $[0,1]$ into $2^{r_i + 1}$ segments, each of the form $[c_{i,j}, m_{i,j}]$ or $[m_{i,j}, c_{i,j+1}]$ for some $j$.

  On these segments, $\cD_{k}\gamma$ alternates between tracing part of a square forward and backward and tracing part of $\gamma_i$.
  That is, for $t\in [c_{i,j}, m_{i,j}]$, we have
  \begin{equation}\label{eq:dkgamma-formula}
    \cD_{k}\gamma(t) = \cD_{k}\gamma_{i+1}(t) = \gamma_i(c_{i,j}) + 2^{-\frac{r_i}{2}} \cD_{k - r_i - 1}\theta_{k_i}((t-c_{i,j}) 2^{r_i + 1}).
  \end{equation}
  By \eqref{eq:cdtheta}, $\cD_k\gamma|_{[c_{i,j}, m_{i,j}]}$ traces part of a square forward, then backward. By \eqref{eqn:defgammai}, $\cD_k\gamma|_{[m_{i,j}, c_{i,j+1}]}$ is the segment $\overline{\gamma_i(c_{i,j}), \gamma_i(c_{i,j+1})}$.

  The segments of $\cD_k\gamma$ that trace squares forward and backward do not contribute to the signed area, so $A(\cD_{k}\gamma) = A(\gamma_{i})$ for $r_i < k \le r_{i+1}$. Setting $k=r_{i+1}$, we have
  $$A(\gamma_{i+1}) = A(\cD_{r_{i+1}}\gamma) = A(\gamma_i).$$
  Therefore, by induction, $A(\gamma_i) = 0$ for all $i$ and $A(\cD_{k}\gamma) = 0$ for all $k$.

  It remains to show that $A(\gamma)$ does not exist. For each $i$ and $j$, let $t_{i,j} = c_{i,j} + \frac{1}{6} 2^{-r_i}$ be the point one-third of the way from $c_{i,j}$ to $m_{i,j}$. Then $\cD_{r_{i+1}} \gamma = \gamma_{i+1}$ traces out a square forward on $[c_{i,j}, t_{i,j}]$ and backward on $[t_{i,j}, m_{i,j}]$. There are points
  $$c_{i,j} < a_{i,j,1} < \dots < a_{i,j,4} < t_{i,j}$$
  such that $a_{i,j,k} \in 2^{-r_{i+1}} \Z$ and $\gamma_{i+1}(a_{i,j,1}), \dots, \gamma_{i+1}(a_{i,j,4})$ are the four vertices of the square, with $\gamma_{i+1}(a_{i,j,4}) = \gamma(c_{i,j})$. Since $\gamma$ agrees with $\gamma_{i+1}$ on $2^{-r_{i+1}} \Z$,
  $$\overline{\gamma(c_{i,j}), \gamma(a_{i,j,1}), \dots, \gamma(a_{i,j,4}), \gamma(c_{i,j+1})}$$
  consists of a square of side length $2^{-\frac{r_i}{2}}$ and a short segment of $\gamma_i$.

  Let
  $$P_i = \bigcup_{j=0}^{2^{r_i}-1} \{c_{i,j}, a_{i,j,1}, \dots, a_{i,j,4},c_{i,j+1}\}.$$
  Then $\gamma_{P_i}$ consists of $\gamma_i$ with $2^{r_i}$ squares of side length $2^{-\frac{r_i}{2}}$ inserted, and
  $$A_{P_i}(\gamma) = A(\gamma_i) + 2^{r_i} (2^{-\frac{r_i}{2}})^2 = 1.$$
  Since $\mesh P_i \to 0$ as $i\to \infty$, this means that $A(\gamma) = \lim_{\mesh P\to 0} A_P(\gamma)$ does not exist.

  Finally, we prove Hölder bounds on $\gamma$. We first consider the case that there is a $c>0$ such that $k_i< c$ for all $i$. By \eqref{eq:gamma-deriv},
  \begin{equation}\label{eq:gamma-deriv-c}
    \|\gamma_{i+1}'\|_\infty \lesssim 2^{8c} 2^{\frac{r_i}{2}}\lesssim_c 2^{\frac{r_i}{2}}.
  \end{equation}
  Let $s,t\in [0,1]$ with $s\ne t$. If $|s-t| > 2^{-r_0}$, then $\gamma_0(s) = \gamma_0(t) = 0$ and, by \eqref{eq:gamma-convergence},
  $$|\gamma(s) - \gamma(t)| \le |\gamma(s) - \gamma_{0}(s)| + |\gamma(t) - \gamma_{0}(t)| \le 8 \cdot 2^{-\frac{r_0}{2}} \lesssim \sqrt{|s-t|}.$$
  Otherwise, let $i$ be the unique integer such that $2^{-r_{i+1}} < |s-t| \le 2^{-r_i}$. Then by \eqref{eq:gamma-convergence} and \eqref{eq:gamma-deriv-c},
  \begin{align*}
    |\gamma(s) - \gamma(t)| 
    &\le |\gamma(s) - \gamma_{i+1}(s)| + |\gamma_{i+1}(s) - \gamma_{i+1}(t)| + |\gamma(t) - \gamma_{i+1}(t)|\\ 
    &\lesssim_c 4\cdot 2^{-\frac{r_{i+1}}{2}} + |s-t| 2^{\frac{r_{i}}{2}} + 4\cdot 2^{-\frac{r_{i+1}}{2}}.
  \end{align*}
  By our choice of $i$, $2^{-\frac{r_{i+1}}{2}} \le \sqrt{|s-t|}$ and $2^{\frac{r_{i}}{2}} \le |s-t|^{-\frac{1}{2}}.$ Therefore, $|\gamma(s) - \gamma(t)| \lesssim_c \sqrt{|s-t|}$, and $\gamma$ is $\frac{1}{2}$--Hölder.

  Next, we consider the case that $k_i$ is unbounded. Let $0 < \epsilon <\frac{1}{2}$ and suppose that
  $r_i \ge 8\epsilon^{-1}k_i$
  for all $i> 0$. We claim that $\gamma$ is $(\frac{1}{2}-\epsilon)$--Hölder.
  Let $s,t\in [0,1]$ with $s\ne t$. As above, we may suppose that $|s-t| \le 2^{-r_0}$. Let $i$ be the unique integer such that $2^{-r_{i+1}} < |s-t| \le 2^{-r_i}$. Then by \eqref{eq:gamma-deriv} and \eqref{eq:gamma-convergence},
  \begin{align*}
    |\gamma(s) - \gamma(t)| 
    &\le |\gamma(s) - \gamma_{i+1}(s)| + |\gamma_{i+1}(s) - \gamma_{i+1}(t)| + |\gamma(t) - \gamma_{i+1}(t)|\\ 
    &\lesssim 4\cdot 2^{-\frac{r_{i+1}}{2}} + |s-t| 2^{\frac{r_i}{2} + 1 + 8k_i} + 4\cdot 2^{-\frac{r_{i+1}}{2}}.
  \end{align*}
  By our choice of $i$, 
  $$2^{-\frac{r_{i+1}}{2}} \le \sqrt{|s-t|}\le |s-t|^{\frac{1}{2} - \epsilon}.$$
  Since $2^{r_i}\le |s-t|^{-1}$, 
  $$2^{\frac{r_i}{2} + 1 + 8k_i} \le 2\cdot 2^{\frac{r_i}{2} + \epsilon r_i} \le 2 \cdot |s-t|^{-\frac{1}{2} - \epsilon}.$$
  Therefore,
  $$|\gamma(s) - \gamma(t)| \le 8 |s-t|^{\frac{1}{2} - \epsilon} + 2 |s-t| \cdot |s-t|^{-\frac{1}{2} - \epsilon} \lesssim |s-t|^{\frac{1}{2} - \epsilon},$$
  and $\gamma$ is $(\frac{1}{2}-\epsilon)$--Hölder.
\end{proof}

It remains to bound $\sigma(\gamma)$. We will prove the following.
\begin{lemma}\label{lem:sigma-bound}
  Suppose that $k_i \ge i$ for all $i$ and $\sum_i \frac{1}{k_i}<\infty$. Then $\gamma$ satisfies
  $$\sigma(\gamma) \lesssim \sum_i \frac{1}{k_i}.$$
\end{lemma}
We first establish some notation. Let $\cQ$ be the set of dyadic intervals in $[0,1]$ and let $\cQ^i$ be the set of dyadic intervals of length $2^{-i}$. Let $\cQ^{< i} = \bigcup_{j< i} \cQ^j$.

For $f\from [0,1]\to \R^n$ and $[a,b]\in \cQ$, let
$$
\delta_{[a,b]}(f) := \diam \left\{f(a), f(\tfrac{a+b}{2}), f(b)\right\}.
$$
Then 
$$\sigma(f)=\sum_{I \in \cQ} \delta_{I}(f)^2.$$

For $[a,b]\in \cQ$, let $f_{[a,b]}\from [0,1]\to \R^n$,
$f_{[a,b]}(t) = f(a + t(b-a))$, so that $f_{[a,b]}$ is a reparameterization of $f|_{[a,b]}$. Then for any $I\in \cQ$, 
\begin{equation}\label{eq:sigma-f-I}
  \sigma(f_{I}) = \sum_{\substack{J\in \cQ \\ J \subset I}} \delta_J(f)^2.
\end{equation}
Before proving Lemma~\ref{lem:sigma-bound}, we prove the following lemmas.
\begin{lemma}\label{lem:partial-lip}
  Let $f\from [0,1]\to \R^n$ be a Lipschitz curve. Then
  $$\sigma(f)\le 2\Lip(f)^2,$$
  and for any $I=[a,b]\in \cQ$, 
  $$\sigma(f) \le 2\Lip(f|_{[0,1]\setminus (a,b)})^2 + \sigma(f_I).$$
\end{lemma}
\begin{proof}
  First, if $f$ is Lipschitz, then $\delta_I(f) \le \ell(I) \Lip(f)$ for all $I$, where $\ell(I)$ is the length of $I$. Therefore, 
  $$\sigma(f) = \sum_{I \in \cQ} \delta_{I}(f)^2 \le \sum_{i=0}^\infty (2^{-i}\Lip(f))^2 2^i = \Lip(f)^2 \sum_{i=0}^\infty 2^{-i} = 2\Lip(f)^2.$$

  Suppose that $I = [a,b] \in \cQ$ and let $L=\Lip(f|_{[0,1]\setminus (a,b)})$. If $J\in \cQ$ and $J\not \subset I$, then both endpoints of $J$ and the midpoint of $J$ lie in $[0,1]\setminus (a,b)$, so $\delta_J(f) \le L \ell(J)$. Therefore,
  $$\sigma(f) = \sum_{\substack{J \in \cQ \\ J \not \subset I}} \delta_{J}(f)^2 + \sum_{\substack{J \in \cQ \\ J \subset I}} \delta_{J}(f)^2 \le \sum_{i=0}^\infty (2^{-i}L)^2 2^i + \sigma(f_I) = 2L^2 + \sigma(f_I),$$
  as desired.
\end{proof}

\begin{lemma}\label{lem:PL-sigma}
  If $f$ is $2^{-i}$--PL and $k\ge i$, then
  $$\sum_{I\in \cQ^{k}} \delta_I(f)^2 \le 2^{i-k} \sigma(f).$$
\end{lemma}
\begin{proof}
  For $k\ge 0$, let $\sigma_k(f) = \sum_{I\in \cQ^{k}} \delta_I(f)^2$. Then, for $k \ge i$,
  $$\sigma_k(f) = \sum_{I\in \cQ^i} \sigma_{k-i}(f_{I}).$$
  Since $f$ is $2^{-i}$--PL, $f_{I}$ is affine for every $I\in \cQ^i$, so
  $$
  \sigma_{k-i}(f_{I}) = 2^{k-i} \big(|f_{I}(1)-f_{I}(0)|2^{i-k}\big)^2 = |f_{I}(1)-f_{I}(0)|^2 2^{i-k} = 2^{i-k} \sigma_{0}(f_{I}).
  $$
  Therefore,
  $$\sigma_k(f) = \sum_{I\in\cQ^i} 2^{i-k} \sigma_{0}(f_{I}) = 2^{i-k}\sigma_i(f) \le 2^{i-k}\sigma(f),$$
  as desired.
\end{proof}

\begin{proof}[Proof of Lemma~\ref{lem:sigma-bound}]
  We first calculate $\sigma(\theta_k)$. Let $f_k(t) = k^{-1} \cD_{8k}\alpha(t)$ so that $\theta_k = \beta\circ f_k$, see \eqref{eq:thetak}. Since $\beta$ is unit-speed, we have
  \begin{equation}\label{eq:sigma-theta-k-alpha}
    \sigma(\theta_k) \le \sigma(f_k) = k^{-2} \sigma(\cD_{8k}\alpha).
  \end{equation}

  We bound $\cD_{8k}\alpha$ inductively. The function $\alpha$ is self-similar in the sense that $\alpha_{[\frac{1}{4},\frac{1}{2}]} = \alpha + 1$. By Lemma~\ref{lem:cdf}, for any $i\ge 0$, 
  \begin{multline*}
    (\cD_{2i}\alpha)_{[\frac{1}{4},\frac{1}{2}]} 
    = (\min\{\alpha,i\})_{[\frac{1}{4},\frac{1}{2}]} 
    = \min\{\alpha + 1,i\} \\ 
    = 1 + \min\{\alpha,i-1\} 
    = 1 + \cD_{2i-2}\alpha.
  \end{multline*}

  Since $\cD_{0}\alpha=0$, we have $\sigma(\cD_{0}\alpha) = 0$. For $i>0$,
  $$\Lip(\cD_{2i} \alpha|_{[0,1]\setminus (\frac{1}{4},\frac{1}{2})}) = \Lip(\alpha|_{[0,1]\setminus (\frac{1}{4},\frac{1}{2})}) = 4,$$
  so by Lemma~\ref{lem:partial-lip},
  $$\sigma(\cD_{2i}\alpha) \le 2\cdot 4^2 + \sigma(\cD_{2i-2}\alpha) = 32 + \sigma(\cD_{2i-2}\alpha).$$
  By induction, $\sigma(\cD_{2i}\alpha) \le 32i$. Therefore, by \eqref{eq:sigma-theta-k-alpha},
  \begin{equation}\label{eq:sigma-theta-k}
    \sigma(\theta_k)\le k^{-2} \sigma(\cD_{8k}\alpha) \le k^{-2}\cdot 128k \approx k^{-1}.
  \end{equation}

  Now we consider the $\gamma_i$'s. Since $\gamma_0$ is constant, $\sigma(\gamma_0)=0$. We claim that
  \begin{equation}\label{eq:sigma-gamma-induct}
    \sigma(\gamma_{i+1}) \le (1+2^{-8(i-1)})\sigma(\gamma_i) + 128k_i^{-1}
  \end{equation}
  for all $i\ge 0$.

  Note that for any $j\ge 0$ and any $f\from [0,1]\to \R^n$, \eqref{eq:sigma-f-I} implies that
  \begin{equation}\label{eq:sigma-f-part}
    \sigma(f) = \sum_{I\in \cQ^{<j}} \delta_I(f)^2 + \sum_{J\in \cQ^j} \sigma(f_J).
  \end{equation}
  We take $f=\gamma_{i+1}$ and recall that $\gamma_i(t) = \gamma_{i+1}(t)$ for all $t\in 2^{-r_i}\Z$, so that $\delta_I(\gamma_{i}) = \delta_I(\gamma_{i+1})$. Then 
  \begin{equation}\label{eq:sigma-gamma-i+1}
    \sigma(\gamma_{i+1}) = \sum_{I\in \cQ^{< r_i}} \delta_I(\gamma_{i})^2 + \sum_{j=0}^{2^{r_i}-1} \sigma(g_j)
  \end{equation}
  where $g_j = (\gamma_{i+1})_{[c_{i,j}, c_{i,j+1}]}$ and $c_{i,j} = j2^{-r_i}$. Let $I_j = [c_{i,j}, c_{i,j+1}]$.
  
  Using \eqref{eqn:defgammai}, for each $j$,
  $$(g_j)_{[0,\frac{1}{2}]}(t) = \gamma_i(c_{i,j}) + 2^{-\frac{r_i}{2}} \theta_{k_i}(t)$$
  and
  $$(g_j)_{[\frac{1}{2},1]}(t) = (\gamma_i)_{I_j}(t).$$
  Therefore, by \eqref{eq:sigma-f-part},
  $$
  \begin{aligned}
    \sigma(g_j) &= \delta_{[0,1]}(g_j)^2 + \sigma((g_j)_{[0,\frac{1}{2}]}) +  \sigma((g_j)_{[\frac{1}{2},1]}) \\
  &= |g_j(0) - g_j(1)|^2 + \sigma((\gamma_i)_{I_j}) + 2^{-r_i}\sigma(\theta_{k_i}).
  \end{aligned}
  $$
  Furthermore,
  $$|g_j(0) - g_j(1)|^2 = |\gamma_i(c_{i,j}) - \gamma_i(c_{i,j+1})|^2 = \delta_{I_j}(\gamma_i)^2.$$
  Thus, $\sigma(g_j) = \delta_{I_j}(\gamma_i)^2 + \sigma((\gamma_i)_{I_j}) + 2^{-r_i}\sigma(\theta_{k_i})$. By \eqref{eq:sigma-gamma-i+1}, 
  \begin{align*}
    \sigma(\gamma_{i+1})
    & = \sum_{J\in \cQ^{< r_i}} \delta_J(\gamma_{i})^2 + \sum_{J\in \cQ^{r_i}} \left(\delta_J(\gamma_i)^2 + \sigma((\gamma_i)_{J})\right) + \sigma(\theta_{k_i}) \\
    & \stackrel{\eqref{eq:sigma-f-part}}{=} \sigma(\gamma_{i}) + \sum_{J\in \cQ^{r_i}} \delta_J(\gamma_i)^2 + \sigma(\theta_{k_i}).
  \end{align*}
  As we noted before Lemma~\ref{lem:gamma-i-no-area}, if $i>0$, then $\gamma_i$ is $2^{-r_{i-1}-1}2^{-8k_{i-1}}$--PL. We chose $r_{i} \ge r_{i-1} + 16k_{i-1} + 2$ and $k_i \ge i$, so by Lemma~\ref{lem:PL-sigma},
  $$\sum_{J\in \cQ^{r_i}} \delta_J(\gamma_i)^2 \le 2^{r_{i-1} - r_i  + 8k_{i-1} + 1}\sigma(\gamma_{i}) \le 2^{-8(i-1)} \sigma(\gamma_{i}).
  $$
  By \eqref{eq:sigma-theta-k}, $\sigma(\theta_{k_i}) \le 128 k_i^{-1}$, so 
  $$\sigma(\gamma_{i+1}) \le (1+2^{-8(i-1)})\sigma(\gamma_i) + 128k_i^{-1}.$$
  By induction on $i$, starting with the case $\sigma(\gamma_0) = 0$,
  $$
  \sigma(\gamma_i) \le 128 \left(\sum_{j=0}^{i-1} k_j^{-1}\right) \left(\prod_{j=0}^{i-1} (1+2^{-8(j-1)})\right) \lesssim \sum_{j=0}^{i-1} \frac{1}{k_j} < \infty.$$
  Let $C = \sup_i \sigma(\gamma_i) < \infty$.
  Since $\gamma(t) = \gamma_i(t)$ for $t\in 2^{-r_i}\Z$,
  $$\sigma(\gamma) = \sup_i \sum_{J\in \cQ^{<r_i}} \delta_J(\gamma)^2 = \sup_i \sum_{J\in \cQ^{<r_i}} \delta_J(\gamma_i)^2 \le \sup_i \sigma(\gamma_i)\le C,$$
  as desired.
\end{proof}

\subsection{Signed area and winding number}\label{rem:H20}
Let $\gamma\from[0,1]\to\mathbb R^2 \cong \mathbb{C}$ be a closed continuous curve and suppose that 
$z\notin \gamma([0,1])$. We define the \emph{winding number} of $\gamma$ around $z$ as
\[
\Wind(\gamma,z)
:=\frac{1}{2\pi}\big(\theta_z(1)-\theta_z(0)\big),
\]
where $\theta_z\from[0,1]\to\mathbb R$ is any continuous function such that
$$e^{i\theta_z(t)}=\frac{\gamma(t)-z}{|\gamma(t)-z|}.$$
Then $\Wind(\gamma,z)$ is an integer which is independent of the choice of $\theta_z$. If $\gamma=(x(t),y(t))\from[0,1]\to\mathbb R^2$ is Lipschitz continuous, then 
$\Wind(\gamma,\cdot)\in L_1(\mathbb R^2)$ and 
\begin{equation}\label{eqn:AreaLipschitz}
\int_{\mathbb R^2} \Wind(\gamma,z)\ud z
= A(\gamma) =
\frac12 \int_0^1 \left(x(t)y'(t)-y(t)x'(t)\right)\ud t.
\end{equation}

In this section, we will show that \eqref{eqn:AreaLipschitz} holds for curves that satisfy the conditions in Theorem~\ref{thm:ExistenceAreaCurves}. We prove the following.  

\begin{prop}\label{prop:winding}
    Let $\gamma\from[0,1]\to\mathbb R^2$ be a closed curve satisfying the assumptions of Theorem~\ref{thm:ExistenceAreaCurves}. Then $\mathcal{H}^2(\gamma([0,1]))=0$,  $\Wind(\gamma,\cdot)\in L_1(\mathbb R^2)$, and 
    \begin{equation}\label{eq:wind-integral}
      \int_{\mathbb R^2}\Wind(\gamma,z)\ud z = A(\gamma).
    \end{equation}
      
\end{prop}

\begin{proof}
  The fact that $\mathcal{H}^2(\gamma([0,1]))=0$ will follow from Lemma~\ref{lem:sq-partition}. Let $\epsilon>0$. By Lemma~\ref{lem:sq-partition}, there is a $\delta>0$ such that for every partition $P=\{t_0,\dots, t_k\}$ with $\mesh(P)<\delta$,
  \[\sum_{j=0}^{k-1}|\gamma(t_{j+1})-\gamma(t_j)|^2 <\epsilon.\]
  Let $P=\{t_0,\dots, t_k\}$ be such a partition. For each $j$, choose an interval $[a_j,b_j]\subset [t_j,t_{j+1}]$. Then
  $$P'=\{a_0,\dots, a_{k-1}\}\cup \{b_0,\dots, b_{k-1}\}\cup P$$
  satisfies $\mesh P' < \delta$. Let $t'_0,\dots,t'_{k'}$ be the elements of $P'$ in order; then
  \[\sum_{j=0}^{k-1}|\gamma(b_j)-\gamma(a_j)|^2 \le \sum_{j=0}^{k'-1}|\gamma(t'_{j+1})-\gamma(t'_j)|^2 <\epsilon.\]
  Since this is true for all choices of $a_j$ and $b_j$, we have
  \[\sum_{j=0}^{k-1}\diam(\gamma([t_j,t_{j+1}]))^2 < \epsilon\]
  for any partition with $\mesh(P)<\delta$. Therefore, $\mathcal{H}^2(\gamma([0,1]))=0$.

  It follows that $\Wind(\gamma,z)$ is defined for almost every $z\in \R^2.$ We claim that it satisfies \eqref{eq:wind-integral}. Let $\gamma_i$ be the piecewise affine function defined in the statement of Theorem~\ref{thm:ExistenceAreaCurves}. Since the $\gamma_i$'s converge uniformly to $\gamma$, if $z\notin \gamma([0,1])$, then $\Wind(\gamma_i,z)\to \Wind(\gamma,z)$ as $i\to\infty$. That is,
  \begin{equation}\label{eqn:aeconvergence}
    \Wind(\gamma_i,z)\to_{i\to\infty} \Wind(\gamma,z), \qquad \text{for $\mathcal{H}^2$--a.e. $z\in\mathbb R^2$}.
  \end{equation}

  We will prove \eqref{eq:wind-integral} using dominated convergence, but we must first bound $\sup_{i\geq 0}|\Wind(\gamma_{i},z)|$. For $i\geq 0$, and $0\leq j\leq 2^i-1$, let $T_{i,j}$ be the (possibly degenerate) closed triangle with vertices $\gamma(j2^{-i})$, $\gamma((j+1)2^{-i})$, and $\gamma((2j+1)2^{-i-1})$. One can check that for every $i\geq 0$, and every $z\in\mathbb R^2\setminus(\gamma_i([0,1])\cup\gamma_{i+1}([0,1]))$
  \[
    |\Wind(\gamma_{i+1},z)-\Wind(\gamma_i,z)|\leq \sum_{j=0}^{2^i-1}\chi_{T_{i,j}}(z).
  \]
  Furthermore, since $\gamma_0$ is constant, we have $\Wind(\gamma_0,\cdot)=0$ for every $z\in \mathbb R^2\setminus \gamma_0([0,1])$, and thus
  \begin{equation}\label{eqn:DomConT}
    \sup_{i\geq 0}|\Wind(\gamma_{i},z)|\leq \sum_{k=0}^{\infty}\sum_{j=0}^{2^k-1}\chi_{T_{k,j}}(z):=\mathcal{T}(z),
  \end{equation}
  for every $z\in\mathbb R^2\setminus\cup_i\gamma_i([0,1])$. By \eqref{eqn:CondSigmaIntro} and the Euclidean isoperimetric inequality, $\mathcal{T}\in L_1(\mathbb R^2)$. 

  Hence, using Theorem~\ref{thm:ExistenceAreaCurves} and dominated convergence, we conclude
  \begin{multline*}
    A(\gamma)  =\lim_{i\to\infty}A(\gamma_i) \stackrel{\eqref{eqn:AreaLipschitz}}{=} \lim_{i\to\infty} \int_{\mathbb R^2}\Wind(\gamma_i,z)\ud z \\
    = \int_{\mathbb R^2}\lim_{i\to\infty}\Wind(\gamma_i,z)\ud z = \int_{\mathbb R^2}\Wind(\gamma,z)\ud z,
    \end{multline*}
  as desired.
\end{proof}

\subsection{Other approaches to signed area}\label{sec:Probability}
While Section~\ref{sec:intro-area-rn} gives one way to define the signed area of a curve in $\R^2$, it is not the only way. The problem of defining integrals of the form $\int F(\gamma) \ud \gamma$ when $\gamma$ is a curve with low regularity often occurs in probabilistic contexts, especially in the case that $\gamma$ is a trajectory of a Brownian motion.

Theorem~\ref{thm:ExistenceAreaCurves} does not cover the case of Brownian motion; indeed, a trajectory $\omega$ is almost surely $\alpha$--H\"older continuous for every $\alpha<\frac{1}{2}$, but almost surely not $\frac{1}{2}$--H\"older continuous. In fact,  $A(\omega)$ is almost surely undefined. If $P_i$ is the $i$th dyadic partition of $[0,1]$, the region between $\omega_{P_i}$ and $\omega_{P_{i+1}}$ consists of $2^i$ triangles. We call the $j$th triangle $\Delta_{i,j}$; it has vertices $\omega(2j2^{-i-1}), \omega((2j+1)2^{-i-1}), \omega((2j+2)2^{-i-1})$ and its expected area is $E[\cH^2(\Delta_{i,j})] \approx 2^{-i}$. The orientation of $\Delta_{i,j}$ is independent of $\cH^2(\Delta_{i,j})$, and if
$$P'_i=P_i\cup \{(2j+1)2^{-i-1} : \Delta_{i,j} \text{ is positively oriented}\},$$
then
$$E[A_{P'_i}(\omega) - A_{P_i}(\omega)] = \frac{1}{2} \sum_j E[\cH^2(\Delta_{i,j})] \approx 1.$$
It follows that $\lim_{\mesh P\to 0} A_P(\omega)$ almost surely diverges. 

The key point here is that $A(\omega)$ is only defined if $A_P(\omega)$ converges in a strong sense. Weaker versions of convergence can still produce a limit. In particular, It\^o showed that for any deterministic sequence of partitions $\{\Pi^n\}$ of $[0,1]$ with $\mathrm{mesh}(\Pi^n)\to 0$, the stochastic Riemann sums
\[
\sum_{t_i\in \Pi^n} \omega^1(t_i) \bigl(\omega^2(t_{i+1})-\omega^2(t_i)\bigr) - \omega^2(t_i) \bigl(\omega^1(t_{i+1})-\omega^1(t_i)\bigr)
\]
converge in probability (or in $L_2$) to the It\^o integral 
\[
\mathcal{B} := \int_0^1 B^1_t \ud B^2_t - B^2_t \ud B^1_t.
\]
The theory of rough paths extends the notion of signed area to an even broader collection of paths, see for instance \cite{FrizHairer}. In this light, Theorem~\ref{thm:ExistenceAreaCurves} can be seen as an attempt to characterize the curves for which the signed area exists in the strongest sense.

\printbibliography[title={References}] 
\end{document}